\documentclass[a4paper,10pt]{article}
\usepackage{amsmath}
\usepackage{amssymb}
\usepackage{hyperref}
\usepackage{a4wide}
\usepackage{amsthm}
\usepackage{amscd}
\usepackage{graphicx}
\usepackage{color}

\theoremstyle{plain}
\begingroup
\newtheorem{theorem}{Theorem}[section]
\newtheorem{theorem*}{Theorem*}[section]
\newtheorem{lemma}[theorem]{Lemma}
\newtheorem{lemmad}[theorem]{Lemma-Definition}
\newtheorem{proposition}[theorem]{Proposition}
\newtheorem{corollary}[theorem]{Corollary}

\newtheorem{definition}[theorem]{Definition}
\newtheorem{rmk}[theorem]{Remark}
\newtheorem{example}[theorem]{Example}

\newtheorem{question}{Question}
\newtheorem{conjecture}{Conjecture}

\endgroup

\makeindex 

\numberwithin{equation}{section}
\newcommand\res{\mathop{\hbox{\vrule height 7pt width .5pt depth 0pt \vrule height .5pt width 6pt depth 0pt}}\nolimits}
\newcommand{\op}[1]{{\text{#1}}}

\newcommand{\be}[1]{\begin{equation}\label{#1}}
\newcommand{\ee}{\end{equation}}
\newcommand{\bee}{\begin{equation*}}
\newcommand{\eee}{\end{equation*}}
\newcommand{\ba}{\begin{eqnarray*}}
\newcommand{\ea}{\end{eqnarray*}}
\newcommand{\er}[1]{\eqref{#1}}
\newcommand{\tb}[1]{\textbf{#1}}

\newcommand{\m}[1]{\mathcal{#1}}
\newcommand{\mb}[1]{\mathbb{#1}}
\newcommand{\e}{\epsilon}
\begin{document}
\setcounter{tocdepth}{1}

\title{The resolution of the Yang-Mills Plateau problem in super-critical dimensions}

\author{Mircea Petrache and Tristan Rivi\`ere\footnote{Forschungsinstitut f\"ur Mathematik, ETH Zentrum,
CH-8093 Z\"urich, Switzerland.}}
\maketitle

\begin{abstract}
We study the minimization problem for the Yang-Mills energy under fixed boundary connection in supercritical dimension $n\geq 5$. We define the natural function space $\mathcal A_{G}$ in which to formulate this problem in analogy to the space of integral currents used for the classical Plateau problem. The space $\mathcal A_{G}$ can be also interpreted as a space of weak connections on a ''real measure theoretic version'' of reflexive sheaves from complex geometry.\\
We prove the existence of weak solutions to the Yang-Mills Plateau problem in the space $\mathcal A_{G}$.\\
We then prove the optimal regularity result for solutions of this Plateau problem. On the way to prove this result we establish a Coulomb gauge extraction theorem for weak curvatures with small Yang-Mills density. This generalizes to the general framework of weak $L^2$ curvatures previous works of Meyer-Rivi\`ere and Tao-Tian in which respectively a strong approximability property and an admissibility property were assumed in addition.\\

\textit{MSC classes}: 58E15, 49Q20, 57R57, 53C07, 81T13, 53C65, 49Q15.
\end{abstract}
\maketitle
\section{Introduction}\label{sec:intro}
\subsection{A nonintegrable Plateau problem}
Consider a smooth compact Riemannian $n$-manifold $M$ with boundary and let $G$ be a compact connected simply connected nonabelian Lie group with Lie algebra $\mathfrak g$. We assume that a principal $G$-bundle $P\to \partial M$ is fixed over the boundary of $M$. On $P$ we consider a $G$-invariant connection $\omega$, which corresponds to an equivariant horizontal $n$-plane distribution $Q$ (see \cite{kono} for notations and definitions).\\

Analogously to the Plateau problem, we may then ask which is the ``most integrable'' extension of $P, Q$ to a horizontal distribution on a principal $G$-bundle over $M$. By Frobenius' theorem, the condition for integrability in this case is that for any two horizontal $G$-invariant vector fields $X,Y$, their lie bracket $[X,Y]$ be again horizontal. The $L^2$-error to integrability of an extension of $Q$ over $M$ can be measured by taking vertical projections $\mathcal V$ of $[X_i,X_j]$ for $X_i,X_j$ varying in an orthonormal basis of $Q$:
\begin{equation}\label{errorhor}
 \int_M\sum_{i,j}|\mathcal V([X_i,X_j])|^2\ .
\end{equation}
Note that $F(X,Y)=\mathcal V([X,Y])$ is known to be a tensor, and $F$ is nothing but the curvature of the connection.

From now on we will work on the vector bundle $E\to M$ associated to the principal bundle $G$ corresponding to a representation of $G$. The covariant derivative $\nabla$ on $E$ is identified, in a trivialization, and via the implicit action of the representation, with the local expression
\[
 \nabla\stackrel{loc}{=}d+A\ ,
\]
where $A$ is a $\mathfrak g$-valued $1$-form on a given chart of $M$. The structure equation relating curvature to connection takes the form
\begin{equation}\label{streq}
 F\stackrel{loc}{=}dA+A\wedge A
\end{equation}
in a trivialization. Here $\wedge$ represents a tensorization of the usual exterior product of forms with the Lie bracket on $\mathfrak g$.
In this setting the $L^2$-error in integrability \eqref{errorhor} is identified with the \textit{Yang-Mills energy}, which we consider as being a functional of the connection $\nabla$:
\begin{equation}\label{YM}
 \mathcal{YM}(\nabla):=\int_M|F_\nabla|^2\ .
\end{equation}
We observe that, similarly to the area functional in the Plateau problem, $\mathcal{YM}$ has a \textit{large invariance group} given by changing coordinates in the fibers via $G$. The gauge group can, in the case of classical bundles, be identified with regular sections of the automorphism bundle $\op{Aut}(E)$, whose fibre is identified with $G$. If $E$ is trivial then $\mathcal G$ identifies to the following space:
\begin{equation}\label{G}
 \mathcal G =\{g:M\to G\}.
\end{equation}
If $E$ is nontrivial, we will sometimes use \textit{singular trivializations}, which always exist in large enough Sobolev spaces (see \cite{PR2} and Prop. \ref{realization} below). Therefore we may then still implicitly identify $\mathcal G$ with maps from $M$ to $G$, which in this case are \emph{singular} ones\footnote{though this might look blasphemous from a strict smooth geometric perspective} (see Example \ref{singular4d}). Then an element $g\in \mathcal G$ acts on the curvature form in local coordinates $F=\sum F_{ij}dx_i\wedge dx_j$ via
\[
 F_{ij}\mapsto g^{-1}F_{ij}g,\quad |F|\mapsto |g^{-1}Fg|=|F|\ ,
\]
where we used the fact that the canonical norm on the Lie algebra $\mathfrak g$ is given by the Killing form (see again \cite{kono}), which is invariant under the adjoint action of $G$. On the other hand the connection form $A$ is transformed by a gauge change $g$ into
\[
A^g=g^{-1}dg + g^{-1}Ag\ , 
\]
which has two relevant effects:
\begin{itemize}
\item The \emph{regularity} of $A$ is dependent on the gauge, since the term $g^{-1}dg$ could, depending on the situation, either compensate or create small oscillations of $g^{-1}Ag$.
\item For what concerns the variational study of $\mathcal{YM}$, the gauge action can be responsible of a loss of \emph{compactness} for minimizing (or Palais-Smale) sequences of connections.
\end{itemize}
Similar difficulties were met in the study of the parametric version of the Plateau problem, where the large invariance group was the group of reparametrizations. Similarly to the appearance of singular branched surfaces in the Plateau problem, our main equations survive, while the classical setting like in \eqref{G} will naturally break up. As a consequence, several differential geometry notions, most notably the presence, in our theory, of underlying classical bundles like $\op{ad}(E)$ and $\op{Aut}(E)$, must be weakened or discarded. 

\subsection{Natural spaces and Coulomb gauge-fixing in low dimension}
\subsubsection{Yang-Mills equations and ellipticity}
The equations satisfied by critical points of $\mathcal{YM}$ are the Yang-Mills equations, which can be written as $d^*_AF_A=0$ where $d_A$ refers to the extension of the covariant derivative corresponding to $A$ to higher degree differential forms and $d_A^*$ is its Hodge adjoint with respect to the metric on $M$. In local coordinates and if we have a flat metric these equations read
\begin{equation}
\label{ymcoord}
 \sum_{i=1}^n\partial_{x_i}F_{ij} +[A_i,F_{ij}]\ ,\quad j=1,\ldots,n\ ,
\end{equation}
where $[\cdot,\cdot]$ is the Lie bracket on $\mathfrak g$. The other underlying equation is the Bianchi identity $d_AF_A=0$ which in local coordinates reads
\[
 \sum_{\text{perm } i,j,k}\partial_{x_k}F_{ij}+[A_k,F_{ij}]=0\ .
\]
In the simpler abelian case $G=U(1), \mathfrak g=\mathfrak u(1)\simeq \mathbb R$ (that we don't consider here) we have $F_A=dA$ and the Yang-Mills equation reads $d^*dA=0$. It is then natural to complete it to an elliptic system for $A$ by choosing the gauge $g$ such that $d^*A^g=0$. This is the so-called \emph{Coulomb gauge} condition. For general groups $G$ the Coulomb gauge extraction for a connection $A$ consists in solving for $g$ the nonlinear equation
\begin{equation}
 \label{couleq}
 d^*\left(g^{-1}dg\right)+ d^*\left(g^{-1}Ag\right)=0,\quad\text{ i.e. }d^*(A^g)=0\ .
\end{equation}
In this gauge the Yang-Mills equations \eqref{ymcoord} become
\begin{equation}
 \label{ymccoord}
 \Delta A^g_j = \sum_{i=1}^n[A_i^g, \partial_{x_j}A_i^g] -[A_i^g, [A_i^g,A_j^g]]\ ,\quad j=1,\ldots,n\ ,
\end{equation}
which is an elliptic quasilinear equation in $A^g$. Note that it makes sense to study equations of the form \eqref{couleq}, \eqref{ymccoord}, without necessarily interpreting $A,g$ as sections of topological bundles over $M$. We can then consider $A$ to be a $\mathfrak g$-valued $1$-form on $M$ and $g$ to be a map from $M$ to $G$.
\subsubsection{Critical and supercritical dimensions}
The natural function spaces in which to consider the minimization of $\mathcal{YM}$ are identified by considering the local form of the structure equation \eqref{streq}. The curvature form $F$ is naturally required to be $L^2$ in order for the energy to be finite. In the abelian situation $G=U(1)$ there holds $A\wedge A=0$ and $\int|F_\nabla|^2=\int|dA|^2$ hence $W^{1,2}$ is a natural space to consider for the connection forms $A$. In a non-abelian framework the situation is more delicate due to the nonlinearity $A\wedge A$. Assuming $A\in W^{1,2}$ the linear term $dA$ of \eqref{streq} belongs to $L^2$, but the $L^2$ control of the quadratic nonlinearity $A\wedge A$ requires a priori $A\in L^4$.\\

In dimensions $n\leq 4$ the norm inequality underlying the Sobolev embedding $W^{1,2}\to L^4$ implies that we have both $dA$ and $A\wedge A$ in $L^2$. This embedding is not valid anymore in dimensions $n\geq 5$, which are called \textit{supercritical dimensions}.\\

As a parallel phenomenon, the nonlinear equation \eqref{couleq} for $g$ with natural $W^{2,2}$-regularity assumptions on $g$ becomes critical in dimension $4$. In the simplest case of a trivial bundle, the space $\mathcal G$ is identified with that of smooth maps $g:M\to G$, which is dense in $W^{2,2}(M,G)$ only for $\op{dim}M\le 4$. The case of dimension $4$ is again critical, as seen in the failure of Kondrachov-Rellich compactness in this case. 
\subsubsection{Extracting controlled Coulomb gauges}
A naive way of solving \eqref{couleq} is by direct minimization, i.e. we may consider
\begin{equation}
 \label{mincoul}
 \min_{g\in W^{1,2}(M,G)}\int_M\left|g^{-1}dg+g^{-1}Ag\right|^2\ .
\end{equation}
While by convexity arguments a minimizer of \eqref{mincoul} exists, we face the following problems:
\begin{itemize}
 \item A priori the minimizer $g$ gives a new connection $A^g$ which is just in $L^2$ and thus we \emph{loose regularity} with respect to the natural assumption $A\in W^{1,2}$.
 \item The so-obtained $A^g$ is controlled just by $A$ and not by the Yang-Mills functional. We thus \emph{loose coercivity and compactness}.
\end{itemize}
In dimensions $n\le 4$, the result resolving the above two problems contemporarily was achieved in \cite{Uhl2} by K. K. Uhlenbeck:
\begin{theorem}[\cite{Uhl2}]\label{uhlereg} 
 Let $n\le 4$, $G$ be a compact Lie group. Then there exists $\epsilon_G$ such that for all $A\in W^{1,2}(B^n, \Lambda^1\mathbb R^n\otimes\mathfrak g)$ such that
 \[
 \int_{B^n}|F_A|^2dx^n<\epsilon_G
\]
there exists $g\in W^{2,2}(B^n,G)$ such that
\[
 \|A\|_{W^{1,2}}\le C\|F_A\|_{L^2}\quad\text{ and }\quad d^*A^g=0\ .
\]
\end{theorem}
Applying the above theorem to the minimization of $\mathcal{YM}$ corresponds to optimizing (over both $A$'s and $g$'s this time) the more coercive functional
\[
\int \left(|F_\nabla|^2+|d^*A^g|^2\right)\ge \int|F_{\nabla}|^2\ .
\]
\subsubsection{Yang-Mills -Plateau problem in dimension $n=4$}
In this subsection we recall some well-known results from dimension $4$, in order to relate the new point of view taken here with the classical previous framework. Just for this subsection and in order to keep a more direct link to the references that we cite, we keep the principal bundle $P\to M$ at the center of our notations. This will have no substantial effect, as ultimately the connections themselves are always on an associated vector bundle $E\to M$ and there is an implicit identification between $\op{ad}(P)$ and $\op{ad}(E)$, given by the representation chosen in the beginning.\\

For a smooth principal bundle $P$ and $\nabla_0$ a smooth connection on $E$, we consider first the space of connections
\begin{equation}\label{A12}
 \mathcal A^{\ell,2}_G(P) :=\{\nabla_0+A:\ A\in\Omega^1(\op{ad} (P))_\ell\}\ ,
\end{equation}
where $\Omega^1(\op{ad} (P))_\ell$ is the space of $W^{\ell,2}$-sections of the $\op{ad}(P)$-valued $1$-forms on $M$. More precisely, for a good cover $\{U_\alpha\}$ of $M$ we have local coefficients $A_\alpha\in W^{\ell,2}(U_\alpha, \Lambda^1U_\alpha\otimes \mathfrak g)$ for which $P$'s (smooth) changes of trivialization $g_{\alpha\beta}$ over $U_\alpha\cap U_\beta$ satisfy $A_\beta = g_{\alpha\beta}^{-1}dg_{\alpha\beta} + g_{\alpha\beta}^{-1}A_\alpha g_{\alpha\beta}$ on $U_\alpha\cap U_\beta$.\\

By Theorem \ref{uhlereg} together with a point removability result one deduces, within a by now classical theory started in the 80's and extended in \cite{isobe1}, the following:
\begin{theorem}[\cite{Uhl2},\cite{sedlacek}]\label{ymp4}
 Let $M$ be a compact Riemannian $4$-manifold and $P\to M$ a classical principal $G$-bundle. Consider a sequence of connections $\nabla_k\in \mathcal A^{1,2}_G(P)$ such that their curvature forms $F_k$ are equibounded in $L^2$ and such that we have the weak convergence\footnote{This can be formalized as follows: for every simply connected domain $U\subset M$ let $i_U:U\to M$ be the inclusion map and let $\sigma$ be a trivialization of $i_U^* P$. Then we require that the curvature forms $F_k^{\sigma_P}$ representing the $F_k$ in this trivialization converge weakly in $L^2$ to some $F^{\sigma_P}$. Since $U$ is contractible also $i_U^*\tilde P$ has a trivialization $\sigma_{P'}$ over $U$ and we may identify $\tilde F^{\sigma_P}$ to a form $F=\tilde F^{\sigma_P^{-1}\sigma_{P'}}$ on $i_U^*P'$.}
\[
 F_k\rightharpoonup F\quad\text{ in }L^2\ .
\]
Then $F$ is the curvature form of a connection $\nabla\in \mathcal A^{1,2}_G(\tilde P)$ where $\tilde P\to M$ is a classical principal $G$-bundle (possibly different than $P$).
\end{theorem}
In the above-described classical result from the 80's the changes of trivialization $g_{\alpha\beta}$ were assumed to be continuous. In \cite{Uhl2} it is proved that if two distinct bundles have change of trivialization maps $g_{\alpha\beta}, h_{\alpha\beta}$ controlled in $W^{2,p}_{loc}, p>2$ and closer than some small constant (depending only on $M$) in $L^\infty$ then they define the same bundle. See \cite{rivnote} V.1.\\

Isobe \cite{isobe1}, \cite{isobe2} suitably extended and precised these ideas proving convergence results for the classes $\mathcal P^{k,p}(M)$ of \textit{Sobolev principal bundles} defined as follows:
\begin{equation}\label{pkp}
\mathcal P_G^{k,p}(M):=\left\{ 
\begin{array}{l}
P=\langle\{U_\alpha\}_{\alpha\in I}, \{g_{\alpha\beta}\}_{\alpha,\beta\in I}\rangle: \ \{U_\alpha\}_{\alpha\in I}\text{ is a good cover of }M\\[5mm]
g_{\alpha\beta}\in W^{k,p}(U_{\alpha\beta},G) \text{ satisfy }g_{\alpha\beta}(x)\cdot g_{\beta\gamma}(x) =g_{\alpha\gamma}(x)\text{ on every }U_\alpha\cap U_\beta\cap U_\gamma\neq\emptyset
\end{array} 
\right\}.
\end{equation}
The fact that $C^\infty(\mathbb B^4,G)$ is dense in $W^{2,2}(\mathbb B^4,G)$ can be used within Isobe's theory and leads to his result that $\mathcal P^{2,2}$-bundles are approximated by smooth bundles in dimension $4$. In our particular case one then considers the following space of connections (for more general spaces see \cite{isobe1},\cite{isobe2}), this time for  $P=\langle\{U_\alpha\}_{\alpha\in I}, \{g_{\alpha\beta}\}_{\alpha,\beta\in I}\rangle\in\mathcal P^{2,2}(M^4)$, i.e. with the assumption that changes of trivialization $g_{\alpha\beta}\in W^{2,2}$:
\begin{equation}\label{a12}
\mathcal A^{1,2}_G(P):=\left\{ 
\begin{array}{l}
A=\{A_\alpha\}_{\alpha\in I}: \ A_\alpha\in W^{1,2}\cap L^4(U_\alpha, \Lambda^1U_\alpha\otimes\mathfrak g)\text{ for all }\alpha\in I,\\[5mm]
A_\beta = g_{\alpha\beta}^{-1}dg_{\alpha\beta} + g_{\alpha\beta}^{-1}A_\alpha g_{\alpha\beta}\text{ on every }U_\alpha\cap U_\beta\neq\emptyset
\end{array} 
\right\}.
\end{equation}
Finally, it turns out to be clarifying to encode also the regularity of the gauge change transformations. This will allow to compare it to the regularity of the bundle itself. The notion of Sobolev isomorphism will have different meanings, depending on this comparison. The classical case is the following: let $P,P'\in\mathcal P^{k,p}(M)$. We consider the case where the good cover $\{U_\alpha\}_{\alpha\in I}$ is the same for $P,P'$, the other case is treated in \cite{isobe1} Def. 3.2. Assume $k\ge 1, p\ge 1$ and that $P, P'$ have respectively change of trivialization cocycles $\{g_{\alpha\beta}\}_{\alpha,\beta\in I}$ and $\{h_{\alpha\beta}\}_{\alpha,\beta\in I}$. Then $g=\langle \{g_\alpha\}_{\alpha\in I}\rangle$ with $g_\alpha\in W^{l,q}(U_\alpha, G)$ for $\alpha\in I$ is a \textit{$W^{l,q}$-bundle isomorphism} between $P$ and $P'$ and we write $g(P)=P'$ if
\begin{equation}\label{bundletransf}
h_{\alpha\beta}(x):=g_\alpha(x) g_{\alpha\beta}(x)g_\beta^{-1}(x)\quad\text{ whenever }x\in U_\alpha\cap U_\beta.
\end{equation}
Following \cite{isobe1} Def. 3.2 we define more generally for $P\in \mathcal P^{k,p}(M)$ and for $p,q,k,l\ge 1$, the class
\[
[P]_{l,q}:=\{P':\ P'=g(P)\text{ for some }g=\langle \{g_\alpha\}_{\alpha\in I}\rangle,\ g_\alpha\in W^{l,q}(U_\alpha,G)\text{ for }\alpha\in I\}
\]
This definition appears in the literature in case $k=l, p=q$. In this special case one defines 
\[
\hat{\mathcal P}^{k,p}(M):=\{[P]_{k,p}:\ P\in\mathcal P^{k,p}(M)\}
\]
For $P'=P\in \mathcal P^{k,p}(M)$ we may consider gauge change groups $\mathcal G^{k,p}(P)$ consisting of $g$ as above such that $g(P)=P$. These are studied in \cite{isobe1}, \cite{isobe2}. If instead $W^{k,p}\subset W^{l,q}$ with a strict inclusion and the $g_\alpha$ have only $W^{l,q}$-regularity, then formula \ref{bundletransf} still makes sense but the new $P':=g(P)$ will have lower regularity. As a more drastic consequence, we will see in Example \ref{singular5d} that for $g$ of low regularity, \textit{the topology of the Sobolev bundles is completely disrupted by the bundle isomorphisms} and we have only one single class $[P]_{l,q}$ in $\mathcal P^{k,p}(M)$. In other words, \textit{all bundles are trivialized}.\\

Realizing that in $4$ dimensions due to Theorem \ref{ymp4} the bundle structure starts losing robustness, we give up the idea of following the bundles supporting the conections, and we encode all the information on the bundle into possibly very singular $1$-forms. Let $(M^m,h)$ be a compact Riemanian $m$-dimensional manifold. We introduce the space defined by
\begin{equation}\label{ag4d}
{\mathfrak A}_G(M^m):=\left\{ 
\begin{array}{l}
A\in L^2(\Lambda^1M, \mathfrak g)\ ; \ \int_{M^m}|dA+A\wedge A|_h^2\ dvol_h<+\infty\\[5mm]
\mbox{ locally }\exists\ g\in W^{1,2}\quad\mbox{ s.t. }\quad g^{-1}dg+g^{-1}Ag\in W^{1,2}
\end{array} 
\right\}.
\end{equation}
Note that a shift in interpretation has occured here. Via the following prototypical example we describe how to reinterpret classical connections on notrivial bundles as elements of $\mathfrak A_G(M^m)$ in a special case. The basic principle is that by considering $\mathfrak A_G(M^m)$ \textit{we are trivializing with $W^{1,2}$-sections bundles which are not possible to trivialize by $W^{2,2}$-sections}.
\begin{example}\label{singular4d}
We consider the case $G=Sp(1)\simeq SU(2)$ and in quaternion notation we define $1$-form $\bar A:\mathbb S^4\to \Lambda^1T^*\mathbb S^4\otimes\mathfrak{sp}(1)$ as 
\[\bar A(x):=\Pi^*\left(\op{Im}\frac{\bar xdx}{1+|x|^2}\right),\]
where $\Pi:\mathbb S^4\to\mathbb R^4$ is the stereographic projection which sends the point at infinity of $\mathbb R^4$ to the south pole of $\mathbb S^4$. This gives the well-known instanton $\bar A$ over $\mathbb S^4$ with Chern number $c_2=1$. Moreover $\bar A$ corresponds to a smooth connection on a nontrivial smooth bundle $E\to\mathbb S^4$. Note that the coefficients of $\bar A$ are singular but controlled: they blow up like the inverse of the distance to the south pole. Thus $|A|$ is still in $L^p(\mathbb S^4)$ for all $p<4$. Moreover note that, still in quaternionic notation,
\[
\op{Im}\left(\frac{\bar xdx}{1+|x|^2}\right)\sim \op{Im}\left(\left(\frac{x}{|x|}\right)^{-1}d\frac{x}{|x|}\right) \quad\text{ for }\quad |x|\to \infty, 
\]
thus pulling back via $\Pi$, in a chart around the singularity we may write $\bar A=g^{-1}dg$ for $g\in W^{1,p}(\mathbb S^4,Sp(1))$ for such $p$ (in fact $A$ is in the space $L^{4,\infty}$ and $g\in W^{1,(4,\infty)}$, which fits precisely in the setting considered in \cite{PR2}, and it is not a coincidence that in $4$ dimensions these spaces are at a borderline regularity with respect to $A\in W^{1,2}, g\in W^{2,2}$). Since away from the singularity $\bar A$ is already smooth we found $g\in W^{1,2}$ that locally transforms $\bar A$ to a form with $W^{1,2}$-coefficients. The curvature form $F_{\bar A}$ is of constant norm, and in particular in $L^2$. Therefore $\bar A\in\mathfrak A_{Sp(1)}(\mathbb S^4)$.
\end{example}
In order to balance the above phenomenon, we next present the general link between Isobe-Uhlenbeck bundles \eqref{pkp} and the bundle-free version \eqref{ag4d}, valid in $4$ dimensions.The principle of the next proposition is summarized in two points:
\begin{itemize}  
\item As exemplified for the instanton bundle $P$ in Example \ref{singular5d}, in $4$ dimensions for nontrivial $P\in\mathcal P^{2,2}(M^4)$ there holds
\[
\left\{\begin{array}{l}\mathcal G^{2,2}(P)\ncong W^{2,2}(M^4, G)\text{ for nontrivial }P,\\[3mm]
\mathcal G^{1,2}(P)\cong W^{1,2}(M^4, G)\text{ for all }P.\
\end{array}\right.
\]
\item However if at the same time we restrict the setting and as in \eqref{ag4d} we consider only forms that are locally gauge-equivalent to $W^{1,2}$-forms, then \textit{from the regularity of the coefficients we may recover the regularity of the bundles}, and come back precisely to $\mathcal A^{1,2}$-connections on $\mathcal P^{2,2}$-bundles.
\end{itemize}

\begin{proposition}[realization of $\mathfrak A_G$-forms as connections on $\mathcal P^{2,2}$-bundles in $4$ dimensions]\label{realization}
Let $M$ be a compact Riemannian $4$-manifold. There exists a surjective \emph{geometric realization map} 
\[
\mathfrak R:\mathfrak A_G(M)\to\left\{([P]_{2,2},\tilde A):\ [P]_{2,2}\in\hat{\mathcal P}^{2,2}_G(M), \tilde A\in \mathcal A^{1,2}_G(P)\right\}, 
\]
such that if $\mathfrak R(A)=([P]_{2,2},\tilde A)$ then for representatives $P=\langle\{U_\alpha\},\{g_{\alpha\beta}\}_{\alpha,\beta\in I}\rangle\in[P]_{2,2}$ and $\tilde A=\{A_\alpha\}_{\alpha\in I}$, we have that on each $U_\alpha$ there exists a change of gauge map $g_\alpha\in W^{1,2}(U_\alpha,G)$ such that $A_\alpha= A^{g_\alpha}$ on $U_\alpha$. 

\end{proposition}
\begin{proof} See Appendix \ref{proofp14}
\end{proof}

In order to make the sequel clearer we reformulate the classical result of Theorem \ref{ymp4} in the new framework \eqref{ag4d} in the following new form.

\begin{theorem}[\cite{rivnote} Thm.VIII.1]\label{viii1}
Let $m\le 4$. For any $A_k\in  {\mathfrak A}_G(M^m)$ satisfying
\[
\limsup_{k\rightarrow+\infty}YM(A_k)<+\infty
\]
there exists a subsequence $A_{k'}$ and a Sobolev connection $A_\infty\in  {\mathfrak A}_G(M^m)$ such that $F_{A_{k'}}\rightharpoonup F_{A_\infty}$ weakly in $L^2$ and
\begin{equation}\label{distconverg}
d(A_{k'},A_\infty):=\inf_{g\in W^{1,2}(M^m,G)}\int_{M^m}|A_{k'}-(A_\infty)^g|_h^2\ dvol_h\longrightarrow 0.
\end{equation}
\end{theorem}

The classical result of Theorem \ref{viii1} \textit{fails to extend} to manifolds $M^m$ of dimension $m>4$:
\begin{proposition}[\cite{rivnote} Prop.VIII.1]\label{prviii1}
For $m> 4$ there exists $A_k\in  {\mathfrak a}_{SU(2)}(M^m)$ satisfying
\[
\limsup_{k\rightarrow+\infty}YM(A_k)<+\infty
\]
 and a Sobolev connection $A_\infty\in  L^2$ such that \eqref{distconv} occurs but in every neighborhood $U$ of every point of $M^m$ there is no $g$ such that $(A_\infty)^g\in W^{1,2}(U)$.
\end{proposition}
We see as a consequence of this proposition that whereas the relationship between the bundle-free space $\mathfrak A_G$ and the classical $\mathcal A^{1,2}$-connections on $\mathcal P^{2,2}$-bundles works in $4$ dimensions, the classical setting fails in dimension higher than $4$. We will see next that however $\mathfrak A_G$-spaces naturally extend to higher dimensions.

\subsection{Supercritical dimension $n=5$ and main goals of the paper}
Proposition \ref{prviii1} leaves the following question open:
\begin{question}\label{mainquestion}
 Which is the correct replacement for the space $\mathfrak A_G(M^4)$ defined in \eqref{ag4d} which allows to extend Theorem \ref{viii1} to dimensions $n\geq 5$?
\end{question}
The main focus of this paper is to answer this question. For the clarity of the presentation we restrict in this work to the case of dimension $5$ and to an euclidean setting. The extension of all our results to higher dimensions $n>5$ as well as to general Riemannian manifolds will be done in a forthcoming work \cite{PR4}.\\

Based on the above discussion, natural requirements for a good substitute $\mathcal A_{G}(\mathbb B^5)$ are the following:
\begin{itemize}
 \item \emph{(good closure properties)} If a sequence $F_k$ corresponding to connections in $\mathcal A_{G}(\mathbb B^5)$ converges weakly in $L^2$ and has bounded $\mathcal{YM}$-energy, then the limit is again in $\mathcal A_{G}(\mathbb B^5)$.
 \item \emph{(Coulomb gauge extraction)} For elements of $\mathcal A_{G}(\mathbb B^5)$ an analogue of the procedure of Theorem \ref{uhlereg} for producing elliptic control via Coulomb gauges should be available.
\end{itemize}
We define the following extension of the class $\mathfrak A_G$ to supercritical dimension:
 \begin{definition}[Weak connections in dimension $5$]\label{defFZ}
We  define the following to be the class of $L^2$ weak connections on singular bundles over $\mathbb B^5$:
\begin{equation*}
 \mathfrak a_{G}(\mathbb B^5):=\left\{
 \begin{array}{c}
  A\in L^2(\mathbb B^5, \Lambda^1\mathbb R^5\otimes\mathfrak g):\ F_A\stackrel{\mathcal D'}{=}dA+A\wedge A\in L^2\\[3mm]
  \forall p\in M\text{ a.e. }r>0,\:i^*_{\partial B_r(p)}A\in\mathfrak A_G(\partial B_r(p))
 \end{array}
 \right\}\ .
\end{equation*}
Let $[A]$ denote the equivalence class of all $A'\in \mathfrak a_G(\mathbb B^5)$ such that $A'=g^{-1}dg+g^{-1}Ag$ for $g\in W^{1,2}(M,G)$. Define also 
\[
\mathcal A_G(\mathbb B^5):=\mathfrak a_G(\mathbb B^5)/\sim.
\]

\end{definition}
The reason for introducing our space $\mathcal A_G(\mathbb B^5)$ is that it is necessary for our distance $\op{dist}$ below, which in turn appears in the geometric definition of traces on $\partial \mathbb B^5$ (see Section \ref{traces}).
\begin{example}
Consider a trivial classical bundle $E\to\mathbb B^5$ and a smooth connection $\nabla$ on it. Then in a trivialization the connection coefficients give a connection form $A\in C^\infty(\mathbb B^5, \Lambda^1\mathbb R^5\otimes \mathfrak g)$, in particular it gives an element of $\mathfrak a_G(\mathbb B^5)$.  
\end{example}
\begin{example}
We may also consider the case of connection coefficients $A\in W^{1,2}\cap L^4(\mathbb B^5,\Lambda^1\mathbb R^5\otimes \mathfrak g)$. Then the conditions on $A, F$ in the above definition are satisfied directly and we have $A\in \mathfrak a_G(\mathbb B^5)$.
\end{example}
\begin{example}\label{singular5d}
We consider the situation and notation of Example \ref{singular4d}. Define the $1$-form $A:\mathbb B^5\to \Lambda^1\mathbb R^5\otimes\mathfrak{sp}(1)$ as 
\[A(x):=\left(\frac{x}{|x|}\right)^*\left(\Pi^*\left(\op{Im}\frac{\bar xdx}{1+|x|^2}\right)\right).\]
The restriction of $A$ to $\mathbb S^4$ gives the $\bar A$ of Example \ref{singular4d}, thus is in $\mathfrak A_G(\mathbb S^4)$. If the south pole of $\mathbb S^4$ is $S$ then the coefficients of $A$, $F$ are controlled like
\[
|A(x)|=O(1)\frac{1+o(1)}{|x|\ d_{\mathbb S^4}(x/|x|,S)}\text{ as }x\to 0\quad \text{ and }\quad|F(x)|=\frac{C}{|x|^2}
\]
and thus $A\in L^p$ for all $p<4$, and $F_{\bar A}\in L^p$ for all $p<5/2$. On slices not passing through the origin we may consider the pullback of the smooth nontrivial bundle $E$ of Example \ref{singular4d} and the pullback of the singular gauges $g$ described there, that $i^*_{\partial B_r(p)}A\in\mathfrak A_G(\partial B^5)$. In particular $A\in \mathfrak a_{Sp(1)}(\mathbb B^5)$. This example gives in fact a minimizer of the Yang-Mills Plateau problem at fixed boundary datum, as shown in \cite{Pradmin}. Further, as shown in Section \ref{everywherediscont}, the above connection form can only live on a bundle having a topological singularity at the origin, as $*d\op{tr}(F_A\wedge F_A)=8\pi^2\delta_0$.
\end{example}
More generally, extending the above examples we will show the following series of inclusions, where the class $\mathfrak r_G^\infty$ will be defined in \eqref{Rinfphi} below. It contains connection coefficients corresponding to classical connections over smooth bundles with defects.
\begin{equation}\label{inclusions}
L^2(\mathbb B^5, \Lambda^1\mathbb R^5\otimes \mathfrak g)\supset\mathfrak a_G(\mathbb B^5)\supset \mathfrak r^\infty_G(\mathbb B^5)\supset W^{1,2}\cap L^4(\mathbb B^5, \Lambda^1\mathbb R^5\otimes \mathfrak g)\supset C^{\infty}(\mathbb B^5, \Lambda^1\mathbb R^5\otimes \mathfrak g).
\end{equation}

For the space $\mathcal A_{G}(\mathbb B^5)$ defined above we obtain the following extra properties, which in particular include the requirements stated before the definition:
\begin{itemize}
 \item The \emph{weak closure} result is given in Theorem \ref{wclos}.
 \item The space $\mathcal A_{G}(\mathbb B^5)$ is proven (see Theorem \ref{naturality}) to be obtainable by strong closure from the space $\mathcal R^\infty(\mathbb B^5)$ defined as in \eqref{Rinfphi}. This space consists of connection forms corresponding to smooth connections on \emph{smooth bundles with defects}, i.e. defined outside a finite set of points of $\mathbb B^5$.
 \item The correct generalization for the \emph{Coulomb gauge extraction} of Theorem \ref{uhlereg} involves scale-invariant Morrey norms as described in \cite{MeRi}, \cite{TaoTian}. We have the same result for $\mathcal A_{G}(\mathbb B^5)$, see Theorem \ref{merinew}. This gives the $\epsilon$-regularity Theorem \ref{eregm} and the optimal partial regularity of Corollary \ref{preg} for weak stationary Yang-Mills connections. The main step is a refinement of the above strong-$L^2$ approximation result, i.e. Theorem \ref{mapproxd}.
 \item The main new feature of the space $\mathcal A_{G}(\mathbb B^5)$ is that we have \emph{both weak closure and partial regularity} results in a unified setting. This is analogous to the achievement obtained by Federer and Fleming \cite{FF} for the \emph{classical Plateau problem} via the introduction of integral currents. The solution of the analogous Yang-Mills-Plateau problem requires the definition of a notion of boundary trace as in Theorem \ref{trace}. Theorem \ref{rymp} states that minimizers of the Yang-Mills-Plateau problem exist and have isolated singularities.
\end{itemize}
\subsection{The weak closure result}
\begin{theorem}[\textbf{sequential weak closure of $\mathcal A_{G}$}]\label{wclos}
 Let $[A_k]\in \mathcal A_{G}(\mathbb B^5)$ be a sequence of connections such that the corresponding curvature forms $F_k$ are equibounded in $L^2(\mathbb B^5)$ and converge weakly to a $2$-form $F$. Then $F$ corresponds to $[A]\in\mathcal A_{G}(\mathbb B^5)$.
\end{theorem}
Definition \ref{defFZ} and Theorem \ref{wclos} are inspired by the slicing approach to the closure theorem for rectifiable currents, initially introduced by B. White \cite{white}, R. L. Jerrard \cite{jerrard} and used by L. Ambrosio and B. Kirchheim \cite{AK} for their striking proof of the closure theorem for rectifiable currents in metric spaces. The idea behind this approach is that a current is rectifiable when its slices via level sets of Lipschitz functions give a metric bounded variation ($MBV$, for short) function with respect to the flat metric between the sliced currents.\\
The closure theorem for rectifiable currents corresponds then to a compactness result for $MBV$ functions, valid when the oscillations of slices are controlled via the overlying total mass functional for sequences of weakly convergent currents. This mass-finiteness condition was weakened by R. M. Hardt and T. Rivi\`ere \cite{HR1}, who introduced the notion of \textit{rectifiable scans}.\\

In \cite{PR1} the authors used the ideas coming from the theory of scans for defining the class of weak $L^p$ curvatures over $U(1)$-bundles and proving the weak closure theorem relevant for minimizing the $p$-Yang-Mills energy $\int_M|F|^p$ in supercritical dimension $3$ for $1<p<3/2$ (see also \cite{KR}). This class of weak curvatures is identified via Poincar\'e duality with the class of $L^p$ vector fields on $3$-dimensional manifolds having integer fluxes through ``almost all spheres''.\\

The new difficulty with respect to such result is mentioned in Section \ref{sec:introode} and amounts to the justification of the existence of gauges $g:\mathbb B^5\to G$ which are $W^{1,2}$-controlled and solve an ODE of the form $\partial_t g =-Ag$ where $A$ is a connection form corresponding to $[A]\in \mathcal A_G(\mathbb B^5)$. Such existence result is based on the strong approximation result of Theorem \ref{naturality}.

\subsection{Naturality of the space $\mathcal A_{G}$}
The proof of the partial regularity of solutions to \eqref{YMP} goes through a more thorough description of our space $\mathcal A_{G}(\mathbb B^5)$ as being the $L^2$-closure of the space of connections which are smooth away from a set of isolated points. We introduce the class of smooth connection forms over bundles with defects:
\begin{equation}\label{Rinfphi}
 \mathfrak r^\infty_G(\mathbb B^5):=\left\{
 \begin{array}{c}
  A\in \mathfrak a_{G}(\mathbb B^5)\text{ s.t. } \exists k,\exists a_1,\ldots,a_k\in \mathbb B^5,\\[3mm]
 \forall p\in B\ \forall r>0\text{ s.t. }\partial B_r(x)\subset\mathbb B^5\setminus\{a_1,\ldots,a_k\}\\[3mm] \mathfrak R(i^*_{\partial B_r(x)}A)\text{ is represented by a smooth connection }\nabla\\[3mm]
  \text{on some smooth }G\text{-bundle }P\to \mathbb B^5\setminus\{a_1,\ldots,a_k\}
 \end{array}
\right\}\ .
\end{equation}
We further define $\mathcal R^\infty(\mathbb B^5):=\mathfrak r^\infty_G(\mathbb B^5)/\sim$ for the same equivalence relation $\sim$ as in Definition \ref{defFZ}\\

The strong approximation will occur with respect to the following geometric distance:
\begin{equation}\label{dcurv}
 \op{dist}_F(F, F'):=\min\{\|F-g^{-1}Fg\|_{L^2(\mathbb B^5)}:\:g:\mathbb B^5\to G\text{ measurable}\}\ .
\end{equation}
We have the following:
\begin{theorem}[\textbf{Naturality of $\mathcal A_{G}$}]\label{naturality}
 Let $[A]\in\mathcal A_{G}(\mathbb B^5)$ and let $F\in L^2$ be the connection form of an $L^2$ representative $A$ of $[A]$. Then there exist curvature forms $F_k$ corresponding to connection forms $A_k$, $[A_k]\in\mathcal R^\infty(\mathbb B^5)$ such that
\[
 A_k\to A\ \text{ in }L^2,\quad F_k\to F\ \text{ in }L^2\ .
\]
In particular there holds
 \[
  \op{dist}_F(F_k,F)\to 0,\qquad\text{ as }k\to\infty\ .
 \]
\end{theorem}
The strategy of proof of Theorem \ref{naturality} is based on the strong approximation procedure that F. Bethuel introduced for his approximation results \cite{bethuel} for Sobolev maps into manifolds. However recall the fact that as discussed above, unlike the case of Sobolev maps (where $\|du\|_{L^p}$ controls $\|u\|_{L^{p^*}}$), here $\|F\|_{L^2}$ does not control the connection form. Hence the strategy for filling the ``good cubes'' differs completely from the one available in the case of Sobolev maps and requires a completely new argument. \\
Pushing the comparison with the case of Sobolev maps into manifolds further, the corresponding weak closure result for Sobolev maps in $W^{1,p}(\mathbb B^m, N^n)$ for instance is a direct consequence of Rellich-Kondrachov's theorem, whereas in our case the analogous result, Theorem \ref{wclos} for weak connections, required a substantial amount of work.

\subsection{Some consequences on weak solutions to ODE}\label{sec:introode}
The application of the strong density theorem \ref{naturality} in the proof of the weak closure result theorem \ref{wclos} goes through the result of the next proposition, which is of independent interest: the ODE \eqref{ODE} can be solved in $W^{1,2}(\mathbb B^5, G)$ provided the field $A$ is given by a connection form $A$ with $[A]\in\mathcal A_G(\mathbb B^5)$. 
\begin{corollary}[controlled solutions to the radial gauge fixing ODE]\label{wsolode}
 Assume that to $[A]\in\mathcal A_G(\mathbb B^5)$ there corresponds a connection form $A$ and a curvature form $F_A$, both of which are in $L^2(\mathbb B^5)$, like in the definition of $\mathcal A_G(\mathbb B^5)$. Further assume that we have the control
 \begin{equation}\label{tracecontrol}
 \lim_{r\to 0} \frac{1}{r}\int_{\mathbb B\setminus\mathbb B_{1-r}}|A|^2 <\infty.
 \end{equation}
If $\rho$ is the radial coordinate $\rho(x)=1-|x|$ on $\mathbb B^5$ then for fixed $t\in[0,1[$ there exists a solution $g\in W^{1,2}(\mathbb B^5, G)$ to the following ODE:
\begin{equation}\label{ODE}
 \left\{\begin{array}{ll}
        \partial_\rho g= -A_\rho g&\quad\text{on }\mathbb B_1\setminus \mathbb B_{1-t}\ ,\\
        g(\omega,0)=id&\quad\text{for }\omega\in\mathbb S^4\ .
        \end{array}
\right.
\end{equation}
In particular the form $A^g:=g^{-1}dg+g^{-1}Ag$ is still $L^2$ and has zero component in the direction $\partial/\partial\rho$ and the formula 
\begin{equation}\label{newf}
F_{A^g}=g^{-1}F_Ag
\end{equation}
holds in the sense of distributions, once we define, $F_{A^g}:\stackrel{\mathcal D'}{=}dA^g+A^g\wedge A^g$.
\end{corollary}
This result should be compared to the theory of \cite{dipli} and \cite{am}, \cite{dlcri} where Lipschitz solutions $g:[0,T]\times\mathbb R^d\to\mathbb R^d$ to the nonlinear ODE $\partial_t g(x,t)=X(t,g(x,t))$ are found, under the requirement that $X\in L^\infty, \op{div}X\in L^\infty$. In that case the existence result is based on the theory of renormalized solutions for the related PDE. In our setting the ODE \eqref{ODE} is linear and the requirement for $g$ to be a renormalized solution appears in the form \eqref{newf} and follows from the fact that $g$ is ensured to be $W^{1,2}$. On the other hand we don't need the incompressibility condition (see e.g. the definition of a regular Lagrangian flow in \cite{dlcri}).\\

What allows this new result is the fact that while in the cited works the existence is ensured by approximating the driving field $X$ by smooth ones through a mollification, in our case the regularization is done via Theorem \ref{naturality}, which is better adapted to the geometry of the flows. Therefore finding a nonlinear generalization of the above result could help improving the theory of weak flows.

\subsection{Coulomb gauge extraction result for weak curvatures with small densities}
We first improve the result of Theorem \ref{naturality} to an approximation result for Morrey curvatures, reading as follows:
\begin{theorem}[Morrey counterpart of Theorem \ref{naturality}]\label{mapproxd}
There exist constants $C,\epsilon_1$ with the following properties. Let $F$ be the curvature form corresponding to an $L^2$ connection form $A$ with $[A]\in \mathcal A_G(\mathbb B^5)$. Assume that 
\begin{equation}
 \label{morrcontrol1}
\sup_{x,r}\frac{1}{r}\int_{B_r(x)}|F|^2< \epsilon_1\ .
\end{equation}
Then we can find curvature forms $\hat F_k$ corresponding to smooth connection forms $\hat A_k$ such that
\begin{equation}
 \label{m1}
\|\hat F_k - F\|_{L^2(\mathbb B^5)}\to 0\ ,
\end{equation}
\begin{equation}
 \label{m11}
\|\hat A_k - A\|_{L^2(\mathbb B^5)}\to 0\ ,
\end{equation}
and
\begin{equation}
 \label{m2}
\sup_{x,r}\frac{1}{r}\int_{B_r(x)}|\hat F_k|^2< C\epsilon_1\ .
\end{equation}
\end{theorem}
We recall that the Morrey norms of a function $f$ are defined as follows:
\[
 \|f\|_{M^{k,p}_\alpha(\mathbb B^n)}:=\left(\sup_{x\in\mathbb B^n, r>0}\frac{1}{r^{n-\alpha p}}\int_{B_r(x)}\sum_{i=0}^k|\nabla^if|^p\right)^{\frac{1}{p}}\ .
\]
Thus the above theorem asserts that for curvature forms which are $M^{0,2}_2$-small on $\mathbb B^5$, Theorem \ref{naturality} can be refined to ensure uniform $M^{0,2}_2$ bounds for the curvatures of the approximating smooth connections, as well as the strong $L^2$-convergence of the connection forms.\\
Combining the previous approximation result with the Coulomb gauge extraction method of \cite{TaoTian} for admissible connections or the one of for smooth connections in Morrey spaces, we have the following theorem which generalizes to the general framework of weak connections the main results of \cite{TaoTian} and \cite{MeRi}.
\begin{theorem}[Coulomb gauge extraction in Morrey norm]\label{merinew}
There exist constants $\epsilon, C$ depending only on the dimension such that the following holds. Let $F$ be a weak curvature corresponding to an $L^2$ connection form $A$ with $[A]\in\mathcal A_{G}(\mathbb B^5)$ and assume that 
\[
 \sup_{x,r}\frac{1}{r}\int_{B_r(x)}|F|^2:=\|F\|_{M^{0,2}_2(\mathbb B^5)}^2\leq\epsilon\ .
\]
Then there exists a gauge change $g\in W^{1,2}(\mathbb B^5, G)$ such that the transformed connection form $A^g=g^{-1}dg+g^{-1}Ag$ satisfies
\begin{eqnarray}
 d^*A^g&=&0\text{ in  }\mathbb B^5\ ,\label{mcoulomb}\\
\left\langle A^g,\frac{\partial}{\partial r}\right\rangle&=&0\text{ on }\partial\mathbb B^5\ ,\label{mnormal}
\end{eqnarray}
\begin{equation}\label{estmor}
\left(\sup_{x,r}\frac{1}{r}\int_{B_r(x)}|A^g|^4\right)^{\frac{1}{4}} + \left(\sup_{x,r}\frac{1}{r}\int_{B_r(x)}|D A^g|^2\right)^{\frac{1}{2}}\leq C\|F\|_{M^{0,2}_2(\mathbb B^5)}\ .
\end{equation}
\end{theorem}
\subsection{$\epsilon$-regularity result for stationary weak curvatures in $\mathcal A_{G}(\mathbb B^5)$}\label{ch:ereg}
The main result of \cite{MeRi} together with Theorem \ref{merinew} gives the $\epsilon$-regularity:
\begin{theorem}[$\epsilon$-regularity]\label{eregm}
 There exists a constant $\epsilon>0$ such that the following holds. Let $F$ be a weak curvature corresponding to an $L^2$ connection form $A$ with $[A]\in\mathcal A_{G}(\mathbb B^5)$, such that for all smooth perturbations $\eta\in C^\infty_0(\mathbb B^5, \wedge^1 \mathbb B^5\otimes\mathfrak{g})$ there holds
\begin{equation}\label{pcrit}
 \frac{d}{dt}\left.\int_{\mathbb B^5}|F_{A+t\eta}|^2\right|_{t=0}=0
\end{equation}
and such that for all vector fields $X\in C^\infty_0(\mathbb B^5, \mathbb R^5)$ the function $\phi_t:=id+tX$ satisfies
\begin{equation}\label{stat}
 \frac{d}{dt}\left.\int_{\mathbb B^5}|\phi_t^*F_A|^2\right|_{t=0}=0\ .
\end{equation}
Assume that
\[
 \frac{1}{r}\int_{B_r(x_0)}|F|^2\leq\epsilon\ .
\]
Then $F$ is the curvature form of a smooth connection over $B_{r/2}(x_0)$.
\end{theorem}
Because of the above theorem we can also extend the regularity result of \cite{MeRi}: 
\begin{corollary}[partial regularity for stationary weak curvatures]\label{preg}
Let $F$ be a weak curvature corresponding to an $L^2$ connection form $A$ with $[A]\in\mathcal A_{G}(\mathbb B^5)$, satisfying \eqref{pcrit} and \eqref{stat}.\\
Then there exists a closed set $K\subset\mathbb B^5$ such that $\mathcal H^1(K)=0$ and locally around every point in $\mathbb B^5\setminus K$ there exist a gauge change such that $A^g$ is a smooth form.
\end{corollary}
\subsection{The Yang-Mills-Plateau problem in dimension $n=5$: a definition of weak traces}\label{traces}
Since an element $[A]\in\mathcal A_{G}(\mathbb B^5)$ is only assumed to be in $L^2$ it seems a priori problematic to define its trace on $\partial\mathbb B^5$ in order to pose the Yang-Mills Plateau problem in $\mathcal A_{G}(\mathbb B^5)$ and take advantage of the Sequential Weak Closure Theorem \ref{wclos}. To obtain a suitable notion of trace, the following idea introduced in \cite{Ptrace} is used. Consider the following distance on$\mathcal A_{G}(\mathbb B^5)$:
\[
  \op{dist}([A],[A']):=\min\{\|A-g^{-1}dg - g^{-1}A'g\|_{L^2(\mathbb S^4)}:\:g\in W^{1,2}(\mathbb S^4, G)\}\ .
\]
Consider the boundary connection $\phi$ as a special slice and impose an oscillation bound for nearby slices. More precisely, we have the following definition:
\begin{definition}[boundary trace for $\mathbb B^5$]
 For a given connection form $\phi\in\mathfrak A^{1,2}(\mathbb S^4)$ we define the space of weak connection classes $[A]$ over $\mathbb B^5$ having trace in the class $[\phi]$ as follows:
 \begin{equation}\label{Fphi}
  \mathcal A_{G}^\phi(\mathbb B^5):=\mathcal A_{G}(\mathbb B^5)\cap
  \left\{
  \begin{array}{c}
  [A]\text{ s.t. for }r\uparrow 1,\quad r\notin N\\[3mm]
  \text{there holds }\op{dist}([A(r,0)],[\phi])\to 0\ .
  \end{array}
\right\}\ ,
 \end{equation}
 where $N$ is a Lebesgue-null set and $A(r,0)$ is the a.e.-defined $L^2$ form $\tau_r^*A$ on $\mathbb S^4$ obtained by pulling back $A$ via the homothety $\tau_r:\mathbb S^4\to\partial B_r(0)$.
\end{definition}
The following result whose proof is similar to the one for the abelian case \cite{Ptrace} guarantees that $\mathcal A_{G}^\phi(\mathbb B^5)$ is the right space on which to define the minimization of the Yang-Mills energy while fixing the boundary trace of the connection.
\begin{theorem}[properties of the trace]\label{trace}
 The classes $\mathcal A_{G}^\phi(\mathbb B^5)$ satisfy the following properties:
 \begin{enumerate}
\item\textbf{(closure)}\label{wellp} for any $1$-form $\phi\in\mathcal A_G(\mathbb S^4)$, the class $\mathcal A_{G}^{\phi}(\mathbb B^5)$ is closed under sequential weak $L^2$-convergence of the corresponding curvature forms $F$.
\item\textbf{(nontriviality)}\label{nontr} if $\phi,\psi$ are $1$-forms in $\mathcal A_G(\mathbb S^4)$ such that $[\phi]\neq[\psi]$ as gauge-equivalence classes, then $\mathcal A_{G,}^{\phi}(\mathbb B^5)\cap\mathcal A_{G}^{\psi}(\mathbb B^5)=\emptyset$.
\item\textbf{(compatibility)}\label{compa} for any smooth connection $1$-form $\phi$, $\nabla$ is a connection of a classical bundle over the finitely punctured ball $E\to \mathbb B^5\setminus\{p_1,\ldots,p_k\}$ satisfying $i^*_{\mathbb S^4}A\in[\phi]$ if and only if the corresponding connection form $A$ belongs to $\mathcal A_{G}^\phi(\mathbb B^5)$.
\end{enumerate}
\end{theorem}

Combining now Theorem \ref{wclos} and Theorem \ref{trace} we obtain the following, which is one of the main results of the present work (see \cite{isobe1} Cor. 4.2 and also \cite{marini} for the critical case):

\begin{theorem}[\textbf{Yang-Mills-Plateau solution in dimension $5$}]\label{YMP5}
 For all $\phi\in\mathcal A_G(\mathbb S^4)$ there exists a minimizer $[A]\in\mathcal A_{G}^\phi(\mathbb B^5)$ to the following Yang-Mills Plateau problem:
 \begin{equation}\label{YMP}
  \inf\left\{\int_{\mathbb B^5}|F|^2:\:F\stackrel{\mathcal D'}{=}dA+A\wedge A,\:[A]\in\mathcal A^\phi_{G}(\mathbb B^5)\right\}\ .
 \end{equation}
\end{theorem}
The analogous result for the case of $G=U(1)$ was proved in \cite{Ptrace} using the result \cite{PR1}.

\subsection{Optimal regularity result for Yang-Mills Plateau minimizers}
Since we work in the natural class $\mathcal A_{G}^\phi(\mathbb B^5)$ in which a Yang-Mills minimizer exists according to Theorem \ref{YMP5}, we may then apply Federer dimension reduction techniques and obtain:
\begin{theorem}[optimal partial regularity for Yang-Mills-Plateau minimizers]\label{rymp}
 Let $\phi$ be a smooth $\mathfrak{g}$-valued connection $1$-form over $\partial \mathbb B^5$. Then for the minimizer of 
\[
 \inf\left\{\|F_A\|_{L^2(\mathbb B^5)}:\:[A]\in\mathcal A_{G,\phi}(\mathbb B^5)\right\}
\]
the corresponding class $[A]\in\mathcal A_{G,\phi}(\mathbb B^5)$ has a representative which is locally smooth outside a set of isolated points.
\end{theorem}
An analogue of this result was proven by a completely different, combinatorial technique in \cite{Preg} for the case of $U(1)$-curvatures.\\

The result of Theorem \ref{rymp} is optimal in the following sense. Recall that in \cite{HLs} it was proven that there exist smooth boundary data for harmonic maps $u:\mathbb B^3\to\mathbb S^2$ such that the energy-minimizing harmonic map would need to have a bounded from below number of singularities. By a similar procedure it is possible to find smooth connection forms $\phi$ on bundles over $\partial \mathbb B^5$ for which the minimizers of \eqref{YMP} are forced to have singularities.\\

Therefore in general (even in the case when the connection corresponding to $\phi$ does not have nontrivial topology) we cannot expect the minimizers of \eqref{YMP} to be smooth. See \cite{Pradmin} for an example of Yang-Mills minimizer with nonempty singular set.\\

In the case of isolated singularities of minimizers like the ones issued from Theorem \ref{rymp} the \emph{uniqueness of tangent cones} was also proved by B. Yang \cite{byang}.

\subsection{Yang-Mills equation and and Bianchi identity}
Note that the requirement \eqref{pcrit} for all $\eta\in C^\infty_0(\mathbb B^5, \wedge^1 \mathbb B^5\otimes\mathfrak{g})$ is equivalent to the fact that the equation
\begin{equation}\label{distrpcrit}
 d(*F) + [*F,A]=0
\end{equation}
holds in the sense of distributions. We say that $[A]\in\mathcal A_{G}(\mathbb B^5)$ is a \emph{weak Yang-Mills connection} in this case.\\

The related works \cite{MeRi}, \cite{Tian}, \cite{TaoTian} proved regularity results analogous to our Corollary \ref{preg} under stronger assumptions, e.g. requiring the limit connection to be approximable in some sense. Our main contribution in this direction is indeed the approximability Theorem \ref{mapproxd}, which allows to extend such results to the space of weak connections on singular bundles $\mathcal A_{G}(\mathbb B^5)$.\\

As a consequence of our strong convergence result as in Theorem \ref{mapproxd} we obtain the following
\begin{proposition}[Bianchi identity for weak curvatures]\label{Bianchi}
 Assume that $A, F$ are the $L^2$ curvature and connection forms corresponding to a weak connection class $[A]\in\mathcal A_{G}(\mathbb R^5)$. Then the equation
\begin{equation}\label{bianchi}
 d_AF:=dF+[F,A]=0
\end{equation}
holds in the sense of distributions.
\end{proposition}

\subsection{From weak notions of Hermitian connections to measure-theoretic singular bundles}
Since the beginning of the mathematical blooming of Yang-Mills theory e.g. with the results of Donaldson \cite{Do} in $4$ dimensions, one of the main themes and inspirations has been the relation with complex algebraic geometry. We recall schematically the evolution of the notion of singular bundles during the last decades, with the goal of highlighting the role that the new spaces introduced in this work could play in future developments.
\subsubsection{Hermitian connections on K\"ahler manifolds and reflexive coherent sheaves}\label{hermitian}
Let $(M^{2m}, \omega, J_M)$ be a K\"ahler manifold and let $\pi:E\to M$ be a rank-$k$ Hermitian vector bundle over $M$ endowed with a complex structure $J_E$ on the fibres. A smooth connection $A$ over $E$ induces a unique compatible almost complex structure on $E$ defined as follows:
\begin{equation*}
 J^A(X) := J_E(\op{Vert}^A(X))+(J_M(\pi_*X))^A\ ,
\end{equation*}
where $\op{Vert}^A$ is the vertical projector associated to $A$, the compatibility condition being that $F_A$ have vanishing anti-holomorphic part:
\begin{equation}
\label{compat}
 F_\nabla^{0,2}=0\ ,
\end{equation}
or more explicitly, if $X^A$ denotes the horizontal lift of $X$ with respect to $A$ this amounts to require
\[
 \forall X,Y\quad\op{Vert}_A\left([X^A - iJ^AX^A,Y^A - iJ^AY^A]\right)\ .
\]
By the Nijenhuis-Newlander-Nirenberg theorem \cite{nijenhuis}, \cite{newlnir} this amounts to requiring the integrability of $J^A$, i.e. the existence of local gauge changes $g:U\to GL(k,\mathbb C)$ such that
\[
 (A^g)^{0,1}=0,\quad (A^g)^{1,0}=\sigma^{-1}\partial\sigma\ ,
\]
where the holomorphic and anti-holomorphic parts are again with respect to $J^A$ and $\sigma$ is the unique Hermitian metric compatible with $A$. This $g$ is the correct analogue of the Coulomb gauge since the self-duality equation for compatible connections $\omega\cdot F_A^{1,1}$ becomes in this gauge $\omega\cdot\overline{\partial}(\sigma^{-1}\partial\sigma)=0$, and since $\omega\cdot\overline{\partial}\partial=\Delta$, this is a nonlinear elliptic equation like \eqref{ymccoord}.\\

The resolution of the Kobayashi-Hitchin conjecture by Donaldson \cite{donaldkh} in $4$ dimensions and by Uhlenbeck-Yau \cite{uhlyaukh} in the general case presented a first appearance of singular versions of Hermitian bundles as limits of holomorphic vector bundles with compatible hermitian connections. The whole framework of the Kobayashi-Hitchin correspondence was extended to the more singular class of coherent reflexive sheaves (for a definition see \cite{kobayashi}, Ch.5) by Bando and Siu \cite{bandosiu} completing a first extension of the Hermitian Yang-Mills theory towards a notion of more singular bundles. In this setting thus the following result was obtained
\begin{theorem}[\cite{donaldkh},\cite{uhlyaukh}, \cite{bandosiu}, \cite{Tian}]
\label{hermitth}
 Let $E$ be an $SU(n)$-bundle over a compact K\"ahler manifold $(M^{2m}, \omega, J_M)$. Let $\nabla_k$ be a sequence of smooth Hermitian Yang-Mills connections. Then, modulo extraction of a subsequence, there exists a family of codimension-$4$ holomorphic subvarieties $C_i\subset M$ and a reflexive coherent sheaf $\mathcal E$ over $M$ which is locally free over $M\setminus\cup_i C_i$ such that
 \[
  \nabla_k\to\nabla_\infty\quad\text{ strongly in }C^l_{loc}\text{ over }M\setminus\cup_i C_i\ ,
 \]
for any $l\in \mathbb N$ and there exist integers $m_i$ such that in the sense of $(2m-4)$-currents there holds
\[
 \op{tr}\left(F_{\nabla_k}\wedge F_{\nabla_k}\right)\rightharpoonup \op{tr}\left(F_{\nabla_\infty}\wedge F_{\nabla_\infty}\right)+8\pi^2\sum_jm_j[C_j]\ .
\]
\end{theorem}
By the algebraic theory of reflexive coherent sheaves (see \cite{kobayashi}) it follows in particular that $\nabla_\infty$ is smooth aside for a set of complex codimension at least $3$.
\subsubsection{$\Omega$-anti-self-dual instantons and singular bundles}
A further step away from the smooth or algebraic setting was introduced by Tian \cite{Tian} for the compactification of $\Omega$-anti-self-dual instantons. In this setting the assumption of having a K\"ahler manifold was relaxed to just having a Riemannian manifold $(M^m, g)$ endowed with a $(m-4)$-form $\Omega$ assumed to be a calibration, i.e. a closed form of co-mass $1$. This form $\Omega$ plays the role that $\omega^{m-2}/(m-2)!$ was playing in the K\"ahler setting. We call a connection $\nabla$ over an $SU(n)$-bundle over $M$ an \emph{$\Omega$-anti-self-dual instanton} if it is anti-self-dual with respect to $\Omega$, i.e. if
\[
 *F_\nabla=-F_\nabla\wedge\Omega\ .
\]
The compactification result available in this case is the following.
\begin{theorem}[\cite{Tian}, \cite{TaoTian}]
\label{oasdth}
 Let $\nabla_k$ be a sequence of smooth $\Omega$-anti-self-dual instantons over a Riemannian manifold $(M^m,g)$. Then there exists an $(m-4)$-rectifiable set $K\subset M$ such that up to extracting a subsequence we find an $\Omega$-anti-self-dual instanton $\nabla_\infty$ over $M\setminus K$ such that for any $l\in\mathbb N$
 \[
   \nabla_k\to\nabla_\infty\quad\text{ strongly in }C^l_{loc}\text{ over }M\setminus K\ ,
 \]
and there exists an $(m-4)$-current of integer multiplicity $C$ calibrated by $\Omega$ and supported on $K$ such that in the sense of $(m-4)$-currents
\[
  \op{tr}\left(F_{\nabla_k}\wedge F_{\nabla_k}\right)\rightharpoonup \op{tr}\left(F_{\nabla_\infty}\wedge F_{\nabla_\infty}\right)+8\pi^2 C\ .
\]
\end{theorem}
We recall the following regularity conjecture made by Tian \cite{Tian} for $\Omega$-anti-self-dual curvatures, analogous to the regularity following from Theorem \ref{hermitth}:
\begin{conjecture}[\textbf{Tian's regularity conjecture}]
\label{conjtian}
 Assume $\Omega$ is a smooth closed differential $(n-4)$-form on a compact $n$-dimensional Riemannian manifold $M$. Curvature forms corresponding to $\Omega$-anti-self-dual instantons have a singular set of Hausdorff dimension $\leq n-6$.
\end{conjecture}
The resolution of this conjecture would be of particular geometric interest on Calabi-Yau $4$-folds or on $G_2$-manifolds, where $\Omega$ is a parallel form invariant by the special holonomy (see \cite{DoTh} and \cite{Tian}).\\

Note that the above conjecture is not true without an assumption on the smoothness of $\Omega$, as shown by the example present in \cite{Pradmin}.

\subsubsection{Hermitian weak connections on singular bundles}
We may define $\mathcal A_{G}(\mathbb B^n)$ in a stratifying way: by requiring that $A\in L^2$, $F\in L^2$ and for all centers $x$ and almost all radii $r>0$ the restriction $i^*_{\partial B_r(x)}A$ belongs, up to measurable gauge and rescaling, to $\mathcal A_{G}(\mathbb S^{n-1})$. This definition extends to compact Riemannian $n$-manifolds by requiring $A$ to be locally equivalent to a form in $\mathcal A_{G}(\mathbb B^n)$.\\

We prove in a future work \cite{PR4} that the techniques and proofs of our main results in the present paper extend to general compact Riemannian manifolds and to higher dimension.\\

We are then in the position of extending the setting of Section \ref{hermitian} to the setting of weak connections. If $(M^{2m}, \omega, J_M)$ is a  K\"ahler manifold and $\nabla\in \mathcal A_{SU(n)}(M^{2m})$ then the condition \eqref{compat} can be imposed by requiring
\begin{equation}
 \label{compatweak}
  F_\nabla^{0,2}:=\overline{\partial}A^{0,1} + [A^{0,1}, A^{0,1}]=0\ ,
\end{equation}
in the sense of distributions. We then formulate the following conjecture, analogous to Theorem \ref{naturality} for the case of compatible connections. Solving this conjecture would allow to tackle Conjecture \ref{conjtian} in the special case of Hermitian bundles over K\"ahler manifolds in $6$ dimensions.
\begin{conjecture}
\label{smoothcstr}
Let $(M^6, \omega, J_M)$ be a compact K\"ahler manifold and $\nabla\in \mathcal A_{SU(n)}(M^6)$ be a Hermitian weak connection such that \eqref{compatweak} holds in the sense of distributions. Then there exist finite sets $\Sigma_k\subset M^6$ and connections $\nabla_k$ over the manifolds $M^6$ such that $\nabla_k$ is locally smooth holomorphic outside $\Sigma_k$ and that $\nabla_k$ approximate $\nabla$ in the sense of Theorem \ref{naturality}.\\
Moreover the pull-backs of the connections $\nabla_k$ over the manifolds $M_k^6$ obtained from $M^6$ by blowing it up at the points $\Sigma_k$ extend to globally smooth holomorphic connections.
\end{conjecture}
The conjecture says that the almost complex structure induced by $\nabla$ becomes holomorphic after possibly blowing up the metric at a countable number of points.

\subsubsection{Everywhere discontinuous weak connections}\label{everywherediscont}
Note that in the setting of Theorem \ref{oasdth} it follows directly that
\[
 \partial C=0\quad \Longleftrightarrow\quad d\left(\op{tr}\left(F_{\nabla_\infty}\wedge F_{\nabla_\infty}\right)\right)=0\ ,
\]
in which case if we were in the setting of a K\"ahler manifold and under the condition \eqref{compat} then it would follow from results of Harvey-Shiffman \cite{harveyshiff} and Siu \cite{siu} that $C$ would be a complex analytic variety.\\

In general we are very far from the analogue of this situation in our $5$-dimensional case, without imposing further restrictions. We already see that if $[A]\in\mathcal R^\infty(\mathbb B^5)$ then it is not true anymore, as in the smooth case, that $d\left(\op{tr}(F\wedge F)\right)=0$. We have indeed
\[
 d\left(\op{tr}(F\wedge F)\right)=8\pi^2\sum_{i=1}^kd_i\delta_{a_i}\quad\text{ in }\mathcal D'(\mathbb B^5)\ ,
\]
where
\[
 d_i=\int_{\partial B_r(a_i)}\op{tr}(F\wedge F)\in\mathbb Z
\]
represent the degrees of topological singularities situated at the points $a_1\ldots, a_k$. For a general element $[A]\in\mathcal A_{G}(\mathbb B^5)$ one can then ask ``how many'' such topological singularities exist.\\ Following the procedure of \cite{KR}, \cite{Kessel} (in which our approximation theorem is stated as a conjecture) one obtains using the new result of Theorem \ref{naturality} the following:
\begin{theorem}[see \cite{Kessel},\cite{KR}]
\label{conj}
 If $F$ is a curvature form of a connection $A$ with $[A]\in \mathcal A_{G}(\mathbb B^5)$ then there exists a rectifiable integral $1$-current $I$ such that 
\[
 \partial I=\frac{1}{8\pi^2}d(\op{tr}(F\wedge F)),\quad \mathbb M(I)\leq C\|F\|_{L^2(\mathbb B^5)}.
\]
where $C$ is a universal constant.
\end{theorem}
Following the seminal works of Brezis, Coron and Lieb \cite{BCL} and of Giaquinta, Modica and Sou\v{c}ek \cite{GMS}, we can define the relaxed energy for connection classes $[A]\in\mathcal A_{G}(\mathbb B^5)$ in terms of their curvature form $F$ as a supremum taken over $1$-Lipschitz functions $\xi$ over $\mathbb B^5$:
\begin{equation}\label{ymr}
 \mathcal{YM}_{rel}(F):=\int_{\mathbb B^5}|F|^2 + \sup_{|d\xi|_\infty\leq 1}\left[\int_{\mathbb B^5}d\xi\wedge\op{tr}(F\wedge F) -\int_{\mathbb S^4}\xi\ \op{tr}(F\wedge F)\right]\ .
\end{equation}
In \cite{isobe0} it was proven that the minimization of $\mathcal{YM}_{rel}$ over $\mathcal R^\infty_\phi(\mathbb B^5)$ presents a gap phenomenon analogous to the celebrated theory of harmonic maps \cite{BBC}, \cite{BB}. We expect the relaxed energy to be lower-semicontinuous in $\mathcal A_{G}(\mathbb B^5)$, in particular it is natural to ask :
 \[
\forall \, \phi\in\mathcal A_G(\mathbb S^4)\ \text{ is }\ \inf_{\mathcal A_{G}^\phi(\mathbb B^5)}\mathcal{YM}_{rel}(F_A)\quad\text{achieved ?}
 \]
Using the relaxed energy 
\[
 \mathcal{YM}_{rel}(F,G)=\int_{\mathbb B^5}|F|^2 + \sup_{|d\xi|_\infty\leq 1}\int_{\mathbb B^5}d\xi\wedge\left[\op{tr}(F\wedge F) - \op{tr}(G\wedge G)\right]\ ,
\]
like in \cite{rivsing} one should be able to construct weak Yang-Mills curvatures $F$ corresponding to $[A]\in\mathcal A_{G}(\mathbb B^5)$ of arbitrarily small Yang-Mills energy and such that the topological singular set is dense:
\[
 \op{spt}\left(d\left(\op{tr}(F\wedge F)\right)\right)=\overline{\mathbb B^5}\ .
\]
In other words, if one would succeed in carrying over the strict dipole inclusion construction from \cite{rivsing}, one should be able to construct \emph{everywhere discontinuous Yang-Mills connections}.
\subsection{Plan of the paper}
The paper is organized as follows. 

In Section \ref{approx5d} we prove the approximation results of Theorem \ref{naturality} and of Theorem \ref{mapproxd}.\\
In Section \ref{ch:wclos} we prove Proposition \ref{wsolode} and the weak closure theorem \ref{wclos}.\\
In Section \ref{ch:reg} we prove the regularity results of Theorem \ref{eregm}, Corollary \ref{preg} and Theorem \ref{rymp}. At the beginning of the section we include a short proof of Proposition \ref{Bianchi}.\\
In Section \ref{pfcorolldens} we prove the properties of the trace stated in Theorem \ref{trace}.\\
Appendix \ref{app1} is dedicated to a modification of the Coulomb gauge extraction of K. Uhlenbeck \cite{Uhl2} which is needed in Section \ref{approx5d} for the proof of the approximation under Morrey norm smallness of Theorem  \ref{mapproxd}. The remaining appendices are required to prove auxiliary results.

\section{Approximation of nonabelian curvatures in 5 dimensions}\label{approx5d}
In this section we prove the fact that weak curvatures $F$ corresponding to classes $[A]\in\mathcal A_{G}(\mathbb B^5)$ can be strongly approximated up to gauge by smooth curvatures on bundles with finitely many defects. We consider the class
\begin{equation}\label{Rinf}
 \mathcal R^{\infty}(\mathbb B^5):=\left\{
 \begin{array}{c}
  F\text{ curvature form s.t. }\exists k,\exists a_1,\ldots,a_k\in \mathbb B^5,\\[3mm]
  F=F_\nabla \text{ for a smooth connection}\nabla\\[3mm]
  \text{on some smooth }G\text{-bundle }E\to \mathbb B^5\setminus\{a_1,\ldots,a_k\}
 \end{array}
\right\}\ .
\end{equation}

\subsection{Approximation on balls with small $A$ and $F$}\label{ext1gc}
In this section we prove the extension result which will help to define our approximating connections. We consider the scale $r=1$.
\begin{proposition}
\label{goodballext}
 Let $F\in L^2(\mathbb B^5_2,\wedge^2\mathbb R^5\otimes\mathfrak g)$ and $A\in L^2(\mathbb B^5_2,\wedge^1\mathbb R^5\otimes\mathfrak g)$ be such that in the sense of distributions
\[
 F=dA+A\wedge A \quad\text{ on }\mathbb B^5_2\ .
\]
Fix also a constant $\bar F\in\wedge^2\mathbb R^5\otimes\mathfrak g$ and a constant $\bar A\in\wedge^1\mathbb R^5\otimes\mathfrak g$. There exists a constant $\epsilon_0>0$ independent of the other choices such that if
\begin{equation*}
 \int_{\mathbb S^4}|F|^2< \epsilon_0,\quad\int_{\mathbb S^4}|A|^2<\epsilon_0,\quad |\bar A|^2<\epsilon_0
\end{equation*}
then there exists $\hat A\in L^2(\mathbb B^5_2,\wedge^1\mathbb R^5\otimes\mathfrak g)$ and  $\hat g:\mathbb B^5\to G$ such that:
 \begin{enumerate}
  \item $i^*_{\mathbb S^4}\hat A=i^*_{\mathbb S^4}A$ and $\hat A=A$ outside $\mathbb B^5$, while the distribution $F_{\hat A}=d\hat A+\hat A\wedge \hat A$ is represented by an $L^2$-form and coincides with $F$ outside $\mathbb B^5$.
  \item We have the following approximation bounds:
\begin{equation}\label{AA}
 \|d\hat A + \hat A\wedge \hat A - \bar F\|_{L^2(\mathbb B^5)}^2\lesssim \epsilon_0(\|\bar F\|_{L^2(\mathbb B^5)}^2 + \|F\|_{L^2(\mathbb S^4)}^2)+\|F-\bar F\|_{L^2(\mathbb S^4)}^2\ .
\end{equation}
\begin{equation}\label{AAA}
 \|\hat A - \bar A\|_{L^2(\mathbb B^5)}\leq C\|A-\bar A\|_{L^2(\mathbb S^4)}\ .
\end{equation}
\item The gauge-transformed form $\hat A^{\hat g}$ is smooth in the interior of $\mathbb B^5$.
\end{enumerate}

\end{proposition}

\begin{proof}
\textbf{Step 1.}\textit{ Coulomb gauge on the boundary.} We start by applying the Theorem \ref{coulstrange} to the restriction $i^*_{\mathbb S^4}A$ above. Let $g$ be the change of gauge $g$ given by Theorem \ref{coulstrange} such that (recalling \eqref{A-1} and \eqref{A-2} from the appendix)
\begin{equation}
\label{A-1*}
\left\{
\begin{array}{l}
d^\ast_{\mathbb S^4}\pi(A^g)=d^\ast_{\mathbb S^4}(g^{-1}dg+\pi(g^{-1} A g))=0\ ,\\[3mm]
\|A^g\|_{W^{1,2}(\mathbb S^4)}\le C(\|F\|_{L^2(\mathbb S^4)}+\|A\|_{L^2(\mathbb S^4)})\ .
\end{array}
\right.
\end{equation}
and
\begin{equation}
\label{A-2*}
\|dg\|_{L^2(\mathbb S^4)}\leq C\|A-\bar A\|_{L^2}.
\end{equation}
We have using \eqref{A-2*} and the fact that $\overline{F}$ is constant
\[
\int_{\mathbb S^4}|g^{-1}i_{\mathbb S^4}^\ast\overline{F} g-i_{\mathbb S^4}^\ast\overline{F}|^2\le 4\ |\overline{F}|^2 \ \int_{\mathbb S^4}|g-id|^2\lesssim\,\epsilon_0\ \|\overline F\|^2_{L^2(\mathbb B^5_2)}\ .
\]
Since $F_{A^g}=g^{-1}\, F\, g$, using the previous identity we obtain
\begin{equation},
\label{A-7}
\int_{\mathbb S^4}|F_{A^g}-i_{\mathbb S^4}^\ast\overline{F}|^2 \lesssim\,\epsilon_0\ \|\overline F\|^2_{L^2(\mathbb B^5)} +\,\int_{\mathbb S^4}|F-i_{\mathbb S^4}^\ast\overline{F}|^2 \ .
\end{equation}
Using now the last line of \eqref{A-1*} we obtain
\begin{eqnarray*}
\int_{\mathbb S^4}|F_{A^g}-d A^g|^2 &\leq& \int_{\mathbb S^4}|A^g|^4 \lesssim \|F\|_{L^2(\mathbb S^4)}^4+\|A\|_{L^2(\mathbb S^4)}^4\ .
\end{eqnarray*}
Combining this with \eqref{A-7} we obtain
\begin{equation}\label{A-9}
\begin{array}{rcl}
\int_{\mathbb S^4}|dA^g-i_{\mathbb S^4}^\ast\overline{F}|^2 &\lesssim&\ \epsilon_0\ \|\overline F\|^2_{L^2(\mathbb B^5)} +\,\int_{\mathbb S^4}| F-i_{\mathbb S^4}^\ast\overline{F}|^2 +\\[3mm]
&&+ \|F\|^4_{L^2(\mathbb S^4)}+ \|A\|^4_{L^2(\mathbb S^4)}\ .
\end{array}
\end{equation}

\textbf{Step 2.} \textit{Extension to the interior.} For any 1-form $\eta$ in $W^{1,2}(\mathbb S^4)$ we denote by $\tilde{\eta}$ the unique solution of the following minimization problem
\begin{equation}
\label{A-10}
\inf\left\{\int_{\mathbb B^5}|dC|^2+|d^{\ast_{{\mathbb R}^5}}C|^2\ dx^5\quad C\in W^{1,2}(\wedge^1\mathbb B^5)\quad i_{\mathbb S^4}^\ast C=\eta\right\}\ .
\end{equation}
A classical argument shows that it is uniquely given by
\begin{equation}
\label{eta}
\left\{
\begin{array}{l}
 d^{\ast_{{\mathbb R}^5}}\tilde{\eta}=0\quad\quad\text{ in } \mathbb B^5\ ,\\[3mm]
 d^{\ast_{{\mathbb R}^5}}(d\tilde{\eta})=0\quad\quad\text{ in } \mathbb B^5\ ,\\[3mm]
 i_{\mathbb S^4}^\ast \tilde{\eta}=\eta\quad\quad\text{ on } \partial \mathbb B^5\ ,
\end{array}
\right.
\end{equation}
and one has
\begin{equation}
\label{A-13}
\|\tilde{\eta}\|_{L^5(\mathbb B^5)}\le C\,\|\tilde{\eta}\|_{W^{3/2,2}(\mathbb B^5)}\le C\ \|\eta\|_{W^{1,2}(\mathbb S^4)}\ .
\end{equation}
Let
\begin{equation}
\label{A-14}
B:=\sum_{i<j}\overline{F_{ij}}\ \frac{x_i\,dx_j-x_j\, dx_i}{2}\ .
\end{equation}
Observe that
\[
\left\{
\begin{array}{l}
 d^{\ast_{{\mathbb R}^5}}B=0\quad\quad\text{ in } \mathbb B^5\ ,\\[3mm]
 d^{\ast_{{\mathbb R}^5}}(d B)=0\quad\quad\text{ in } \mathbb B^5\ .
\end{array}
\right.
\]
Thus $B$ is the solution to \eqref{A-10} for its restriction to the boundary : $i_{\mathbb S^4}^\ast B$
\[
\tilde{i_{\mathbb S^4}^\ast B}=B\ .
\]
Observe that $<B,dr>\equiv 0$ and $d^{\ast_{{\mathbb R}^5}}B=0$ therefore
\begin{equation}
\label{B}
d^{\ast_{\mathbb S^4}}\left(i_{\mathbb S^4}^\ast B\right)\equiv 0\quad\quad\text{ on } \mathbb S^4\ .
\end{equation}
We apply the same extension technique $\eta\mapsto\tilde\eta$ to $\eta=\pi(A^g)$ obtaining a $1$-form $\widetilde{\pi(A^g)}$ satisfying the analogues of \eqref{eta}. We also define the constant $1$-form
\[
 \overline{A^g}:=\sum_{k=1}^5dx_k \frac{1}{|\mathbb S^4|}\int_{\mathbb S^4}\langle A^g,i^*_{\mathbb S^4}dx_k\rangle
\]
and we note
\[
\tilde{A^g}=\widetilde{\pi(A^g)}+\ \overline{A^g}\ .
\]
\textbf{Step 3.} \textit{Estimates on the extended curvatures.} Note that $d\pi(A^g)=dA^g$ since $\overline{A^g}$ is constant. Using \eqref{A-2*}, \eqref{B} and \eqref{A-9} we have that by Hodge inequality
\begin{equation}
\label{A-17}
\begin{array}{l}
\|\pi(A^g)-i_{\mathbb S^4}^\ast B\|^2_{W^{1,2}(\mathbb S^4)}\le C\, \int_{\mathbb S^4}|d(\pi(A^g)-i_{\mathbb S^4}^\ast B)|^2 \\[3mm]
\quad\quad=\int_{\mathbb S^4}|dA^g-i_{\mathbb S^4}^\ast\overline{F}|^2     \le\, C\ \epsilon_0\ \|\overline{F}\|^2_{L^2(\mathbb B^5)} +\\[3mm]
\quad\quad+\,C\,\int_{\mathbb S^4}|F-i_{\mathbb S^4}^\ast\overline{F}|^2 +\,C\ \|F\|^4_{L^2(\mathbb S^4)}+\,C\ \|A\|^4_{L^2(\mathbb S^4)}\ .
\end{array}
\end{equation}
Combining now \eqref{A-13} and \eqref{A-17} we obtain
\begin{equation}
\label{A-18}
\begin{array}{l}
\|d\tilde{A^g}-\overline{F}\|^2_{L^2(\mathbb B^5)}=\|d\widetilde{\pi(A^g)}-\overline{F}\|^2_{L^2(\mathbb B^5)}\\[3mm]
\quad\quad\le C\, \int_{\mathbb S^4}|d(A^g-i_{\mathbb S^4}^\ast B)|^2 \, \le\, C\ \epsilon_0\ \|\overline{F}\|^2_{L^2(\mathbb B^5)} +\\[3mm]
\quad\quad+\,C\,\int_{\mathbb S^4}|F-i_{\mathbb S^4}^\ast\overline{F}|^2 +\,C\ \|F\|^4_{L^2(\mathbb S^4)}\ +\,C\ \|A\|^4_{L^2(\mathbb S^4)}\ .
\end{array}
\end{equation}
Using \eqref{A-13} again, we obtain
\begin{equation}
\label{A-19}
\|\tilde{A^g}\wedge\tilde{A^g}\|_{L^2(\mathbb B^5)}^2\lesssim\|\tilde{A^g}\|_{L^4(\mathbb B^5)}^4\le\ \|A^g\|_{W^{1,2}(\mathbb S^4)}^4\le C\,\|F\|_{L^2(\mathbb S^4)}^4+\,C\ \|A\|^4_{L^2(\mathbb S^4)}\ .
\end{equation}
Combining \eqref{A-18} and \eqref{A-19} we obtain
\begin{equation}
\label{A-20}
\begin{array}{l}
\|d\tilde{A^g}+\tilde{A^g}\wedge\tilde{A^g}-\overline{F}\|^2_{L^2(\mathbb B^5)}\le \, C\ \epsilon_0\ \|\overline{F}\|^2_{L^2(\mathbb B^5)} +\\[3mm]
\quad\quad+\,C\,\int_{\mathbb S^4}|F-i_{\mathbb S^4}^\ast\overline{F}|^2 +\,C\ \|F\|^4_{L^2(\mathbb S^4)}+\,C\ \|A\|^4_{L^2(\mathbb S^4)}\ .
\end{array}
\end{equation}
\textbf{Step 4.} \textit{Correcting the restriction on the boundary.} Extend now $g$ radially in $\mathbb B^5$ and denote by $\hat{g}$ this extension. We have using \eqref{A-2*}
\begin{equation}
\label{A-21}
\begin{array}{l}
\int_{\mathbb B^5}|\hat{g}^{-1}\overline{F}\hat{g}-\overline{F}|^2\le\, 4\ |\overline{F}|^2\ \int_{\mathbb B^5}|\hat{g}-id|^2\ dx^5\\[3mm]
\quad\quad\quad\le C\,\|\overline{F}\|_{L^2(\mathbb B^5)}^2\ \int_{\mathbb S^4}|g-id|^2 \le \ C\ \epsilon_0\ \|\overline{F}\|_{L^2(\mathbb B^5)}^2\ .
\end{array}
\end{equation}
Combining \eqref{A-20} and \eqref{A-21} gives
\[
\begin{array}{l}
\|d\tilde{A^g}+\tilde{A^g}\wedge\tilde{A^g}-\hat{g}^{-1}\overline{F}\hat{g}\|^2_{L^2(\mathbb B^5)}\le \, C\ \epsilon_0\ \|\overline{F}\|^2_{L^2(\mathbb B^5)} +\\[3mm]
\quad\quad+\,C\,\int_{\mathbb S^4}|F-i_{\mathbb S^4}^\ast\overline{F}|^2 +\,C\ \|F\|^4_{L^2(\mathbb S^4)}+\,C\ \|A\|^4_{L^2(\mathbb S^4)}\ .
\end{array}
\]
Denote $\hat{A}:=(\tilde{A^g})_{\hat g^{-1}}:=\hat{g}\tilde{A^g}\hat{g}^{-1}+\hat{g}d\,\hat{g}^{-1}$. Observe that with this notation one has $F_{\hat{A}}=\hat{g}\, F_{A^g}\, \hat{g}^{-1}$ thus the above estimate implies the estimate \eqref{AA}. Moreover we have $\hat A^{\hat g}=\tilde A^g$ in the previous notations, thus, being harmonic, $\hat A^{\hat g}$ is smooth in the interior of $\mathbb B^5$, as required by point 3. of the Proposition. Note that 
\[
i^*_{\mathbb S^4}\hat A=i^*_{\mathbb S^4}(\tilde A^g)_{\hat g^{-1}}=(i^*_{\mathbb S^4}\tilde A^g)_{\hat g^{-1}}=i^*_{\mathbb S^4}A\ .  
\]
Then define $\hat A=A$ outside $\mathbb B^5$. Since $i^*_{\mathbb S^4}(\hat A-A)=0$ and $A,\hat A$ are $L^2$ we obtain via integration by parts that both terms of the so-obtained distributional expression $F_{\hat A}=d\hat A+\hat A\wedge \hat A$ are well defined in $L^1_{loc}$. Since they also coincide on both sides of $\mathbb S^4$ with previously defined $L^2$ functions, we have $F_{\hat A}\in L^2$. Thus we verified the requirement of point 1. of the Proposition.\\
\textbf{Step 5.} \textit{Verification of \eqref{AAA}.} We now use the formula for $\hat A$ from the previous step, as well as the estimates \eqref{A-19} and \eqref{A-2*} to prove the following sequence of estimates:
\begin{eqnarray*}
 \|\hat A - \bar A\|_{L^2(\mathbb B^5)}^2&\lesssim&\int_{\mathbb B^5}|d\hat g|^2 + \|\hat g -id\|_{L^4(\mathbb L^4)}^2\|\bar A - \tilde A^g\|_{L^4(\mathbb B^5)}^2\\
 &\lesssim&(1+\epsilon_0)\left(\|dg\|_{L^2(\mathbb S^4)}^2+\ \|g-id\|_{L^4(\mathbb S^4)}^2\right)\\
 &\lesssim&\|A-\bar A\|_{L^2(\mathbb S^4)}\ .
\end{eqnarray*}
This concludes the proof.
\end{proof}

\subsection{Approximation under a smallness condition on $F$ only}
In this section we state a modification of Proposition \ref{goodballext} which can be applied when only a bound on $F$ and not one on $A$ is available. This modification will prove useful for Theorem \ref{mapproxd}.
\begin{proposition}[modified version of Prop. \ref{goodballext}]
\label{badballext}
  Let $F\in L^2(\mathbb B^5_2,\wedge^2\mathbb R^5\otimes\mathfrak g)$ and $A\in L^2(\mathbb B^5_2,\wedge^1\mathbb R^5\otimes\mathfrak g)$ be such that in the sense of distributions
\[
 F=dA+A\wedge A \quad\text{ on }\mathbb B^5_2\ .
\]
Fix also a constant $\bar F\in\wedge^2\mathbb R^5\otimes\mathfrak g$. There exists a constant $\epsilon_0>0$ independent of the other choices such that if
\begin{equation*}
 \int_{\mathbb S^4}|F|^2< \epsilon_0
\end{equation*}
then there exists $\hat A\in L^2(\mathbb B^5_2,\wedge^1\mathbb R^5\otimes\mathfrak g)$ and  $\hat g:\mathbb B^5\to G$ such that:
 \begin{enumerate}
  \item $i^*_{\mathbb S^4}\hat A=i^*_{\mathbb S^4}A$ and $\hat A=A$ outside $\mathbb B^5$, while the distribution $F_{\hat A}=d\hat A+\hat A\wedge \hat A$ is represented by an $L^2$-form and coincides with $F$ outside $\mathbb B^5$.
  \item We have the approximation bounds
\begin{equation}\label{AAb}
 \|d\hat A + \hat A\wedge \hat A \|_{L^2(\mathbb B^5)}^2\lesssim \|F\|_{L^2(\mathbb S^4)}^2
\end{equation}
\begin{equation}\label{AAAb}
 \|\hat A \|_{L^2(\mathbb B^5)}\leq  \|F\|_{L^2(\mathbb S^4)}^2 + \|A\|_{L^2(\mathbb S^4)}^2\ .
\end{equation}
\item The gauge-transformed form $\hat A^{\hat g}$ is smooth in the interior of $\mathbb B^5$.
\end{enumerate}

\end{proposition}
\begin{proof}
 We follow the proof of Proposition \ref{goodballext}, with slightly less refined estimates.\\ 
\textbf{Step 1.}\textit{ Classical Coulomb gauge on the boundary.} Let $g$ be the Coulomb gauge as constructed by Uhlenbeck \cite{Uhl2}, i.e. such that
\[
\left\{
\begin{array}{l}
d^\ast_{\mathbb S^4} A^g=d^\ast_{\mathbb S^4}(g^{-1}dg+g^{-1} A g)=0\ ,\\[3mm]
\|A^g\|_{W^{1,2}(\mathbb S^4)}\le C\|F\|_{L^2(\mathbb S^4)}\ .
\end{array}
\right.
\]
We deduce using the definition of $A^g$ that
\[
 \|dg\|_{L^2(\mathbb S^4)}^2\leq C\left(\|A^g\|_{L^2(\mathbb S^4)}^2+\|A\|_{L^2(\mathbb S^4)}^2\right)\lesssim \|F\|_{L^2(\mathbb S^4)}^2+\|A\|_{L^2(\mathbb S^4)}^2\ .
\]
\textbf{Steps 2-3.}\textit{ Estimates for the extensions.} We define $B$ as in Proposition \ref{goodballext} and $\tilde A^g$ will be the similar extension of $A^g$. By elliptic and Hodge estimates using the fact that $d^*_{\mathbb S^4}\tilde A^g=0$ we obtain
\[
 \|d\tilde A^g\|_{L^2(\mathbb B^5)}\lesssim\|F\|_{L^2(\mathbb S^4)}^2
\]
and 
\[
 \|\tilde A^g\wedge\tilde A^g\|_{L^2(\mathbb B^5)}\leq\|\tilde A^g\|_{L^4(\mathbb B^5)}^4\lesssim\|A^g\|_{L^4(\mathbb S^4)}^4\lesssim\epsilon_0\|F\|_{L^2(\mathbb S^4)}^2\ .
\]
These estimate give
\[
 \|F_{\tilde A^g}\|_{L^2(\mathbb B^5)}^2\lesssim \|F\|_{L^2(\mathbb S^4)}^2\ .
\]
\textbf{Step 4.}\textit{ Correcting the extension on the boundary.} We consider the harmonic extension $\tilde g$ to $g$. Note that $W^{1,2}(\mathbb B^5,G)$ is the strong $W^{1,2}$-closure of $C^\infty(\mathbb B^5,G)$ since $\pi_2(G)=0$, therefore the extension exists and is smooth. We also have the estimates
\[
 \|\tilde g - id\|_{L^2(\mathbb B^5)}^2\lesssim\|d\tilde g\|_{L^2(\mathbb S^4)}^2\lesssim\|dg\|_{L^2(\mathbb S^4)}^2\lesssim\|F\|_{L^2(\mathbb S^4)}^2+\|A\|_{L^2(\mathbb S^4)}^2\ ,
\]
thus if we define $\hat A=\tilde g\tilde A^g\tilde g^{-1} + \tilde gd\tilde g^{-1}$ it follows that 
\begin{eqnarray*}
 \|\hat A\|_{L^2(\mathbb B^5)}^2&\lesssim&\|\tilde A^g\|_{L^2(\mathbb B^5)}^2+\|d\hat g\|_{L^2(\mathbb B^5}^2\\
&\lesssim&\|F\|_{L^2(\mathbb S^4)}^2+\|A\|_{L^2(\mathbb S^4)}^2\ .
\end{eqnarray*}
\end{proof}

\subsection{Preservation of $L^4$ bounds}
Our goal is to apply Proposition \ref{goodballext} or \ref{badballext} iteratively to overlapping balls belonging to a grid, after which we smoothen the final modified connection in order to obtain smoothness near the $4$-skeleton given by the boundaries of the balls. To ensure convergence for the final smoothing a crucial point is the creation and preservation of $L^4_{loc}$ bounds on the approximating connection forms. We encode this in a proposition, which uses two simple lemmas from Appendix \ref{ch:slice}:
\begin{proposition}
\label{preservel4}
Let $A,F$ be as in Proposition \ref{goodballext} (resp. \ref{badballext}) and assume $\hat A,\hat g$ are produced as described in the proof of Proposition \ref{goodballext} (resp. \ref{badballext}). Then the following hold:
\begin{enumerate}
 \item In the gauge $\hat g$ the form $\hat A$ has $L^4$ bounds:
\begin{equation}
 \label{GA}
 \|\hat A^{\hat g}\|_{L^4(\mathbb B^5)}\lesssim\|F\|_{L^2(\mathbb S^4)}+\|A\|_{L^2(\mathbb S^4)}
\end{equation}
and for $x\in\mathbb R^5\setminus \mathbb B^5, r>1/2$ such that the trace $i^*_{\partial B(x,r)}\hat A^{\hat g}$ is well defined in $L^4$ over $\partial B(x,r)\cap \mathbb B^5$,
  \begin{equation}
   \label{GAA}
   \|i^*_{\partial B(x,r)}\hat A^{\hat g}\|_{L^4(\partial B(x,r)\cap \mathbb B^5)}\lesssim\|F\|_{L^2(\mathbb S^4)}+\|A\|_{L^2(\mathbb S^4)}\ .
  \end{equation}
 \item Assume $U\subset \mathbb B^5_2$ is open and $k:U\to G$ is such that $A^k\in L^4(U,\wedge^1\mathbb R^5\otimes\mathfrak g)$ and $i^*_{\mathbb S^4}A^k$ is well defined and $L^4$ over $U\cap \mathbb S^4$. Then there exists $h:U\to G$ such that $\hat A^h\in L^4(U,\wedge^1\mathbb R^5\otimes\mathfrak g)$ as well and has the bounds
\begin{equation}
 \label{L4A}
 \|\hat A^h\|_{L^4(U)}\lesssim \|A^k\|_{L^4(U)}+\|A^k\|_{L^4(U\cap \mathbb S^4)}+\|F\|_{L^2(\mathbb S^4)}+\|A\|_{L^2(\mathbb S^4)}\ .
\end{equation}
\end{enumerate}
\end{proposition}

\begin{proof}
The proof is based on properties of the harmonic extension. Note that due to the definition of $\hat g$ as in the proof of Proposition \ref{goodballext} (which is the same as used also in the homologous part of the proof of Proposition \ref{badballext}), $\hat A^{\hat g}$ is exactly the extension $\tilde A^g$ of the $W^{1,2}$-form $A^g$. The estimate \eqref{GA} is thus already proved in both cases in \eqref{A-19} (which is a consequence of \eqref{A-13}).\\

 To prove \eqref{GAA} consider $\Omega_1=\mathbb B^5\cap B(x,r), \Omega_2=\mathbb B^5\setminus B(x,r)$ and define $\eta_i=i^*_{\mathbb S^4}\tilde A^g$ on $\partial\Omega_i$. The extensions as in \eqref{A-10}, \eqref{eta} with $\Omega_i$ in place of $\mathbb B^5$ then coincide with $\tilde A^g$, thus we deduce by Lemma \ref{bddconst} and Lemma \ref{bdddeform} below and by \eqref{A-13} and \eqref{A-1*} that
 \begin{eqnarray*}
  \|i^*_{\partial B(x,r)}\tilde A^g\|_{L^4(\partial B(x,r)\cap \mathbb B^5)}&\leq&\|i^*_{\partial\Omega_i}\tilde A^g\|_{L^4(\partial\Omega_i)}\lesssim \|\tilde A^g\|_{W^{3/2,2}(\Omega_i)}\\
  &\leq&\|\tilde A^g\|_{W^{3/2,2}(\mathbb B^5)}\lesssim\|A^g\|_{W^{1,2}(\mathbb S^4)}\lesssim\|F\|_{L^2(\mathbb S^4)}+\|A\|_{L^2(\mathbb S^4)}\ .
 \end{eqnarray*}
The fact that in the second above inequality we have a geometric constant independent of $B(x,r)$ follows because for $r>1/2, x\notin \mathbb B^5$ the largest domain $\Omega_i$ is deformable to $\mathbb B^5$ via a diffeomorphism $\phi$ such that $\phi, \phi^{-1}$ have uniform bounds in $W^{1,\infty}\cap W^{2,\frac{5}{3}}$, as stated in Lemma \ref{bdddeform}.\\

The control \eqref{L4A} follows from the formula expressing $i^*_{\partial B(x,r)}A^{\tilde g}$ in terms of $i^*_{\partial B(x,r)}A^k$ for $h=\tilde g k^{-1}$:
\[
 dh=h i^*A^{\tilde g} - i^*A^kh\ .
\]
We obtain indeed $\|dh\|_{L^4}\lesssim \|A^{\tilde g}\|_{L^4}+\|A^k\|_{L^4}$, which allows to conclude by using the bound analogous to \eqref{GAA} assumed for $A^k$.
\end{proof}

\subsection{Smoothing}
for the smoothing of our connection forms we will use the following classical result:
\begin{lemma}\label{4dapprox}
Assume that $A$ is a connection form over a $n$-manifold $X$ which in local gauges is $L^4$ and has distributional exterior derivative $dA$ in $L^2$. Let $K$ be a (possibly empty) compact set on which $A$ is smooth. Then there exists a sequence $A_\eta$ of smooth connections over $X$ such that $A_\eta|_K=A|_K$ and
\[
A_\eta \stackrel{L^4_{loc}}{\to} A\ \text{ and }F_{A_\eta}  \stackrel{L^2_{loc}}{\to}F_A\ \text{ in local gauges as }\eta\to 0.
\]
\end{lemma}
\begin{proof}
 If we had just functions $f, f_\eta:X\to \wedge^1\mathbb R^n\otimes \mathfrak g$ in our statement, then the result would be classical (even without the restriction on $p$) and it would suffice to mollify $f$ in order to obtain approximants $f_\eta=f*\rho_\eta$ where $\rho_\eta$ is a scale-$\eta$ smooth mollifier.\\
 The problem which we face is just the fact that $A$ is not globally defined: we have instead local expressions $A_i$ in the chart $U_i$, and we must mollify $A_i$ to $A_{i,\eta}$ for which $A_{i,\eta}=g_{ij}^{-1} dg_{ij} + g_{ij}^{-1} A_{j,\eta}g_{ij}:=g_{ij}(A_{j,\eta})$ are still true. We use a partition of unity $(\theta_i)_i$ adapted to the charts $U_i$ and define $\rho_\eta(x)=\eta_x^{-n}\rho(x/\eta_x)$, where $\eta_x:=\min\{\eta,\op{dist}(x,K)/2\}$. Then we define
 \[
  (A_\eta)_i=\theta_i A_i*\rho_\eta + \sum_{i'\neq i} \theta_{i'}g_{ii'}(A_{i'}*\rho_\eta)\ .
 \]
 By the cocycle condition $g_{ii'}g_{i'j}=g_{ij}$ we obtain the desired $(A_\eta)_i=g_{ij}((A_\eta)_j)$. The derivatives of $\theta_i$ enter the estimate of $\|dA_{i,\eta} - dA_i\|_{L^2(U_i)}$ introducing a possibly huge $L^\infty$ factor, however this factor is independent on $\eta$. We therefore have $\lim_{\eta\to 0}\sum_i\|dA_{i,\eta}- dA_i\|_{L^2(U_i)}=0$.\\
We now prove the convergence of curvatures. This is based on the following inequality:
 \begin{eqnarray*}
  \|F_A-F_B\|_{L^2}&\lesssim&\|dA - dB\|_{L^2} +\|(A-B)\wedge A\|_{L^2} + \|(A-B)\wedge B\|_{L^2}\\
  &\lesssim&\|d A - d B\|_{L^2} + \|A-B\|_{L^4}(\|A\|_{L^4} +\|B\|_{L^4})\ .
 \end{eqnarray*}
We are able to conclude using the convergence of the $A_{i,\eta}, dA_{i,\eta}$ in local gauges.
\end{proof}

\subsection{Good grids and good balls} 
In order to detect the regions where to apply the approximation step of the previous section we construct controlled families of balls which depend on $F$ and on $A$ and are used for the approximation.
\subsubsection{Good grids}
We thus define our basic object:
\begin{definition}\label{ggrid}
 Assume that $\Lambda\subset \mathbb R^5$ is a discrete set and $1<\alpha<2$ is a constant such that the balls $B_1(p), p\in\Lambda$ cover $\mathbb R^5$ and for each $p\in\Lambda$ the only ball of the form $B_{\alpha}(q),q\in\Lambda$ covering $p$ is the one with $q=p$. Fix a scale $r>0$. A collection of balls $B_i=B_{r_i}(x_i)$ with $r_i\in\left[r, \alpha r\right]$ and $\{x_i\}=r\Lambda\cap \mathbb B^5$ will be called a \emph{grid of balls of scale $r$}.
\end{definition}
$\Lambda, \alpha\in]1,2[$ as above can be found, e.g. we may take $\Lambda$ to be a body-centered cubic lattice:
\[
 \Lambda=\beta^{-1}\left[2\mathbb Z^5 \cup ((1,\ldots,1)+2\mathbb Z^5)\right],\quad\alpha\in]1,2/\beta[,\quad \beta\in]\sqrt 5 /2,2[\ .
\]
$\alpha,\Lambda$ will be fixed from now on; their only role is to ensure that for any choice of $r_i$ in the allowed the balls of the grid cover $\mathbb B^5$. We can choose the $r_i$ above such that a good control on the boundary of our grids is available:

\begin{proposition}\label{centergridball}
Let $F\in L^2(\mathbb B^5,\wedge^2\mathbb R^5\otimes\mathfrak g)$ and $A\in L^2(\mathbb B^5,\wedge^1\mathbb R^5\otimes\mathfrak g)$. For each fixed scale $r>0$ pick the finitely many radii $r_i\in \left[r, \alpha r\right]$ uniformly and independently at random.\\

There exist a constant $C$ depending only on the dimension and a modulus of continuity $o(r)$ depending only on $F$ such that at fixed $r$ the following hold with positive probability:
 \begin{equation}
  \label{boundaryenb}
  r\sum_i\int_{\partial B_i}|F|^2\leq C \int_{\mathbb B^5}|F|^2\ ,
 \end{equation}
 \begin{equation}
  \label{boundaryconnb}
  r\sum_i\int_{\partial B_i}|A|^2\leq C \int_{\mathbb B^5}|A|^2
 \end{equation}
and, with the notation $\overline F_i:=\frac{1}{|B_{\alpha r}|}\int_{B_{\alpha r}(x_i)}F$,
 \begin{equation}
  \label{boundaryavenb}
  r\sum_i\int_{\partial B_i}|F-\overline F_i|^2\leq o(r)\ ,
 \end{equation} \begin{equation}
  \label{boundaryavconnb}
  r\sum_i\int_{\partial B_i}|A-\overline A_i|^2\leq o(r)\ .
 \end{equation}
 
\end{proposition}
\begin{proof} 
Since the annuli $B_{\alpha  r}(x_i)\setminus B_{ r}(x_i)$ can be divided into $N$ families having no overlaps we obtain
\[
 \int_{ r}^{ \alpha r}\left(\sum_i \int_{\partial B_\rho(x_i)}|F|^2\right)d\rho\lesssim\|F\|_{L^2(\mathbb B^5)}^2\ ,
\]
therefore for randomly picked $r_i\in[ r,\alpha  r]$ 
\[
r\sum_i\int_{\partial B_{r_i}(x_i)}|F|^2\lesssim\|F\|_{L^2(\mathbb B^5)}^2
\]
with probability $\geq 1-X$, where $C$ depends on $X$, which in turn will be fixed later. This will give \eqref{boundaryavenb}, \eqref{boundaryavconnb}. The same reasoning can be applied also to $A$ and we obtain that uniformly chosen $\rho\in[ r,2 r]$ satisfies a \eqref{boundaryconnb} with probability $\geq 1-X$.\\

Fix now smooth approximants $G^k$ to $F$ as a function in $L^2(\mathbb B^5,\Lambda^2\mathbb R^2\otimes \mathfrak g)$: assume that 
\[
 \int_{\mathbb B^5} |G^k-F|^2\leq \frac{1}{k}\ .
\]
Take $o_\infty(r)=\min_k o_k(r)$ for $o_k(r):=\frac{1}{k} +r^2\|G^k\|_{C^1}$. For $r$ such that $o_\infty(r)=o_k(r)$ we apply the above argument to $G^k-F$ and obtain 
\[
 r\sum_i\int_{\partial B_{r_i}(x_i)} |G^k-F|^2\lesssim\int_{\mathbb B^5}|G^k-F|^2
\]
with probability $\geq 1-X$. Let $\bar G_i^k:=\frac{1}{|B_{\alpha r}|}\int_{B_{\alpha r}(x_i)}G^k$. By a straightforward computation and by Jensen's inequality we have, independently of $r$,
\begin{eqnarray*}
 r\sum_i\int_{\partial B_{r_i}(x_i)}|\bar G_i^k - \bar F|^2&\lesssim& \sum_i\int_{B_{\alpha r}(x_i)}|\bar G_i^k - \bar F_i|^2\\[3mm]
&\lesssim&\sum_i\int_{B_{\alpha r}(x_i)}|G^k-F|^2\\[3mm]
&\lesssim&\frac{1}{k}\ .
\end{eqnarray*}
We then estimate by triangle inequality between $F,\bar F,\bar G_k, G_k$
\[
 r\sum_i\int_{\partial B_{r_i}(x_i)}|F-\bar F_i|^2 \lesssim \frac{1}{k}+r\sum_i\int_{\partial B_{r_i}(x_i)}|G^k-\bar G^k_i|^2\lesssim o_\infty(r)\ .
\]
 This shows \eqref{boundaryavenb} once we take $o(r)=C\ o_\infty(r)$. We proceed similarly to obtain also \eqref{boundaryavconnb} with probability higher than $X$. For each $r$ each one of the events \eqref{boundaryenb}, \eqref{boundaryconnb}, \eqref{boundaryavenb}, \eqref{boundaryavconnb} fails with probability $\leq X$ thus their intersection fails with probability $\leq 4X$. We thus choose $X>1/4$ and conclude the proof.
\end{proof}

The conditions obtained via Proposition \ref{centergridball} are contemporarily valid for a positive probability on uniformly chosen radii, thus the new condition of having a $W^{1,2}$ representative of the connection class on each $\partial B_\rho(x_i)$ keeps them valid too.\\
\subsubsection{Good grids for Morrey curvatures}
We denote $\|\cdot\|_M$ the following Morrey norm:
\[
 \|f\|_M^2:=\sup_{x,r}\frac{1}{r}\int_{B_r(x)}|f(y)|^2dy\ .
\]
We next extend the statement of Proposition \ref{centergridball} to a situation where we have a Morrey control on $F$:
\begin{proposition}[extension of Prop. \ref{centergridball}]\label{cgbmorey}
Consider a grid as in Definition \ref{ggrid}. Let $F\in L^2(\mathbb B^5,\wedge^2\mathbb R^5\otimes\mathfrak g)$ and $A\in L^2(\mathbb B^5,\wedge^1\mathbb R^5\otimes\mathfrak g)$. For each fixed scale $r>0$ pick the finitely many radii $r_i\in \left[r, \alpha r\right]$ uniformly and independently at random.\\

There exist a constant $C$ depending only on the dimension and a modulus of continuity $o(r)$ depending only on $F$ such that at fixed $r$ we have \eqref{boundaryavenb}, \eqref{boundaryavconnb} and the following, with positive probability:
\begin{equation}
 \label{morrecontrol}
  \int_{\partial B_i}|F|^2\leq C \frac{1}{r_i}\int_{B_i}|F|^2\quad\text{ for all }i
\end{equation}
and
\begin{equation}
 \label{morrccontrol}
  \int_{\partial B_i}|A|^2\leq C \frac{1}{r_i}\int_{B_i}|A|^2\quad\text{ for all }i\ .
\end{equation}
\end{proposition}
\begin{rmk}\label{morrcontrol}
 In particular if $\|F\|_M^2<\infty$ then we directly obtain from \eqref{morrecontrol} that
$\|F\|_{L^2(\partial B_i)}^2\leq C \|F\|_M^2$.
\end{rmk}
\begin{proof}
We note that in the end of the proof of Proposition \ref{centergridball} we had obtained that the estimates \eqref{boundaryavenb} and \eqref{boundaryavconnb} hold contemporarily with probability at least $1-2X$. In other words the estimates hold once we choose $r_k/r\in I_k\subset[1,\alpha]$ and $\prod_k|I_k|>1-2X$. In particular all of the $I_k$ satisfy 
\begin{equation}
\label{morr2}
 1\geq|I_k|\geq 1-2X\ .
\end{equation}
We then obtain by Chebychev's inequality that 
\begin{equation}
 \label{morr3}
 |Y_{C,k}|:=\left|\left\{\rho:\ \int_{\partial B_{\rho}(x_k)}|F|^2>\frac{C}{\alpha r}\int_{B_{\alpha r}(x_k)}|F|^2\right\}\right|\leq \frac{\alpha r}{C}
\end{equation}
by recalling that $\alpha$ is bounded from above depending only on the dimension and using \eqref{morr2} we see that there exists a choice
\[
 C\sim\frac{1}{1-2X}
\]
which will ensure that for each $k$ there holds $|Y_{C,k}|\leq |I_k|r/2$. Since the number of balls is finite, with positive probability for each $k$ we have \eqref{boundaryavenb}, \eqref{boundaryavconnb} and
\[
 \int_{\partial B_{\rho}(x_k)}|F|^2\leq\frac{C}{\alpha r}\int_{B_{\alpha r}(x_k)}|F|^2\ ,
\]
which implies \eqref{morrecontrol}. We may similarly ensure \eqref{morrccontrol} as well, up to increasing $C$ by a controlled factor.
\end{proof}

\subsubsection{Good and bad balls}\label{gbballs}
We intend to apply Proposition \ref{goodballext} to $B_i$ belonging to grids as in Proposition \ref{centergridball}, for $F,A$ as in the definition of $\mathcal A_{G}(\mathbb B^5)$ and for $\bar F=\bar F_i$ on $B_i$ with the notations of Proposition \ref{centergridball}. In this situation (rescaled versions of) the estimates of Proposition \ref{goodballext} are valid for all but few ``good'' balls. We start by fixing the definition of ``good'' and ``bad'':
\begin{lemmad}\label{manygoods}
 Fix a constant $\delta>0$ and a scale $r>0$. Let $A, F, B_i, o(r)$ be as in Proposition \ref{centergridball}. We say that $B_i$ is a \emph{$\delta$-good ball with respect to $A,F,o(r)$} if the following bounds hold:
\begin{equation}\label{g1}
  \int_{\partial B_i}|F|^2\leq \delta\ ,
 \end{equation}
  \begin{equation}\label{g2}
  \frac{1}{r^2}\int_{\partial B_i}|A|^2\leq \delta\ ,
 \end{equation}
  \begin{equation}\label{g3}
  \frac{1}{r^2}\int_{\partial B_i}|F-\overline F_i|^2\leq o(r)\ ,
 \end{equation}
   \begin{equation}\label{g4}
  \frac{1}{r^2}\int_{\partial B_i}|A-\overline A_i|^2\leq o(r)\ .
 \end{equation}
In this case we will denote $\mathcal G_r$ the set of good balls and $\mathcal B_r$ the set of the remaining (so-called ``bad'') balls of scale $r$.\\

The cardinality of $\mathcal B_r$ can then be estimated as follows:
\[
 \#\mathcal B_r\lesssim \frac{\|F\|_{L^2(\mathbb B^5)}}{\delta r} +\frac{\|A\|_{L^2(\mathbb B^5)}}{\delta r^3} +\frac{1}{r}\ .
\]
In particular the total volume of the bad balls vanishes as $r\to 0$.
\end{lemmad}

\begin{proof}
 The second statement follows from the first because the volume of each bad ball is $\sim r^5$. To prove the estimate on $\#\mathcal B_r$ we separately estimate the sets $\mathcal B_i$ of cubes for which ($gi$) fails.\\
 Using Proposition \ref{centergridball} we then obtain 
\begin{eqnarray*}
 \delta \# \mathcal B_1&\lesssim&\sum_{B_i\in \mathcal B_1}\int_{\partial B_i}|F|^2\lesssim\frac{1}{r}\int_{\mathbb B^5}|F|^2\ ,\\
 \delta r^2\# \mathcal B_2&\lesssim&\sum_{B_i\in \mathcal B_2}\int_{\partial B_i}|A|^2\leq \frac{1}{r}\int_{\mathbb B^5}|A|^2\ ,\\
o(r)\# \mathcal B_3&\lesssim&\sum_{B_i\in \mathcal B_3}\int_{\partial B_i}|F-\overline F_i|^2\leq \frac{o(r)}{r}\ ,\\
o(r)\# \mathcal B_4&\lesssim&\sum_{B_i\in \mathcal B_4}\int_{\partial B_i}|A-\overline A_i|^2\leq \frac{o(r)}{r}\ .
\end{eqnarray*}
Since $\mathcal B=\cup_{i=1}^4\mathcal B_i$ we obtain the desired result.
\end{proof}

Going back to the $r$ scale by pull backing all forms to the good ball $C_r^i$ using the dilation map $x\rightarrow r^{-1} x$ ,
denoting $\hat{A}_r=r^{-1}\sum_{j=1}^5 \hat{A}_j(r^{-1} x)\ dx_j$,
\[
\begin{array}{l}
\int_{C_r^i}|d\hat{A}_r+\hat{A}_r\wedge\hat{A}_r-\overline F|^2\ dx^5\le \, C\ \delta\ \int_{C_r^i}|\overline{F}|^2\ dx^5 + \\[3mm]
\quad\quad+\,C\,r\,\int_{\partial C_r^i}|F-i_{\partial C_r^i}^\ast\overline{F}|^2\ dvol_{\partial C_r^i}+\,C\ r \,\delta\,
\int_{\partial C_r^i}|F|^2\  dvol_{\partial C_r^i}\ .
\end{array}
\]
Summing up over the good balls - index i - using (\ref{boundaryenb}) and (\ref{boundaryavenb}) we finally obtain the desired estimate
\[
\sum_{i\in{\mathcal G_r}}    \int_{C_r^i}|d\hat{A}_r+\hat{A}_r\wedge\hat{A}_r-\overline F|^2\ dx^5\le C\ \delta+o_r(1)\ .
\]

\subsubsection{Good balls in the Morrey case}\label{gbm}
We now provide a version of the previous results useful for the approximation with bounds on Morrey norms. The relevant new feature is that there exists a constant $\epsilon_1$ depending only on the underlying manifold (in our case $\mathbb B^5$) such that when the Morrey norm of $F$ satisfies
\begin{equation}
\label{condgoodmorrey}
 \|F\|_{M(\mathbb B^5)}^2\leq\epsilon_1\ ,
\end{equation}
from Remark \ref{morrcontrol} we automatically have the condition
\[
 \int_{\mathbb S^4}|F|^2< \epsilon_0\ .
\]
In this case we will nevertheless fix $\delta>0$ much smaller than $\epsilon_0$, depending on $r$. The gain of the Morrey bound will be that under condition \eqref{condgoodmorrey} are able to apply Proposition \ref{badballext} in order to perform a controlled smooth extension on $\delta$-bad balls.

\subsection{Proof of Theorem \ref{naturality}}
We are going to prove the following result:
\begin{theorem}\label{approxstrong2}
Let $F$ be the distributional curvature corresponding to an $L^2$ connection form $A$ with $[A]\in \mathcal A_{G}^\phi(\mathbb B^5)$. Then there exist $F_n\in\mathcal R^\infty_\phi(\mathbb B^5)$ such that
\[
 \|F-F_n\|_{L^2(\mathbb B^5)}\to0,\quad\text{ as }n\to 0\ .
\]
Moreover we can also insure at the same time
\[
 \|A-A_n\|_{L^2(\mathbb B^5)}\to0,\quad\text{ as }n\to 0\ .
\]
\end{theorem}
\begin{proof}
The proof consists in giving an ``approximation algorithm'' for $F$, which is divided into several steps. After each step the approximant connection obtained at that point will be denoted by $\hat A$, therefore this notation represents different connection forms at different steps of the approximation.
\subsubsection*{Step 1}
Start with $F, A$ as in the definition of $\mathcal A_{G}(\mathbb B^5)$ and fix $r>0$. Apply Proposition \ref{centergridball} and choose well behaved radii $r_i$ such that \eqref{boundaryenb}, \eqref{boundaryconnb} and \eqref{boundaryavenb} hold. We may also assume that $i^*_{\partial B_i}A\in\mathcal A_G(\partial B_i)$ for each $i$, as remarked immediately after Proposition \ref{centergridball}.
\subsubsection*{Step 2}
Apply Lemma-Definition \ref{manygoods} and define the families $\mathcal G_r, \mathcal B_r$ with respect to the data from Step 1 and for a small constant $\delta>0$ to be fixed later.\\
The family $\mathcal G_r$ can be partitioned into subfamilies of disjoint balls $\mathcal G^1,\ldots,\mathcal G^N$, where $N$ depends only on the discrete set $\Lambda$ and on the constant $\alpha$ fixed in Definition \ref{ggrid}.
\subsubsection*{Step 3}
Fix $B_i=B(x_i,r_i)\in \mathcal G^1$. Let $(i^*_{\partial B_i}A)_{g_{B_i}}\in\mathcal A_G(\partial B_i)$, as in the definition of $\mathcal A_G(\partial B_i)$. Define then $A_{B_i}:=\tau_{B_i}^*A, F_{B_i}:=\tau_{B_i}^*F$, where $\tau:\mathbb B^5\to B_i$ is the homothety $\tau(x)=x_i+r_ix$. From the estimates \eqref{g1}, \eqref{g2} we obtain 
\[
 \int_{\mathbb S^4}|F_{B_i}|^2<\delta,\quad \int_{\mathbb S^4}|A_{B_i}|^2<\delta\ .
\]
We require $\delta$ to be smaller than the constant $\epsilon_0$ of Proposition \ref{goodballext}. Combining with \eqref{g4} and requiring $r$ to be sufficiently small, we also obtain 
\[
 |\bar A_i|^2<\epsilon_0\ .
\]
We may thus apply Proposition \ref{goodballext} to $A=A_{B_i}, F=F_{B_i}, \bar F=\bar F_i,\bar A=\bar A_i $. We then pull back the approximants to $B_i$ via $\tau_{B_i}^{-1}$ and we denote the resulting approximant connection by $\hat A$. The error estimate \eqref{AA} of Proposition \ref{goodballext} becomes:
\[
\|d\hat{A}+\hat{A}\wedge\hat{A}-\bar F_i\|_{L^2(B_i)}^2\lesssim \, \delta \|\bar F_i\|^2_{L^2(B_i)}+\delta r\|F\|_{L^2(\partial B_i)}^2+r\|F-i_{\partial B_i}^\ast\bar F_i\|_{L^2(\partial B_i)}^2\ .
\]

\subsubsection*{Step 4: iteration}
Iterate Step 3 for all $B_i\in\mathcal G^1$. Since such balls are disjoint, the local replacements of $A, F$ by $\hat A, F_{\hat A}$ are done independently. The total error that we obtain at the end is, using the estimates of Proposition \ref{centergridball},
\begin{eqnarray*}
 \|F_{\hat{A}}-F\|_{L^2(\mathbb B^5)}^2&\lesssim&\sum_{B_i\in\mathcal G^1}\|F-\bar F_i\|_{L^2(B_i)}^2+\delta \sum_{B_i\in\mathcal G^1}\|\bar F_i\|^2_{L^2(B_i)}+\\[3mm]
 &+&\delta r\sum_{B_i\in\mathcal G^1}\|F\|_{L^2(\partial B_i)}^2+r\sum_{B_i\in\mathcal G^1}\|F-i_{\partial B_i}^\ast\bar F_i\|_{L^2(\partial B_i)}^2\\[3mm]
 &\lesssim&\delta \|F\|_{L^2(\mathbb B^5)} +o(r) +\sum_{B_i\in\mathcal G^1}\|F-\bar F_i\|_{L^2(B_i)}^2\ .
\end{eqnarray*}
Note that in particular the total $L^2$-error of averages satisfies
\[
e_1:=\sum_i|B_i|\left|\frac{1}{|B_{2r}|}\int_{B(x_i, 2r)}F_{\hat A}-\frac{1}{|B_{2r}|}\int_{B(x_i, 2r)}F\right|^2\leq N \|F_{\hat{A}}-F\|_{L^2(\mathbb B^5)}^2\ .
\]

\subsubsection*{Step 5: iteration}
We iterate Step 4. More precisely, we start with $\hat A_0=A$ and at step $k\geq 1$ we use the balls from family $\mathcal G^k$ to approximate the curvature forms $F_{\hat A^{k-1}}$ obtained from step $k-1$. At step $k$ we use the constants
\[
\bar F_i^k:=\frac{1}{|B_{2r}|}\int_{B(x_i,2r)}F_{\hat A^{k-1}}\ .
\]
Denote the new error introduced on the averages by $e_k$, analogously as $e_1$ above. Note that each $B_i$ intersects a finite number of other balls (this number depends only on $\Lambda,\alpha$ from Definition \ref{ggrid}). Therefore the total error after the final step $k=N$ is
\begin{eqnarray*}
 \|F_{\hat{A}^N}-F\|_{L^2(\mathbb B^5)}^2&\lesssim&\sum_{k=1}^N\|F_{\hat{A}^k}-F_{\hat A^{k-1}}\|_{L^2(\mathbb B^5)}^2\\[3mm]
 &\lesssim&N\delta \|F\|_{L^2(\mathbb B^5)} +No(r) +\sum_{k=1}^N e_k\\[3mm]
 &\lesssim&C(N)\left(\delta \|F\|_{L^2(\mathbb B^5)} + o(r) +\sum_i\|F-\bar F_i\|_{L^2(B_i)}^2\right)\ ,
\end{eqnarray*}
where the last sum is taken over all the balls $B_i$ of our grid and $C(N)$ depends just on $\Lambda,\alpha$ from Definition \ref{ggrid}. Since for any $L^2$ function $f$ there holds
\[
 \lim_{|h|\to 0}\int|f(x+h) - f(x)|^2dx=0
\]
we deduce that 
\[
 \sum_i\|F-\bar F_i\|_{L^2(B_i)}^2=o'(r)\to 0\quad \text{ as }r\to 0
\]
as well. Thus we have the following final estimate on our approximation:
\[
 \|F_{\hat{A}^N}-F\|_{L^2(\mathbb B^5)}^2\lesssim \delta \|F\|_{L^2(\mathbb B^5)} + o(r) +o'(r)\ .
\]
Note that as a result of Proposition \ref{goodballext} we also have that $\hat{A}^N$ is continuous on the interior of $\cup\{B_i:\:B_i\in\mathcal G_r\}$.

\subsubsection*{Step 5': $L^4_{loc}$ control on good balls}
In order to apply the smoothing as in Lemma \ref{4dapprox} we have to ensure that locally there exist gauges such that $A$ is bounded in $L^4$. The bound on $dA$ will then follow from the formula $F=dA+A\wedge A$ since the norm of $F$ does not depend on the gauge. The $L^4$ bound on $\cup\mathcal G_r$ follows from Proposition \ref{preservel4}. We may cover such set by open sets $U_i$ such that if we rescale any ball $B\in\mathcal G_r$ to scale $1$ the $U_i$ which intersected it map to sets $U$ as in Prop. \ref{preservel4}. We may further require that each $U_i$ is included in some good ball. By the first part of point 1 of Prop. \ref{preservel4} at the first iteration of Step 5 where a good ball $B_{1,i}$ containing a given $U_i$ is modified, we obtain a gauge $g_1:U_i\to G$ such that after scaling:
\[
 r^{-1/4}\|\hat A^{g_1}\|_{L^4(U_i)}\lesssim r^{-1}\|A\|_{L^2(\partial B_{1,i})}+\|F\|_{L^2(\partial B_{1,i})}\ .
\]
If we have to modify $A|_{U_i}$ at later iterations when modifying $A$ on balls $B_{2,i},\ldots,B_{m,i}$ then at step $k$ we use the gauge $g_{k-1}$ and apply point 2 and the second part of point 1 of Prop. \ref{preservel4}. At the end of the changes of Step 5 we always have that there exist some gauge $h_m:U_i\to G$ such that after scaling there holds
\[
 r^{-1/4}\|\hat A^{h_m}\|_{L^4(U_i)}\lesssim \sum_{k=1}^m\left(r^{-1}\|A\|_{L^2(\partial B_{k,i})}+\|F\|_{L^2(\partial B_{k,i})}\right)\ .
\]
\subsubsection*{Step 6}
Divide the bad balls in $N$ disjointed families $\mathcal B_1,\ldots,\mathcal B_N$ as for the good balls and consider the first family $\mathcal B_1$. We extend $\hat A$ on a each ball $B_j\in \mathcal B_1$ as follows. First apply Lemma \ref{4dapprox} to mollify $\hat A$ on $\left(\cup \mathcal B_1\right)\cap\left(\cup\mathcal G_r\right)$ using the bounds from Step 5' and the fact that we have a finite number of $U_i$. Then we smooth on $\partial B_j$ while keeping $A$ unchanged on $\partial B_j\cap\left(\cup\mathcal G_r\right)$ for each $B_j\in\mathcal B_1$. This time the hypotheses of Lemma \ref{4dapprox} are met by the assumption that $i^*_{\partial B_j}A$ is in $W^{1,2}$ up to gauge, and by the Sobolev embedding $W^{1,2}\to L^4$ in $4$-dimensions.\\
We obtain $A_\eta, F_{A_\eta}$ which approximate $A, F$ on $\left(\cup\mathcal G_r\right)\cup \left(\cup_{\mathcal B_1}\partial B_j\right)$ such that $A_\eta$ is smooth. Then for each $B_j\in \mathcal B_1$ we use the radial projection $\pi_j:B_j\setminus\{x_j\}\to \partial B_j$ and define $\hat A_j:=\pi_j^*A_\eta$. We have the following estimate, using Step 5:
\begin{eqnarray*}
 \|F_{\hat A_j}\|_{L^2(B_j)}^2&\lesssim& r\left(\|F_{\hat A_j} - F_{\hat A}\|_{L^2(\partial B_j)}^2 + \|F_{\hat A}\|_{L^2(\partial B_j)}^2\right)\\
&\lesssim&r(o_\eta + \|F_{\hat A}-\bar F_j\|_{L^2(\partial B_j)}^2) +\|F\|_{L^2(B_j)}^2\ .
\end{eqnarray*}
\subsubsection*{Step 7: iteration}
We iterate Step 6 for all families $\mathcal B_1,\ldots,\mathcal B_N$. Since we modify at most $N$ times the connection on each ball, the final bound for the connection $\hat A$ obtained after this process is still
\[
 \sum_{B_j\in\mathcal B_r}\|F_{\hat A}\|_{L^2(B_j)}^2\lesssim ro_\eta + o(r) +\|\bar F\|_{L^2(\cup \mathcal B_r)}^2\ .
\]
The total error which we obtain is as follows:
\begin{eqnarray*}
  \|F_{\hat A} - F\|_{L^2(\mathbb B^5)}^2&\lesssim&\sum_{B_i\in \mathcal G_r}\|F_{\hat A} - F\|_{L^2(B_i)}^2 + \sum_{B_j\in \mathcal B_r}\|F_{\hat A} - F\|_{L^2(B_j)}^2\\
&\lesssim&\delta\|F\|_{L^2(\mathbb B^5)} + o(r)+o'(r)+ ro_\eta+ o(r) +\|F\|_{L^2(\cup \mathcal B_r)}^2\ .
\end{eqnarray*}
For $r,\delta,\eta$ small enough the first terms become as small as desired. The last term converges to zero by dominated convergence: indeed $|\cup\mathcal B_r|\to 0$ as $r\to 0$ by Lemma \ref{manygoods} and the function $\chi_{\cup B_r}F$ is dominated by $F\in L^2$.
\subsubsection*{Step 8}
From the previous step we have $\hat A$ such that $\|F_{\hat A} - F\|_{L^2(\mathbb B^5)}\leq \frac{1}{2k}$ and $\hat A$ is $C^0$ outside the centers of bad balls by construction (see Step 3 and Step 6, and recall that by Definition \ref{ggrid} the ball $B_j\subset B_{\alpha r}(x_j)$ does not cover $x_i$ for $j\neq i$). We now mollify $\hat A$ outside this finite set of centers, and we obtain the desired curvature $F_{A_k}\in\mathcal R^\infty$.\\

By a similar reasoning we also insure $\|A_n - A\|_{L^2(\mathbb B^5)}\to 0$ utilizing \eqref{AAA} instead of \eqref{AA} as above.\\

Utilizing the fact that the construction of Proposition \ref{goodballext} and the radial extension on the bad balls do not affect the boundary condition on our balls we obtain the approximation also in $\mathcal R^\infty_\phi(\mathbb B^5)$ for weak connections in $\mathcal A_{G}^\phi(\mathbb B^5)$.
\end{proof}

\subsection{Proof of Morrey approximation Theorem \ref{mapproxd}}
We now provide the modifications needed to prove the Theorem \ref{mapproxd} along the same steps as Theorem \ref{approxstrong2}.
\subsubsection{Strategy of $L^2$ approximation}
It is enough to prove that for each fixed $\epsilon>0$ we may find a smooth approximating curvature $\hat F$ which is closer than $\epsilon $ to $F$ in $L^2$-norm and satisfies \eqref{m2}. To do this, we use the division into good and bad cubes like in the previous section and the construction for $\hat F$ proceeds as in the proof of Theorem \ref{approxstrong2} with the following modifications:
\begin{itemize}
 \item In Step 1 we use Proposition \ref{cgbmorey} instead of Proposition \ref{centergridball}.
 \item In Step 2 we further partition also the family of $\delta$-bad balls $\mathcal B_r$ into disjointed subfamilies $\mathcal B_1,\ldots,\mathcal B_N$.
 \item In Step 3 we keep also track of the error estimate \eqref{AAA} of Proposition \ref{goodballext}, which reads:
\[
 \|\hat A - \bar A_i\|_{L^2(B_i)}\leq Cr\|A-\bar A_i\|_{L^2(\partial B_i)}\ .
\]
 \item The above estimate propagates through Step 4 where we obtain
\[
 \|\hat{A}-A\|_{L^2(\mathbb B^5)}^2 \lesssim \sum_{B_i\in\mathcal G^1}\|A - \bar A_i\|_{L^2(\partial B_i)}^2\ .
\]
 \item In Step 5 this and \eqref{boundaryavconnb} gives
\[
 \|\hat A^N - A\|_{L^2(\mathbb B^5)}^2\lesssim \sum_i\|A-\bar A_i\|_{L^2(\mathbb B^5)}^2=o'(r)\ .
\]
 \item In Step 5' we may use again Proposition \ref{preservel4} since it applies in the setting of Proposition \ref{badballext} used in Step 3 as well.
 
\item In Step 6 we still apply Lemma \ref{4dapprox} but we replace the radial extension by the application of Proposition \ref{badballext} to the groups of bad balls $\mathcal B_k$ constructed in Step 2. This is allowed by the hypothesis $\|F\|_M^2<\epsilon_0$ and by the discussion of Section \ref{gbm}. After this procedure on each bad ball $B_j$ we obtain the estimate
\[
 \|F_{\hat A}\|_{L^2(B_j)}^2\lesssim r(o_\eta + \|F\|_{L^2(\partial B_j)}^2)\ .
\]
We similarly have the estimate for $\hat A$:
\[
 \|\hat A\|_{L^2(B_j)}^2\lesssim r(o_\eta + \|A\|_{L^2(\partial B_j)}^2)\ .
\]
\item In Step 7 we then collect the contributions from all bad balls like in Steps 4-5. We use the properties stated in Proposition \ref{cgbmorey} to obtain
\begin{eqnarray*}
 \sum_{B_j\in\mathcal B_r}\|F_{\hat A}\|^2_{L^2(B_j)}&\lesssim&ro_\eta+o(r)+\|F\|_{L^2(\cup\mathcal B_r)}+\|A\|_{L^2(\cup\mathcal B_r)} \ ,\\
\sum_{B_j\in\mathcal B_r}\|\hat A\|^2_{L^2(B_j)}&\lesssim&ro_\eta+o(r)+\|F\|_{L^2(\cup\mathcal B_r)}+\|A\|_{L^2(\cup\mathcal B_r)} \ ,
\end{eqnarray*}
and by the same dominated convergence reasoning as in Step 7 of Theorem \ref{approx5d} we obtain \eqref{m1} and \eqref{m11}.
\item Step 8 proceeds exactly as in Theorem \ref{approx5d}.
\end{itemize}
We now prove the bounds \eqref{m2} for $\hat F$ constructed as above. We need to estimate 
\[
 \frac{1}{\rho}\int_{B_\rho(x)}|\hat F|^2
\]
 uniformly in $\rho,x$. We consider separately the cases $\rho\gtrsim r$ and $\rho\ll r$.

\subsubsection{The case $\rho\gtrsim r$} 
In this situation we simply estimate 
\[
 \int_{B_\rho(x)}|\hat F|^2\leq\sum_i\int_{B_\rho(x)\cap B_i}|\hat F|^2\leq\sum_{i:B_{\alpha r}(x_i)\cap B_\rho(x)\neq\emptyset}\int_{B_i}|\hat F|^2
\]
In this case we use the fact that the cover $\{B_i\}$ had the bounded intersection property, the fact that $\alpha$ is bounded and the fact that as a consequence of Prop. \ref{goodballext} or Prop. \ref{badballext} (depending on the balls involved), $\|\hat F\|_{L^2(B_i)}\lesssim\|F\|_{L^2(B_i)}$ thus
\[
\int_{B_\rho(x)}|\hat F|^2\lesssim\int_{B_{c\rho}(x)}|\hat F|^2\lesssim \int_{B_{c\rho}(x)}|F|^2\ .
\]
By definition of Morrey norm, we continue with
\[
\frac{1}{\rho}\int_{B_\rho(x)}|\hat F|^2\lesssim  \frac{1}{\rho}\int_{B_{c\rho}(x)}|F|^2\lesssim c\|F\|_M^2\ ,
\]
which finishes the proof.
\subsubsection{The case $\rho\ll r$} 
In this case we will use elliptic regularity for the proof. We note the following scale-invariant inequalities valid for the harmonic extensions:
\[
 \|d\tilde A^g\|_{L^{5/2}(B_{r_i})}^2\leq C\int_{\partial B_{r_i}}|dA^g|^2\ ,\quad \|\tilde A^g\|_{L^5(B_{r_i})}^4\leq C\int_{\partial B_{r_i}}|A^g|^4\ .
\]
If $B_\rho(x)\subset B_i$ then for an application of Step 3 or 6 on $B_i$ we can thus write:
\begin{eqnarray*}
 \|\hat F\|_{L^2(B_\rho(x)}^2&=&\int_{B_\rho(x)}|d\tilde A^g +\tilde A^g\wedge \tilde A^g|^2\\
&\lesssim&\int_{B_\rho(x)}|d\tilde A^g|^2 + \int_{B_\rho(x)}|\tilde A^g|^4\\
&\lesssim&|B_\rho|^{\frac{1}{5}}\left(\int_{B_\rho(x)}|d\tilde A^g|^{5/2}\right)^{\frac{4}{5}} + |B_\rho|^{\frac{1}{5}}\left(\int_{B_\rho(x)}|\tilde A^g|^5\right)^{\frac{4}{5}}\\
&\lesssim&\rho\left[\left(\int_{B_i}|d\tilde A^g|^{5/2}\right)^{\frac{4}{5}} + \left(\int_{B_i}|\tilde A^g|^5\right)^{\frac{4}{5}}\right]\\
&\lesssim&\rho\left(\int_{\partial B_{r_i}}|dA^g|^2 + \int_{\partial B_{r_i}}|A^g|^4\right)\\
&\lesssim&\rho(1+\epsilon_0)\|F\|_{L^2(\partial B_i)}^2\ ,
\end{eqnarray*}
where in the first equality we used the gauge-invariance of $\hat F$, making the gauge change $\hat g$ irrelevant, and in the last estimate we use the results of Propositions \ref{goodballext}, \eqref{badballext}.\\

The desired estimate then follows similarly to the case $\rho\gtrsim r$. In the general case $B_\rho(x)\cap B_i\neq\emptyset$ we have to just replace $B_\rho(x)$ by $B_\rho(x)\cap B_i$ and the same estimates work. We note that the number of steps of type 3 or 6 in which we modify $\hat F$ over $B_\rho(x)$ is bounded above by a constant $C(N)$ which ultimately depends only on the dimension. $\square$

\section{Weak closure for non-abelian curvatures in 5 dimensions}\label{ch:wclos}

\subsection{Ingredients for the proof of Theorem \ref{wclos}}
We describe here what enters the proof of Theorem \ref{wclos}, while making a parallel to the works \cite{AK} and \cite{HR1} on metric currents and scans, which present analogous definitions of weak objects as sets of slices ``connected'' via a compatibility condition based on an overlying integrable quantity (in our case this control comes from the curvature $2$-form $F$). Our closure result comes from the interplay of three ingredients:
\begin{itemize}
\item A \textit{geometric distance} on sliced $1$-forms: for $A,A'$ which are $L^2$ connection forms over $\mathbb S^4$ we use the gauge-orbit distance 
 \[
  \op{dist}([A],[A']):=\min\{\|A-g^{-1}dg - g^{-1}A'g\|_{L^2(\mathbb S^4)}:\:g\in W^{1,2}(\mathbb S^4, G)\}\ . 
 \]
This corresponds to the use of the flat distance for the closure theorem of integral currents by Ambrosio-Kirchheim \cite{AK}.
\item The fact that \textit{the above distance interacts well with our energy at the level of slices}, which follows from Theorem \ref{ymp4}. More precisely we have that sublevels of $A\mapsto \|F_A\|_{L^2(\mathbb S^4)}$ are dist-compact. In \cite{HR1} a similar interaction occurs between the flat distance and the fractional mass of rectifiable currents.
\item The \textit{oscillation control on slices} of a fixed weak curvature, obtained via the overlying $2$-form $F$. More precisely, if we identify $\mathbb S^4$ by homothety with each one of the spheres $S:=\partial B_t(x), S':=\partial B_{t'}(x')$ then the pullbacks $A(t,x), A(t',x')$ of $i^*_{S}A, i^*_{S'}A$ satisfy 
\[
\op{dist}([A(t,x)],[A(t',x')])\leq C\|F\|_{L^2(\mathbb B^5)}(|x-x'|+|t-t'|)^{1/2}\ .
\]
In \cite{AK} the corresponding fact is the interpretation of rectifiability as a bound of the metric variation of the slices.
\end{itemize}

We can find $L^2$-controlled connection forms $A_n$ corresponding to $F_n$ and obtain a weak limit $A$ which will be an $L^2$ connection form corresponding to $F$. The main difficulty is to find gauges $g$ in which the slices $i^*_{\partial B_r(x)}A$ become $W^{1,2}_{loc}$.\\

The above overall strategy is the one which worked in the abelian case $G=U(1)$ as well and was employed in \cite{PR1}.\\

We start by identifying the traces on lower dimensional sets $\partial B_\rho(x_0)$ with elements of a metric space $(\mathcal Y,\op{dist})$ where $\mathcal Y=\mathcal A_G(\mathbb S^4)/\sim$ and $\sim$ is the gauge-equivalence relation, such that we have a local control of the H\"older norm of the slice functions in terms of the $L^2$-norms of the $F_n$. We will use Proposition \ref{abstractthm} for this.\\

Mixing a compactness result for slice functions with respect to the distance on $\mathcal Y$ with the weak convergence of the $A_n$ we will manage to obtain the convergence of a.e. slice to an element which is gauge-equivalent to an element in $\mathcal A^g(\mathbb S^4)$ as desired.
\subsection{The metric space $\mathcal Y$}\label{spaceY}
To prove the weak closure result for $\mathcal A_{G}$ we use a slicing technique. In the definition of $\mathcal A_{G}$ we required that any weak connection have a gauge on each slice in which it is represented by a 
$W^{1,2}$ form. Therefore we consider the following space of possible slice classes:
\begin{equation}\label{Y}
 \mathcal Y:=\mathcal A_G(\mathbb S^4)/\sim,
\end{equation}
where the equivalence relation $\sim$ on global $L^2$ connections is 
\[
 A\sim B\text{ if }\exists g\in W^{1,2}(\mathbb S^4, G)\text{ s.t. }g^{-1}dg+g^{-1}Ag=B\ .
\]
We define the following gauge-invariant function:
\[
  \op{``dist''}(A, A'):=\left(\inf\left\{\int_{\mathbb S^4}|A-g^{-1}dg -g^{-1}A'g|^2:\:g\in W^{1,2}(\mathbb S^4,G)\right\}\right)^{\frac{1}{2}}\ .
\]
For two connection forms $A, A'$ if $g_A, g_{A'}$ are $W^{1,2}$ gauges such that 
\[
 B=g_A^{-1}dg_A + g_A^{-1}Ag_A,\quad B'= B=g_{A'}^{-1}dg_{A'} + g_{A'}^{-1}A'g_{A'}
\]
then, since $A\mapsto g^{-1}dg+g^{-1}Ag$ is a continuous group action of $W^{1,2}(\mathbb S^4,G)$ on $\mathcal A_G(\mathbb S^4)$, we have
\[
 \op{``dist''}(A, A')= \op{``dist''}(B, B')\ .
\]
$\op{``dist''}$ then descends to a well-defined distance $\op{dist}([A],[A'])$ on equivalence classes of connection forms.
Let
\[
 [A]=\text{ image of }A\text{ under the projection }\mathcal A_G(\mathbb S^4) \to \mathcal A_G(\mathbb S^4)/\sim\ .
\]
The natural metric to impose on $\mathcal Y$ is the $L^2$-distance between (global) gauge orbits (cfr \cite{DoKr}):
\begin{equation}\label{dslice}
 \op{dist}([A], [B])=\inf\left\{\|A'-B'\|_{L^2(\mathbb S^4)}:\:A'\in [A],\,B'\in[B]\right\}\ .
\end{equation}
On the metric space $(\mathcal Y,\op{dist})$ we will study the functional 
\begin{equation}\label{N}
 \mathcal N:\mathcal Y\to\mathbb R^+,\quad \mathcal N([A])=\int_{\mathbb S^4}|F_A|^2\ .
\end{equation}
Note that because the curvature satisfies $F_{g^{-1}dg + g^{-1}Ag}=g^{-1}F_Ag$ and since the norm on $2$-forms is $G$-invariant, we have that $\mathcal N([A])$ does not depend on the representative $A$ employed to compute $F_A$.
\subsection{The slice a.e. convergence}\label{sliceaeconv}
We employ the following abstract theorem. See \cite{HR1} Thm. 9.1 for the original inspiration. We use the notation overlapping with the previous section. The goal will be to justify this overlap in notation subsequently, by proving that the spaces and functions of Section \ref{spaceY} satisfy the hypotheses of the theorem.

\begin{proposition}\label{abstractthm}
 Consider a metric space $(\mathcal Y,\op{dist})$ on which a function $\mathcal N:\mathcal Y\to \mathbb R^+$ is defined. Suppose that the following hypothesis is met:
\begin{equation}\label{hypoth}
 \tag{$H$}
  \forall C>0\text{ the sublevels }\{\mathcal N\leq C\}\text{ are seq. compact in }\mathcal Y\ .
\end{equation}
Suppose $f_n:[0,1]\to \mathcal Y$ are measurable maps such that
 \begin{equation}\label{distcontrol}
 \op{dist}(f_n(t), f_n(t'))\leq C|t-t'|^{1/2}
 \end{equation}
and that
\begin{equation}\label{integraln}
 \sup_n \int_0^1\mathcal N(f_n(t))dt<C\ .
\end{equation}
Then $f_n$ have a subsequence which converges pointwise almost everywhere. The limiting function $f$ also satisfies
\[
 \op{dist}(f(t), f(t'))\leq C|t-t'|^{1/2},\quad  \int_0^1\mathcal N(f(t))dt<C\ .
\]
\end{proposition}

\begin{proof}
 By using the equicontinuity implied by \eqref{distcontrol}, we obtain that we may extract a subsequence of the $f_n$ labelled $n'$ such that for each $t\in[0,1]$ the sequence $f_{n'}(t)$ is converging to a limit $f(t)\in \overline{\mathcal Y}$, where $\overline{\mathcal Y}$ is the $\op{dist}$-completion of $\mathcal Y$. Therefore there exists a unique pointwise limit function $f:[0,1]\to\overline{\mathcal Y}$ satisfying the desired H\"olderianity bound $\op{dist}(f(t), f(t'))\leq C|t-t'|^{1/2}$.
 
By Fatou's lemma we obtain from \eqref{integraln} that 
\begin{equation}\label{fatouconseq}
 \int_0^1\liminf_{n'\to \infty} \mathcal N(f_{n'}(t))dt\leq C\ ,
\end{equation}
in particular there exists a negligible set $E\subset [0,1]$ such that for every $t\in[0,1]\setminus E$ the sequence $\mathcal N(f_{n'}(t))$ has a subsequence which we will label by $(n_k^t)_k$ such that
\begin{equation}\label{precisetdepseq}
\forall t\in [0,1]\setminus E,\quad \lim_{k\to\infty} \mathcal N(f_{n_k^t}(t)) =\liminf_{n'\to\infty}\mathcal N(f_{n'}(t))<\infty\ .
\end{equation}
 By the compactness hypothesis \eqref{hypoth} applied to the sequences $(f_{n_k^t}(t))_k$, up to replacing $(n_k^t)_k$ by a further subsequence, they converge for every $t\in E$ and the $\op{dist}$-limit of the $(f_{n_k^t}(t))_k$ belongs to $\mathcal Y$. As the pointwise limit $f$ obtained previously is unique, we find that $f$ has values in $\mathcal Y$ for all $t\in[0,1]\setminus E$, i.e. the initial sequence $n'$ itself converges for almost every $t\in[0,1]$ as desired. By combining \eqref{fatouconseq} and \eqref{precisetdepseq} with the uniqueness of the pointwise limit of the $f_{n'}(t)$ we also obtain the desired integral bound on $\mathcal N\circ f$.
\end{proof}

\subsection{Verifying the hypothesis of Proposition \ref{abstractthm}}\label{verifhyp}
We verify that we can apply Proposition \ref{abstractthm} to our situation, where the goal is to prove weak closure for the class $\mathcal A_{G}$. We start with an auxiliary result proved by techniques close to \cite{rivinterpol} Thm. IV.1.
\subsubsection{Coulomb gauges with Lorentz-improved regularity}
\begin{proposition}\label{coulombLorentz},
 Suppose that $A$ and $B=g^{-1}dg+g^{-1}Ag$ are connection forms corresponding to two gauge-related connections belonging to $\mathfrak A^{1,2}(E)$ where $E\to\Omega$ is a trivial bundle over a domain $\Omega\subset\mathbb R^4$ such that 
 \[
  d^*A=d^*B=0\ .
 \]
If $A,B\in W^{1,2}$ then the gauge change $g$ is $W^{2,2}\cap C^0$. Moreover for some $\bar g\in G$ we have the bound
\begin{equation}\label{Linfest}
 \|g-\bar g\|_{L^\infty\cap W^{2,2}}\lesssim \|A\|_{W^{1,2}}^2 +\|B\|_{W^{1,2}}^2\ .
 \end{equation}
\end{proposition}
\begin{proof}
From 
\[
 dg=gB-Ag\ ,
\]
since multiplication is continuous from $W^{1,2}\times(W^{1,2}\cap L^\infty)$ to $W^{1,2}\hookrightarrow L^{(4,2)}$ it follows that $dg\in W^{1,2}\hookrightarrow L^{(4,2)}$ and 
\[
 \|dg\|_{L^{(4,2)}}\lesssim \|A\|_{W^{1,2}} +\|B\|_{W^{1,2}}\ .
\]
From the above equation and using $d^*A=d^*B=0$ and identifying $1$-forms with vector fields we obtain
\[
 \Delta g= d^*dg=dg\cdot A - B\cdot dg\ ,
\]
where both terms are products of elements of $L^{(4,2)}$ therefore belong to $L^{(2,1)}$. We have
\[
 \|\Delta g\|_{L^{(2,1)}}\lesssim \|dg\|_{L^{(4,2)}}(\|A\|_{L^{(4,2)}} +\|B\|_{L^{(4,2)}})\lesssim \|A\|_{L^{(4,2)}}^2 +\|B\|_{L^{(4,2)}}^2\ .
\]
By the continuous embeddings $W^{2,(2,1)}\hookrightarrow W^{1,(4,1)}\hookrightarrow L^\infty$ valid in $4$ dimensions, we obtain 
\[
 \|g-\tilde g\|_{L^\infty\cap W^{2,2}}\lesssim \|A\|_{L^{(4,2)}}^2 +\|B\|_{L^{(4,2)}}^2:=(*)\ ,
\]
where $\tilde g$ is the average of $g$ done in the space $\mathbb R^N, N=k\times k$ in which the manifold $G$ is embedded as group of matrices. Since $g\in G$ a.e., we also have
\[
 \op{dist}_{\mathbb R^N}(\tilde g, G)\lesssim (*)\ ,
\]
therefore there exists $\bar g\in G$ such that 
\[
 \|g-\bar g\|_{L^\infty}\lesssim (*)\lesssim \|A\|_{W^{1,2}}^2 +\|B\|_{W^{1,2}}^2\ ,
\]
as desired. Note that $W^{1,2}$ connections in $4$-dimensions can be approximated by smooth connections in $W^{1,2}$-norm (see Lemma
\ref{4dapprox}
). By applying the above result on balls $B_\rho(x)$ with $\rho\to 0$ for a.e. $x$, we obtain that $g\in C^0$ too.
\end{proof}

\subsubsection{The compactness result \eqref{hypoth}}
We start by verifying the first statement of the hypothesis \eqref{hypoth} for $\mathcal Y,\mathcal N$ as in Section \ref{spaceY}:
\begin{proposition}\label{verifh1}
Let $\mathcal Y$ be the space of slices as in \eqref{Y} and $\mathcal N:\mathcal Y\to \mathbb R^+$ be the norm of the curvature as in \eqref{N}. Then $\mathcal N$ has sublevels which are compact with respect to the distance $\op{dist}$ defined in \eqref{dslice}.
\end{proposition}
\begin{proof}
We assume that we are given a sequence of curvatures $F_n$ corresponding to connection form classes $[A_n]$, such that 
\[
 \|F_n\|_{L^2(\mathbb S^4)}\leq C\ .
\]
The claim of the proposition is that the $[A_n]$ have a convergent subsequence with respect to the distance $\op{dist}$.\\
Up to a global gauge change we may assume that the $A_n$ are controlled globally in $L^2$ (see Lemma \ref{verifh2}):
\[
 \|A_n\|_{L^2(\mathbb S^4)}\lesssim \|F_n\|_{L^2(\mathbb S^4)}\ .
\]
Up to extracting a subsequence we have that
\[
 A_n\rightharpoonup A_\infty\quad, F_n\rightharpoonup F_\infty\quad\text{ in }L^2(\mathbb S^4)\ .
\]
\textbf{Step 1.} \textit{Concentration points of the curvature energy and a good atlas.} 
By usual covering arguments we have that up to extracting a subsequence there exist a finite number of concentration points of the curvature's $L^2$-energy $a_1,\ldots,a_N$ in $\mathbb S^4$. In other words there holds 
\[
 \forall\epsilon>0,\rho_\epsilon:=\quad\liminf_{n\to\infty}\inf\left\{\rho>0,x_0\in \mathbb S^4\setminus\cup B_\epsilon(a_i)\:\int_{B_\rho^{\mathbb S^4}(x_0)}|F_n|^2\geq \delta\right\}>0\ .
\]
The number $N$ of such points is $N\leq C/\delta$ where $C$ is the above $L^2$-bound on the curvatures.\\

Up to diminishing $\epsilon$ and $\rho:=\rho_\epsilon$ we may suppose $\epsilon+\rho_\epsilon<\rho_{inj}(\mathbb S^4)$ and that the balls $B_\epsilon(a_i)$ are disjoint. We can find a cover by the balls $B_\epsilon(a_i)$ and by finitely many balls $B_\rho(x_i)$ such that the maximum number of overlaps of those balls is a universal constant. The $B_\rho(x_i)$'s will be called \textit{good balls} and they will be simply denoted $B_i$ below.\\

\textbf{Step 2.} \textit{Uhlenbeck Coulomb gauges converge weakly on the good balls.} Using Uhlenbeck's gauge extraction of Theorem \ref{uhlereg} on each $B_i$ one finds a gauge $g_n^i$ such that $A_n^i:=(g_n^i)^{-1}dg_n^i + (g_n^i)^{-1}A_ng_n^i\in W^{1,2}$ and such that
\[
 d^*A_n^i=0,\quad \|A_n^i\|_{W^{1,2}}\lesssim\|F_n\|_{L^2} \text{ on }B_i\ .
\]
Therefore up to a diagonal subsequence we also may assume that 
\begin{equation}\label{conv-1}
 A_n^i\to A^i\text{ weakly in }W^{1,2}\text{ and strongly in }L^2\ .
\end{equation}
By interpolation since the $g_n^i$ are bounded in $L^\infty$ we see that
\[
 g_n^i\to g^i\text{ weakly in }W^{1,2}\text{ and strongly in }L^q,\forall q<\infty\ .
\]
This strong convergence in $L^q$ together with the weak convergence of $A_n$ and of the $dg_n^i$ in $L^2$ implies that
\[
 A_n=g_n^id(g_n^i)^{-1} +g_n^iA_n^i (g_n^i)^{-1}\rightharpoonup g^id(g^i)^{-1} + g^iA^i(g^i)^{-1}=A\text{ in }\mathcal D'
\]
and by uniqueness of weak limits the $A^i$ obtained above are the local expressions of the limit $A$ in the limit gauges $g^i$.\\

\textbf{Step 3.} \textit{Point removability and strong global gauge convergence on good part.} By Proposition \ref{coulombLorentz} the gauge changes $g_n^{ij}:=g_n^j(g_n^i)^{-1}$ needed to pass from $A_n^i$ to $A_n^j$ are controlled in $W^{2,2}\cap C_0$. Therefore up to taking a diagonal subsequence we have for all $i,j$
\[
 g_n^{ij}\to g^{ij}\text{ weakly in }W^{2,2}\text{, strongly in }W^{1,2}\text{ and locally uniformly in }C^0\ .
\]
In particular we can apply the gauge extension method as in the proof in \cite{rivnote} Thm. V.6 of \cite{Uhlchern} Thm. 2.1 for $g_n^{ij}$ and $g^{ij}$ on balls covering any open contractible subset $U^{good}$ in the complement of the bad balls $B_\epsilon(a_1),\ldots,B_\epsilon(a_N)$, obtaining gauge transformations $g_n^{good}, g^{good}$. We recall that in this process we multiply gauges by the constants $\overline{g_n^{ij}}$ then truncate the error terms $(\overline{g_n^{ij}})^{-1}g_n^{ij}$ away from $B_i\cap B_j$. We note that up to extracting subsequences we may assume (by compactness of $G$ and finiteness of the balls intersecting $U^{good}$) that the constants involved also converge:
\[
 \overline{g_n^{ij}}\to\overline{g^{ij}}\ .
\]
This implies together with \eqref{conv-1} that on $U^{good}$ 
\[
 g_n^{good}(A_n)\to g^{good}(A)\text{ in }L^2(U^{good})\ .
\]
\textbf{Step 4.} \textit{The bad part's contribution.} The last part of the proof consists of noticing that by diminishing $\epsilon$ and by letting $U^{good}$ increase to a set of full measure, we may find gauges $g_n^k=(g^{good})^{-1}g_n^{good}$ such that 
\[
( g_n^k)^{-1}dg_n^k +(g_n^k)^{-1}A_ng_n^k\to A\text{ in }L^2\text{ outside a set of measure }\frac{1}{k}\ .
\]
By extracting a diagonal subsequence we obtain $g_n$ such that
\[
 g_n^{-1}dg_n +g_n^{-1}A_ng_n\to A\text{ in }L^2(\mathbb S^4)\ .
\]
Therefore
\[
\op{dist}([A_n],[A])\to 0\ ,
\]
as desired.
\end{proof}
\subsubsection{The second hypothesis of Proposition \ref{abstractthm}}
We now assume given a sequence of weak curvatures $F_n$ corresponding to $[A_n]\in\mathcal A_{G}$ on $\mathbb B^5$ which are bounded in $L^2$ and converge weakly in $L^2$ to a $2$-form $F$. For a fixed center $x_0\in \mathbb B^5$ and for a radii $t\in[r,2r]$ with $r>0$, the slices of the connections $A_n$ via spheres $\partial B_t(x_0)$ are defined and taking values in $\mathcal Y$ for a.e. $t$ by the assumption that $[A_n]\in \mathcal A_{G}$. We then define (classes of) functions 
\[
f_n:[r,2r]\to\mathcal Y,\quad f_n(t):=\left[i^*_{\partial B_t(x_0)}A_n\right]\ .
\]
\textbf{Notation: }We denote $A(s)$ the slice along $\partial B_s(x_0)$ i.e. the pullback of $i_{\partial B_s(x_0)}^*A$ to $\mathbb S^4$ via the homothety $\mathbb S^4\to \partial B_s(x_0)$ when it exists.\\

We verify that the $f_n$ satisfy the hypothesis \eqref{distcontrol}:
\begin{lemma}\label{verifh2}
 Assume that $F$is the curvature form corresponding to $[A]\in\mathcal A_{G}$ and choose a representative $A$ which is $L^2$ on $B_{2r}(x_0)\setminus B_r(x_0)$. Then there exists a gauge change $g$ such that $A':=g^{-1}dg+g^{-1}Ag$ has no radial component and such that for a.e. $t>t'\in[r,2r]$
 \begin{equation}\label{verifh2e}
\int_{\mathbb S^4}|A'(t) - A'(t')|^2\lesssim \frac{1}{r^2}|t-t'|\int_{B_t(x_0)\setminus B_{t'}(x_0)}|F|^2\ ,
 \end{equation}
for a universal implicit constant.
\end{lemma}
\begin{proof}
We will assume $x_0=0$ for simplicity. Note that
\[
\int_{t'}^t\|A(t)\|_{L^2(\mathbb S^4)}^2dt=\int_{\mathbb S^4}\int_{t'}^t|\rho\: i^*_{\partial B_\rho}A|^2\rho^4d\rho d\omega\ .
\]
Use Corollary \ref{wsolode} to solve the following ODE in polar coordinates:
\begin{equation}\label{ode}
\left\{
\begin{array}{ll}
\partial_\rho g(\omega,\rho)=-A_\rho(\omega,\rho)g(\omega,\rho),&\text{ for }\rho\in[t',t]\ ,\\[3mm]
g(\omega,t')=id,&\text{ for all }\omega\in \mathbb S^4\ .
\end{array}
\right.
\end{equation}
It then follows that for $A'=g^{-1}dg+g^{-1}Ag$ there holds
\[
\sum_k\frac{x_k}{\rho}A'_k:=A_\rho'=0\ ,
\]
therefore at $(\omega,\rho)$ we write
\[
\sum_k x_kg^{-1}F_{k i}g=\sum_k x_k\partial_k A'_i - \sum_kx_k\partial_iA'_k +\sum_kx_k[A'_k,A'_i]=\partial_\rho(\rho A'_i)\ .
\]
In other words
\[
 \rho\partial_\rho\res (g^{-1}Fg)|_{\partial B_s(x_0)}=\partial_\rho(\rho\: i^*_{\partial B_\rho}A')\ .
\]
Integrating in $s$ we have for a.e. $t>t'$ and then in $\omega$ we obtain
\begin{eqnarray*}
\int_{\mathbb S^4}|t\:i^*_{\partial B_t}A' - t'\:i^*_{\partial B_{t'}}A'|^2&=&\int_{\mathbb S^4}\left|\int_{t'}^t \rho\partial_\rho\res (g^{-1}Fg)\:d\rho\right|^2\\[3mm]
&\lesssim&|t-t'|\int_{\mathbb S^4\times[t',t]}\rho^2|\partial_\rho\res F|^2\ .
\end{eqnarray*}
We used Jensen's inequality and the fact that the norm is $G$-invariant. Note that for $\omega\in \mathbb S^4$ there holds
\[
A'(s)(\omega)=s\:i^*_{\partial B_s}A'(s\omega)\ ,
\]
therefore from above it follows
\[
 \int_{\mathbb S^4}|A'(t) - A'(t')|^2\lesssim \frac{|t-t'|}{(t')^2}\int_{B_t\setminus B_{t'}}|F|^2\ .
\]
Since $t'>r$ the thesis follows.
\end{proof}
In the end the functions $f_n(t)$ which will satisfy \eqref{distcontrol} in our situation will be the slice functions of the connection forms $A_n(t)$ in the gauges given by Lemma \ref{verifh2}. Note that as a direct consequence of Lemma \ref{verifh2} we have also
\begin{equation}\label{eqverifh2}
\op{dist}([A_n(t)],[A_n(t')])\lesssim\frac{\|F_n\|_{L^2(B_{2r}\setminus B_r)}}{r}|t-t'|^{1/2}\leq\frac{\|F_n\|_{L^2}}{r}|t-t'|^{1/2}\ .
\end{equation}
\subsubsection{Proof of Corollary \ref{wsolode}}
\begin{proof}
By Theorem \ref{naturality} we have a sequence of connections $[A_k]\in\mathcal R^\infty(\mathbb B^5)$ such that for some $L^2$-representatives $A_k$ and for their distributional curvature forms $F_k$ there holds
\[
 A_k\to A\text{ in }L^2,\quad F_k\to F\text{ in }L^2\ .
\]
For each $k$ using the control \eqref{tracecontrol} and the above strong convergence, we select by mean value theorem a radius $\rho_k\in [0,1/k]$ such that
\begin{equation}\label{goodsliceode}
\int_{\rho = \rho_k}|A_k|^2\le k C \int_{\mathbb B\setminus \mathbb B_{1-1/k}}|A_k|^2\le C.
\end{equation}
We then solve (recalling that $\rho$ is the radial coordinate equal to zero on $\mathbb S^4$, i.e. $\rho= 1-|x|$ for $x\in\mathbb B^5$)
\[
 \left\{\begin{array}{ll}
       \partial_\rho g_k(\omega,\rho)=-(A_k)_\rho(\omega,\rho) g_k(\omega,\rho)&\quad\text{for }\omega\in\mathbb S^4, \rho\in[0,t]\ ,\\[3mm]
       g_k(\omega, \rho_k)=id&\quad\text{for }\omega\in\mathbb S^4\ ,
        \end{array}
\right.
\]
where the solution $g_k$ is now defined on all rays $\omega=const$ except for the (finitely many) ones which contain one of the singular points of $A_k$. We have the following, where the indices $i$ indicate the directions orthogonal to $\rho$:
\begin{eqnarray}
(A_k^{g_k})_\rho&=&0  \ ,\label{ode2}\\[3mm]
(A_k^{g_k})_i&=&(A_k)_i \ \text{ at }\rho=\rho_k\label{boundaryode}\\[3mm]
(F_k^{g_k})_{\rho i}&=&\partial_\rho (A_k^{g_k})_i - \partial_i (A_k^{g_k})_\rho + [(A_k^{g_k})_\rho, (A_k^{g_k})_i]\stackrel{\eqref{ode2}}{=}\partial_\rho (A_k^{g_k})_i\ ,\label{ode3}\\[3mm]
\partial_i g_k&=& g_k(A_k^{g_k})_i - (A_k)_ig_k\ .\label{ode4}
\end{eqnarray}
Integrating \eqref{ode3} in the radial direction we find that the nonzero components $(A^{g_k}_k)_i$ are $L^2$-integrable with bounds depending on $\|F_k^{g_k}\llcorner \rho\|_{L^2}=\|F_k\llcorner \rho\|_{L^2}\le \|F_k\|_{L^2}$ only:
\begin{eqnarray*}
(A_k^{g_k})_i(\omega,\rho)&=&(A_k^{g_k})_i(\omega,\rho_k) + \int_{\rho_k}^\rho \partial_\rho (A^{g_k}_k)_i(\omega,\rho')d\rho' \stackrel{\eqref{ode3}, \eqref{boundaryode}}{=} (A_k)_i(\omega,\rho_k) + \int_{\rho_k}^\rho (F_k^{g_k})_{\rho i}(\omega,\rho')d\rho'\ ,\\
\|A_k^{g_k}\|_{L^2(\mathbb B\setminus\mathbb B_t)}&\le& C\left(\|A_k\|_{L^2(\partial \mathbb B_{\rho_k})} + \|F_k^{g_k}\llcorner \rho\|_{L^2(\mathbb B\setminus \mathbb B_t)}\right)\stackrel{\eqref{goodsliceode}}{\le} C\left(\|A_k\|_{L^2(\mathbb B\setminus \mathbb B_t)} + \|F_k\|_{L^2(\mathbb B\setminus\mathbb B_t)}\right).
\end{eqnarray*}
We have then from \eqref{ode4} that 
\begin{equation}\label{estgk}
 \|\nabla g_k\|_{L^2(\mathbb B\setminus \mathbb B_t)}\lesssim \|A_k\|_{L^2(\mathbb B\setminus \mathbb B_t)}+\|F_{A_k}\|_{L^2(\mathbb B\setminus \mathbb B_t)}\ .
\end{equation}
Up to extracting a subsequence we may assume
\[
 g_k\rightharpoonup g\quad\text{weakly in }W^{1,2}
\]
thus in particular in the sense of traces we have the following convergence establishing our desired boundary datum
\[
\lim_{k\to\infty}g_k|_{\rho=\rho_k} = g|_{\rho=0}=id,
\]
and also we have $g_k\to g$ a.e. and strongly in all $L^p,p<\infty$, thus by interpolation between $L^{2^*}$ and $L^\infty$ (recall that $g_k\in L^\infty$ because $G$ is compact). In particular since $g_k^{-1}$ converges in $L^2$ and $dg_k$ converges weakly in $L^2$ we have 
\[
 g_k^{-1}dg_k\stackrel{\mathcal D'}{\rightharpoonup}g^{-1}dg
\]
and by the above strong convergence results of $A_k$ in $L^2$ and of $g_k$ in all $L^p,\ p<\infty$ we have
\[
 g_k^{-1}A_k g_k\to g^{-1}Ag\quad\text{strongly in }L^q,\ q<2\ .
\]
Therefore we achieve the distributional convergence
\[
 A_k^{g_k}:=g_k^{-1}dg_k + g_k^{-1}A_k g_k\stackrel{\mathcal D'}{\rightharpoonup}g^{-1}dg + g^{-1}Ag=:A^g\ .
\]
As all terms in the limit ODE converge in the sense of distributions, the ODE also remains true for the limit terms. If we insert the above expression of $A^g$ into the formula for the distributional curvature $F_{A^g}=dA^g + A^g\wedge A^g$ we obtain:
\begin{eqnarray*}
 F_{A^g}&=&d(g^{-1}dg +g^{-1}Ag) + (g^{-1}dg +g^{-1}Ag)\wedge(g^{-1}dg +g^{-1}Ag)\\
 &=&-g^{-1}dg\wedge g^{-1}dg -g^{-1}dg\wedge g^{-1}Ag + g^{-1}dA\ g -g^{-1}Ag\wedge g^{-1}dg\\
&&+g^{-1}dg\wedge g^{-1}dg + g^{-1}Ag\wedge g^{-1}dg +g^{-1}dg\wedge g^{-1}Ag+ g^{-1}A\wedge Ag\\
&=&g^{-1}(dA+A\wedge A)g=g^{-1}F_Ag\ .
\end{eqnarray*}
Note that the above formal calculations are actually rigorous again due to the facts that in the analogous calculation for the approximants we have uniform bounds on $dg_k\in L^2$, $A_k\in L^2$ and $g_k, g_k^{-1}\in L^\infty$.
\end{proof}

\subsection{Proof of the Closure Theorem \ref{wclos}}
We consider a sequence $F_n$ corresponding to $[A_n]\in\mathcal A_{G}(\mathbb B^5)$ as in Theorem \ref{wclos} and we construct representatives $A_n$ such that 
\[
 \int_{\mathbb B^5}|A_n|^2\leq C\int_{\mathbb B^5}|F_n|^2\ ,
\]
like in Lemma \ref{verifh2}. We thus have that up to extracting a subsequence there holds 
\begin{equation}\label{weakconvan}
 A_n\rightharpoonup A\quad\text{ in }L^2(\mathbb B^5)\ .
\end{equation}
As noted above it suffices that for all centers $x_0$ and a.e. radius $t>0$ the homothety pullback to $\mathbb S^4$ of the slice $i^*_{\partial B_t}A$ of the limit connection form $A$ is in $\mathcal A_G(\mathbb S^4)$ or equivalently corresponds to a class in $\mathcal Y$. Fix $x_0\in \mathbb B^5$ and a range of radii $[r,2r]$. It is sufficient to prove that 
\begin{equation}\label{slicea12}
 \text{a.e. }s\in[r,2r],\quad A(s)\in\mathcal A_G(\mathbb S^4)\ .
\end{equation}
We will assume for simplicity that $x_0=0$ and we apply Lemma \ref{verifh2} obtaining new gauges for the $A_n$ in which \eqref{eqverifh2} is valid. From now on we are going to work in these gauges only. For simplicity of notation we still denote the expressions of the $A_n$ in these gauges by $A_n$. Note that we still obtain the control
\[
\|A_n\|_{L^2(B_{2r}\setminus B_r)}\lesssim \|F_n\|_{L^2}
\]
if in the proof of Lemma \ref{verifh2} for $A=A_n$ we replace the ODE \eqref{ode} by
\[
\left\{
\begin{array}{ll}
\partial_\rho g(\omega,\rho)=-(A_n)_\rho(\omega,\rho)g(\omega,\rho),&\text{ for }\rho\in[s,t]\ ,\\[3mm]
g(\omega,s)=id,&\text{ for all }\omega\in \mathbb S^4\ .
\end{array}
\right.
\]
for $s$ such that $A_n(s)$ satisfies
\[
 \|A_n(s)\|_{L^2}\lesssim\frac{1}{r}\|F_n\|_{L^2}\ .
\]
Thus we may still suppose that \eqref{weakconvan} holds on $B_{2r}\setminus B_r$. We next prove that in this case we have a stronger convergence:
\begin{lemma}\label{weaksl}
 Assume that for a sequence of connection forms $A_n\in L^2(B_{2r}\setminus B_r,\wedge^1\mathbb R^5\otimes\mathfrak g)$ there holds
 \[
  \|A_n(t) - A_n(t')\|_{L^2(\mathbb S^4)}\leq C|t-t'|^{1/2}
 \]
and that 
\[
 A_n\rightharpoonup A\quad\text{ weakly in }L^2\text{ on }B_{2r}\setminus B_r\ .
\]
Then there exists a subsequence $n'$ such that 
\begin{equation}\label{slicewiseweak}
 \text{for a.e. }s\in[r,2r]\text{ there holds }A_{n'}(s)\rightharpoonup A(s)\quad\text{ weakly in }L^2(\mathbb S^4)\ .
\end{equation}
\end{lemma}
\begin{proof}
The weak convergence hypothesis means that
\[
 \int A_n\wedge \beta \to \int A\wedge \beta\text{ for all }\beta\in L^2(B_{2r}\setminus B_r,\wedge^3\mathbb R^5\otimes\mathfrak g)\ .
\]
Consider an arbitrary $3$-form $\omega$ which is $L^2$ on $\mathbb S^4$ and a test $1$-form $\varphi(t)$ on $[r,2r]$. By taking 
\[
 \beta:=h^*_t\omega\wedge\varphi(t)\quad\text{ where }h_t:\mathbb S^4\to\partial B_t\text{ is a homothety}
\]
we obtain
\[
 \int_r^{2r}\int_{\mathbb S^4}A_n(t)\wedge\omega\wedge \varphi(t)\to \int_r^{2r}\int_{\mathbb S^4}A(t)\wedge\omega(x)\wedge \varphi(t)\ .
\]
If we use the notation
\[
f_n^\omega(t)=\int_{\mathbb S^4}A_n(t)\wedge\omega\ ,
\]
then from the first hypothesis it follows that 
\begin{eqnarray*}
 \left|f_n^\omega(t) - f_n^\omega(t')\right|&\leq& \|A_n(t)-A_n(t')\|_{L^2}\|\omega\|_{L^2}\\
 &\leq&C|t-t'|^{1/2}\|\omega\|_{L^2}\ .
\end{eqnarray*}
By Arzel\`a-Ascoli theorem the $f_n^\omega$ have a subsequence which converges uniformly to a $1/2$-H\"older function with the same H\"older constant:
\[
\sup_{t\in[r,2r]}\left|f_n^\omega(t) - f^\omega(t)\right|\to 0\ .
\]
By applying this reasoning to a countable $L^2$-dense subset $D$ of $\omega$'s in $L^2(\mathbb S^4,\wedge^3T\mathbb S^4\otimes\mathfrak g)$ and by a diagonal procedure we obtain that 
\[
 \forall\omega\in D,\quad \sup_{t\in[r,2r]}|f_n^\omega (t) - f^\omega(t)| \to 0\ .
\]
Since the functionals $\omega\mapsto \int A_n(t)\wedge \omega$ are strongly continuous on $L^2$ forms for a.e. $t$, we obtain that the above convergence holds on all $\omega\in L^2$, completing the proof.
\end{proof}
We are now ready to conclude the proof of the weak closure result.

\begin{proof}[End of proof of Theorem \ref{wclos}:]
Consider the global weak limit connection form $A\in L^2(\mathbb B^5)$. As said above we prove that a.e. slice of it is in $\mathcal A_G(\mathbb S^4)$ by considering separately the groups of slices with center $x_0$ and radii in $[r,2r]$. We assumed $x_0=0$ for simplicity and we obtained that the $A_n$ have a weakly convergent subsequence on $B_{2r}\setminus B_r$, therefore we may apply Lemma \ref{weaksl}. We obtain up to extracting a subsequence the slice-wise a.e. weak convergence \eqref{slicewiseweak}:
\[
\text{for a.e. }s\in[r,2r]\text{ there holds }A_n(s)\rightharpoonup A(s)\quad\text{ weakly in }L^2(\mathbb S^4)\ .
\]
Note that in this case the slice-wise weak limit $A(s)$ is indeed the slice of the limit connection.\\

On the other hand we saw in Section \ref{verifhyp} that the hypotheses of Proposition \ref{abstractthm} are verified for our $A_n$ therefore we also have up to another subsequence extraction
\[
\text{for a.e. }s\in[r,2r]\text{ there holds }[A_n(s)]\to [A^d(s)]\quad\text{ in }(\mathcal Y,\op{dist})\ .
\]
We have now to compare the slice $A(s)$ of the weak limit with the $\op{dist}$-limit of slices $A^d(s)$. Since 
\[
 \op{dist}([A_n(s)],[A^d(s)])=\inf_{g\in W^{1,2}(\mathbb S^4,G)}\|g^{-1}dg +g^{-1}A_n(s)g - A^d(s)\|_{L^2}\ ,
\]
we obtain a sequence $g_n(s)\in W^{1,2}(\mathbb S^4,G)$ such that 
\begin{equation}\label{distconv}
 g_n(s)^{-1}dg_n(s) + g_n(s)^{-1}A_n(s)g_n(s) - A^d(s)\to 0\quad\text{ strongly in }L^2\ .
\end{equation}
It follows that 
\[
 \|dg_n(s)\|_{L^2}\lesssim\|A^d(s)\|_{L^2} + \|A_n(s)\|_{L^2}\ .
\]
From 
\[
  \|A_n(t) - A_n(t')\|_{L^2}\leq C|t-t'|^{1/2}
\]
and from the fact that for all $n$ there exists $s\in[r,2r]$ such that
\[
\|A_n(s)\|_{L^2}\lesssim\|F_n\|_{L^2}\leq C
\]
it follows that $A_n(s)$ is bounded in $L^2$. Thus $dg_n(s)$ is also bounded in $L^2$. Thus up to extracting a subsequence (dependent on $t$)
\[
 dg_n(t)\rightharpoonup dg_\infty(t)\quad\text{ weakly in }L^2\ .
\]
Since $g_n(s)$ is also bounded in $L^\infty$ we obtain by Rellich's theorem and by interpolation that up to extracting a subsequence $n(t)$
\[
 g_n(t)\to g_\infty(t)\quad\text{ in }L^q\:\forall q<\infty\ .
\]
The last two facts together with the convergence $A_n(t)\stackrel{L^2}{\rightharpoonup}A(t)$ suffice to prove that 
\begin{eqnarray*}
 g_n(t)^{-1}A_n(t)g_n(t)&\to& g_\infty(t)^{-1}A(t)g_\infty(t)\text{ in }\mathcal D'(\mathbb S^4)\ ,\\[3mm]
 g_n(t)^{-1}dg_n(t)&\to& g_\infty(t)^{-1}dg_\infty(t)\text{ in }\mathcal D'(\mathbb S^4)\ .
\end{eqnarray*}
This is valid for a.e. $t\in[r,2r]$. Therefore 
\[
 A^d(t)=g_\infty(t)^{-1}dg_\infty(t) + g_\infty(t)^{-1}A(t)g_\infty(t),\quad\text{ for a.e. }t\in[r,2r]\ .
\]
Since $A^d(t)\in\mathcal A_G(\mathbb S^4)$, this shows that for a.e. $t$ the slice $A(t)$ of the limit connection form $A$ belongs to $\mathcal A_G(\mathbb S^4)$, as desired.
\end{proof}

\section{Regularity results}\label{ch:reg}
This section is devoted to the proofs of Theorem \ref{eregm} and its important Corollary \ref{preg} and the regularity of minimizers, Theorem \ref{rymp}. The structure of the proofs is analogous to the celebrated theory of harmonic maps, cfr. \cite{Simon} and the references therein. We apply our new approximation and extended regularity results in order to complete all the steps for curvatures in $\mathcal A_{G}(\mathbb B^5)$. The analogous results hold on general Riemannian compact $5$-manifolds and the proofs can be extended by working in charts and including error terms corresponding to the fact that the metric is not euclidean.\\

We start by proving Proposition \ref{Bianchi}, according to which the Bianchi identity $d_AF_A=0$ is verified by curvature forms $F$ and connection forms $A$ corresponding to $[A]\in\mathcal A_{G}(\mathbb B^5)$.
\begin{proof}[Proof of Proposition \ref{Bianchi}:]
We use the result of Theorem \ref{naturality}, namely the existence of a sequence of connection forms $A_k$ which are $L^2$ and have curvatures $F_k$ also in $L^2$, such that $[A_k]\in\mathcal R^\infty(\mathbb B^5)$ and 
\[
 A_k\to A\ \text{ in }L^2,\quad F_k\to F\ \text{ in }L^2\ .
\]
In particular we have $dF_k\stackrel{W^{-1,2}}{\rightharpoonup}dF$ and $\int_{\mathbb B^5}\varphi\wedge[F_k,A_k]\to\int_{\mathbb B^5}\varphi\wedge[F,A]$ for all $C^\infty_c(\mathbb B^5)$ test $1$-forms $\phi$. This implies in particular that 
\[
 d_{A_k}F_k\rightharpoonup d_AF\quad\text{ in the sense of distributions}\ ,
\]
thus we reduce to prove \eqref{bianchi} for $[A]\in\mathcal R^\infty(\mathbb B^5)$. In this case we see directly from the classical results that $d_AF\equiv 0$ locally outside the defects $a_1,\ldots,a_k$ of the classical bundle from the definition of $\mathcal R^\infty$. Since we have that $d_AF$ is a tempered distribution, it must then be locally near $a_i$ of the form $\sum_{\alpha=0}^lc_\alpha\delta_{a_i}^{(\alpha)}$, where $\delta_x^{(\alpha)}$ is the $\alpha$-th distributional derivative of the Dirac mass at $x$. On the other hand, since $F\in L^2$ and $[A,F]\in L^1$ we obtain that $d_AF\in W^{-1,2}_{loc}$ near $a_i$. Since we can construct forms $\phi_n$ which are bounded in $W^{1,2}$ but have values of the first $l$ derivatives in $a_i$, larger than $n$ we see that if $c_\alpha\neq 0$ for some $\alpha$ then
\[
 C\geq\langle d_AF,\phi_n\rangle=\sum_{\alpha=1}^nc_\alpha \phi_n^{(\alpha)}\to\infty\ ,                                                                                                                                                                                                                                                                                                                                                                                                                                                                                                                                                                                                                                                                                                                                                                                                                                                                                                                                                                        
                                                                                                                                                                                                                                                                                                                                                                                                                                                                                                                                                                                                                             \]
which is a contradiction. Thus $d_AF=0$ and this concludes the proof.
\end{proof} 

\subsection{Partial regularity for stationary connections in $\mathcal A_{G}$}
In this section we show how to bootstrap the results of \cite{MeRi} to the space $\mathcal A_{G}(\mathbb B^5)$, in order to prove the partial regularity result of Corollary \ref{preg}.\\

The main step is to improve on the result of \cite{MeRi} by removing the smooth approximability requirement (cfr. Theorem I.3 of \cite{MeRi}). Once this proof is done, the strategy of \cite{MeRi} can proceed to the proof of Theorem \ref{eregm} and to the regularity result of Corollary \ref{preg} with no changes.\\

\begin{proof}[Proof of Theorem \ref{merinew}:]
 In \cite{MeRi} the existence of $\epsilon, C$ for which a gauge $g$ in which \eqref{mcoulomb}, \eqref{mnormal} and \eqref{estmor} hold was proved under the assumption that $A$ be strongly approximable in $W^{1,2}\cap L^4$ by connection forms of smooth connections. In particular we may apply the result of \cite{MeRi} to the connection forms $\hat A_k$ furnished by Theorem \ref{mapproxd}. We obtain gauge changes $g_k$ such that
\[
 A_k:=\left(\hat A_k\right)_{g_k}\quad\text{ satisfies \eqref{mcoulomb},\eqref{mnormal}, \eqref{estmor}}
\]
with $F$ replaced by $F_k$. Since $\hat A_k\stackrel{L^2}{\to}A, \|A_k\|_{L^2}\lesssim \|F_k\|_{L^2}\lesssim \|F\|_{L^2}$ we obtain 
\[
 \|dg_k\|_{L^2}\leq C(\|\hat A_k\|_{L^2} + \|A_k\|_{L^2})\leq C
\]
therefore up to subsequence we can assume that $g_k$ converge pointwise a.e., weakly in $W^{1,2}$ and (by interpolation with $L^\infty$) in $L^p$ for all $p<\infty$. Similarly we may assume that $A_k\to A_\infty$ in $L^q$ for all $q<2^*$. It follows from the defining equation $g_k^{-1}dg_k +g_k^{-1}\hat A_k g_k=A_k$ that 
\[
 g^{-1}_kdg_k\to g^{-1}_\infty dg_\infty\quad \text{ strongly in }L^2\ ,
\]
thus we have that 
\[
 A_{g_\infty}=A_\infty\ ,
\]
in particular $g_\infty$ is such that conditions \eqref{mcoulomb}, \eqref{mnormal} and \eqref{estmor} hold, since they are stable under strong $L^2$ limits.
\end{proof}

\subsection{The regularity of local minimizers of the Yang-Mills energy in dimension $5$}
In this section we prove Theorem \ref{rymp}, which is a new result since the existence of minimizers and thus the availability of energy comparison techniques was not available before the introduction of the class $\mathcal A_{G}$.

\subsubsection{Luckhaus type lemma for weak curvatures}
Our aim in this section is to prove the following proposition, using a Luckhaus-type lemma for interpolating weak connections with $L^2$-small curvatures while paying a small curvature cost.
\begin{proposition}\label{minws}
Assume that $F_k$ are curvature forms corresponding to local minimizers $[A_k]\in\mathcal A_{G}(\mathbb B^5)$ and that $F_k\rightharpoonup F$ weakly in $L^2$ and $\sup_k\|F_k\|_{L^2(\mathbb B^5)}\leq C$. Then $F_k\to F$ strongly sin $L^2$ on a smaller ball $\mathbb B^5_{\frac{1}{2}}$, and $F$ is a local minimizer as well.
\end{proposition}
The main tool for the proof above is the following lemma:
\begin{lemma}[Luckhaus-type lemma for $\mathcal A_{G}$]\label{luckhaus}
Assume that $F_0, F_1$ are curvature forms on $\mathbb B^5_{t+4\epsilon}$ corresponding to connection forms $A_0,A_1\in\mathcal A_G(\mathbb B_{t+3\epsilon})$ such that
\begin{equation}\label{smallnessluck}
 \|F_\alpha\|_{L^2(\mathbb B_{t+3\epsilon}\setminus\mathbb B_{t-2\epsilon})}<\epsilon_0,\quad \|A_t\|_{L^2(\mathbb B_{t+3\epsilon}\setminus\mathbb B_{t-2\epsilon})}<\epsilon_0\ .
\end{equation}
 Then there exists a connection form $\hat A$ corresponding to $[\hat A]\in\mathcal A_{G}(\mathbb B_{t+4\epsilon})$ such that 
\begin{equation}\label{restr}
\hat A=A_0\text{ on }\mathbb B_{t-2\epsilon},\quad\hat A=A_1\text{ on }\mathbb B_{t+4\epsilon}\setminus\mathbb B_{t+3\epsilon}
\end{equation}
and
\begin{equation}\label{luckest}
 \|F_{\hat A}\|_{L^2(\mathbb B_{t+3\epsilon}\setminus\mathbb B_{t-2\epsilon})}\leq C\|F_0\|_{L^2(\mathbb B_{t+3\epsilon}\setminus\mathbb B_{t-2\epsilon})}+\|F_1\|_{L^2(\mathbb B_{t+3\epsilon}\setminus\mathbb B_{t-2\epsilon})}\ .
\end{equation}
\end{lemma}
\begin{proof}
\textbf{Step 1.}\textit{ Good grid of balls. }Like in Proposition \ref{ggrid} construct a good grid of balls of scale $\epsilon$ which form a cover of $\mathbb B_{t+\epsilon}\setminus \mathbb B_t$ and have centers on $\partial\mathbb B_{t+\epsilon/2}$. Note that since $\alpha\in]1,2[$ such balls will stay in $\mathbb B_{t+3\epsilon}\setminus\mathbb B_{t-2\epsilon}$.\\

\textbf{Step 2. }\textit{$W^{1,2}$ representatives on the boundary of a ball.} From now on we will work on a fixed ball $B$ of the above-defined good grid. We want to perform a modification of the approximation procedure like in the proof of Theorem \ref{naturality}. This consists in first interpolating on the boundary $\partial B$ and then extending the interpolant to $B$. We note that by the definition of $\mathcal A_G$, for $\alpha=0,1$ we may find gauges $g_\alpha\in W^{1,2}(\partial B,G)$ such that $\tilde A_\alpha:=g_\alpha^{-1}dg_\alpha + g_\alpha^{-1}A_\alpha g_\alpha\in W^{1,2}$.\\

\textbf{Step 3. }\textit{Interpolating gauges and connections on $\partial B$.} By Fubini's theorem and a pigeonhole principle we may find numbers $a_0\in[0,1/4], a_1\in[3/4,1]$ and a universal constant $C$ such that 
\begin{equation}\label{sliceg}
 \|g_\alpha|_{W^{1,2}(\partial B\cap\partial\mathbb B_{t+a_\alpha\epsilon})}\leq C\|g_\alpha\|_{W^{1,2}(\partial B)}\quad\text{for }\alpha=0,1\ .
\end{equation}
We may then use \eqref{sliceg} and apply Luchkaus' \cite{Luck} procedure for the extension of $W^{1,2}$ maps into manifolds and find $\tilde g\in W^{1,2}(\partial B, G)$ such that
\begin{eqnarray*}
 \tilde g&=&g_0\quad\text{on }\partial B\cap \mathbb B_{t+a_0\epsilon}\ ,\\
 \tilde g&=&g_1\quad\text{on }\partial B\setminus\mathbb B_{t+a_1\epsilon}\ ,\\
\|\tilde g\|_{W^{1,2}(\partial B\cap(\mathbb B_{t+a_1\epsilon}\setminus \mathbb B_{t+a_0\epsilon})}&\leq&C(\|g_0\|_{W^{1,2}(\partial B\cap(\mathbb B_{t+a_1\epsilon}\setminus \mathbb B_{t+a_0\epsilon})}+ \|g1\|_{W^{1,2}(\partial B\cap(\mathbb B_{t+a_1\epsilon}\setminus \mathbb B_{t+a_0\epsilon})})\ .
\end{eqnarray*}
We then extend the curvature forms simply by interpolating along meridians, i.e. we fix an increasing smooth function $\eta:[0,t+4\epsilon]\to [0,1]$ such that $\eta\equiv 0$ on $[0,t-a_0\epsilon]$ and $\eta\equiv 1$ on $[t+a_1\epsilon,t+4\epsilon]$ and for polar coordinates $(\omega,\tau)=x$ centered at $0$ and $(\omega,\tau)\in\partial B$ we define 
\[
 \tilde A(\omega,\tau)=(1-\eta(\tau))i_{\partial B}^*A_0 + \eta(\tau)i_{\partial B}^*A_1\ .
\]
As a consequence we obtain 
\[
\|\tilde A\|_{W^{1,2}(\partial B)}\leq C(\|A_0\|_{W^{1,2}(\partial B)}+\|A_1\|_{W^{1,2}(\partial B)})\ .
\]
\textbf{Step 4. }\textit{Extension on good and bad balls. }We use the same notion of good and bad balls as in Lemma-Definition \ref{manygoods} with the exception that we require the inequalities to be contemporarily valid for both $A_0, A_1$. The estimates of the mentioned lemma remain true, up to changing the constants by a universal factor.  In the case of a good ball $B$ the extension of $\tilde A$ to the interior of $B$ and the construction of $\hat g$ starting from $\tilde g$ are done as in Proposition \ref{badballext}. The estimates on $\tilde g, \tilde A$ from Step 3 together with the proof of Proposition \ref{badballext} give, as a consequence of the rescaled versions of \eqref{AAb}, \eqref{AAAb}, the estimates
\[
 \|d\hat A + \hat A\wedge \hat A \|_{L^2(B)}^2\lesssim \epsilon\|F_0\|_{L^2(\partial B)}^2+\epsilon\|F_1\|_{L^2(\partial B)}^2
\]
and
\[
 \|\hat A \|_{L^2(B)}\lesssim  \sum_{\alpha=0,1}\left(\epsilon\|F_\alpha\|_{L^2(\partial B)}^2 + \epsilon\|A\|_{L^2(\partial B)}^2\right)\ .
\]
If $B$ is a bad ball we directly extend $\tilde A$ radially inside.\\

\textbf{Step 5. }\textit{Summing up the estimates. }The conclusion of our proof consists of repeating Steps 1-5 and 8 of the proof of Theorem \ref{naturality}, i.e. we just jump the part where we perform the smoothing on the $4$-skeleton of our good grid. The estimates from the previous step and the trivial estimates for the bad balls give then the desired result.
\end{proof}

\begin{proof}[Proof of Proposition \ref{minws}:]
 \tb{Step 1. }We divide the interval $[1/2,1-4\epsilon]$ in $N$ equal subintervals of length $5\epsilon$, for $1/N\le\e_0/C$. By pigeonhole principle there exists one of such intervals $I=[t-2\epsilon,t+3\epsilon]\subset[1/2,1]$ such that up to subsequence we may assume
\[
 \|F_k\|_{L^2(\{x: |x|\in I\})}\leq \e_0,\quad \|F\|_{L^2(\{x: |x|\in I\})}\leq \e_0\ .
\]
\tb{Step 2. }We may reduce to the setting of Lemma \ref{luckhaus} with $F_0=F_k, F_1=F$. Let $\hat F_k$ be the interpolant produced in the Lemma \ref{luckhaus}. We have the following estimate:
\[
 \|\hat F_k\|_{L^2(\mb B_{t+3\epsilon}\setminus \mb B_{t-2\epsilon})}\lesssim N^{-1}(\|F_k\|_{L^2(\mb B_{t+3\epsilon}\setminus \mb B_{t-2\epsilon})} + \|F\|_{L^2(\mb B_{t+3\epsilon}\setminus \mb B_{t-2\epsilon})})\ .
\]
It is easy to check that the curvature $\hat F_k$ is still in $\m F_{\mb Z}(\mb B^5)$.\\

\tb{Step 3. }We use the fact that $F_k$ is locally minimizing to write the following inequalities:
\ba
\|F_k\|_{L^2(\mb B_{t-2\epsilon})}^2&\le&\|F_k\|_{L^2(\mb B_{t+3\epsilon})}^2\\
&\le&\|\tilde F_k\|_{L^2(\mb B_{t+3\epsilon})}^2\\
&=& \|F\|_{L^2(\mb B_{t-2\epsilon})}^2+\|\hat F_k\|_{L^2(\mb B_{t+3\epsilon}\setminus \mb B_{t-2\epsilon})}^2\\
&=&\|F\|_{L^2(\mb B_{t-2\epsilon})}^2 + o_\e(1)\ . 
\ea
In particular we see that no energy is lost in the limit on $\mb B_{t-2\epsilon}$:
\[
 \|F_k\|_{L^2(\mb B_{t-2\epsilon})}\to\|F\|_{L^2(\mb B_{t-2\epsilon})}\ ,
\]
which proves the result.
\end{proof}

\subsubsection{Dimension reduction for the singular set}
This section is devoted to the proof of Theorem \ref{rymp}.
We use the following definition:
\begin{definition}
 We denote by $\op{reg}(F)$ the set of points $x$ such that over some neighborhood $U\ni x$ there exists a smooth classical $G$-bundle $P\to U$ such that $F$ is the curvature form of a smooth connection over $P$. The complement of $\op{reg}F$ is denoted $\op{sing}(F)$.
\end{definition}
\begin{proof}[Proof of Theorem \ref{rymp}:]
It can be proved (see \cite{Tian} or \cite{MeRi}) from the monotonicity formula (see \cite{Price}) that for minimizing curvatures $F$, $\mathcal H^1(\op{sing}(F))=0$. If $S:=\op{sing}F$ and $F$ is a minimizing curvature we consider now $s\ge 0$ for which $\mathcal H^s(S\cap\Omega')>0$. Then $\mathcal H^s$-a.e. $x_0$ there holds
\begin{equation}\label{denss}
\liminf_{\lambda\downarrow 0}\lambda^{-s}\mathcal H^s(S\cap B_{\lambda/2}(x_0))>0\ .
\end{equation}
From the monotonicity formula we have (see \cite{Tian}) that for any subsequence $\lambda_i\to 0$ such that the blown-up curvature forms $F_{\lambda_i}:=\tau_{\lambda_i,x_0}^*F$, the weak limit curvature form $F_0$ is radially homogeneous. Here $\tau_{\lambda,x}$ is the homothety of factor $\lambda$ and center $x$. By Proposition \ref{minws} the convergence is also strong and $F_0$ is a minimizer.\\

$S_i:=\op{sing}F_{\lambda_i}$ which are the blow-ups of $S$, satisfy $\mathcal H^s(S_i\cap B_{1/2})=\lambda_i^{-s}\mathcal H^s(S\cap B_{\lambda_i/2})$ thus from \er{denss} we obtain
\begin{equation}\label{tgdens}
\mathcal H^s(S_0\cap B_{1/2})>0\ .
\end{equation}
As in \cite{Tian} from the stationarity we deduce that $F_0$ is radial and radially homogeneous. In particular $S_0$ is also radially invariant, i.e. $\lambda S_0\subset S_0$ for $\lambda>0$. Assume $S_0\neq \{0\}$. In particular $S_0$ must then contain a line and in this case $\m H^1(S_0)>0$. However since $F_0$ is still a minimizer this contradicts Corollary \ref{preg}.\\

The fact that $S_0=\{0\}$ for blown-up curvatures implies also that for a minimizer $F$ the singular points do not accumulate. Indeed if $x_i\to x_0$ were accumulating singular points, then by carefully choosing the blowup sequence we would be able to obtain $F_0$ such that $S_0\supset\{0,u/4\}$ where $u$ is a unit vector.
\end{proof}

\section{Consequences of closure and approximability}\label{pfcorolldens}

We will prove here Theorem \ref{trace} which completes the proof of Theorem \ref{YMP5}. The proofs are along the lines of the reasoning \cite{Ptrace} done in the case of abelian curvatures. 

The distance $\op{dist}$ on gauge-equivalence classes of connections is used to compare the boundary datum with the slices of forms $F\in\mathcal A_{G}$. We abuse notation and denote by $f(x+\rho)$ the form (with variable $x\in \mathbb S^4$) corresponding to the restriction to $\partial B_{1-\rho}$ of the form $F$. This notation is inspired by the analogy to slicing via parallel hyperplanes, instead of spheres. 
We then define the class $\mathcal A_{G,\varphi}(\mathbb B^5)$ via the continuity requirement
\begin{equation}\label{cond}
\op{dist}(f(x+\rho'),\varphi(x))\to 0\text{, as }\rho'\to 0^+\ .
\end{equation}
It is clear that the definition \eqref{cond} satisfies the \textit{nontriviality} and \textit{compatibility} conditions, since $\op{dist}(\cdot,\cdot)$ is a distance and since for $\mathcal R^\infty$ having smooth boundary datum implies that in a neighborhood of $\partial \mathbb B^5$ the slices are smooth up to gauge and converge in the smooth topology to $\varphi$. The validity of the \emph{well-posedness} is a bit less trivial, therefore we prove it separately.
\begin{theorem}\label{bdrytrace}
 If $F_n\in\mathcal A_{G, \varphi}(\mathbb B^5)$ are converging weakly in $L^2$ to a form $F\in\mathcal A_{G}(\mathbb B^5)$ then also $F$ belongs to $\mathcal A_{G, \varphi}(\mathbb B^5)$.
\end{theorem}
\begin{proof}
 By weak semicontinuity of the $L^2$ norm we have that $F_n$ are bounded in this norm, $||F_n||_{L^2(B_1\setminus B_{1-h})}\leq C$.\\

Therefore by Lemma \ref{verifh2} the $f_n$ are $\op{dist}$-equi-H\"older, so a subsequence (which we do not relabel) of the $f_n$ converges to a slice function $f_\infty$ with values in $Y$ a.e.. For all $\rho'\in[0,\rho]$ the forms $f_n(\cdot+\rho')$ are a Cauchy sequence in $n$, for the distance $\op{dist}$. This is enough to imply that $f_\infty$ is equal to the slice of $F$. Even if $F$ is just defined up to zero measure sets, it still has a $\op{dist}$-continuous representative. By uniform convergence it is clear that $f$ still satisfies \eqref{cond}.
\end{proof}
The same proof also gives an apparently stronger result:
\begin{theorem}\label{boundarys2}
 If $F_n\in\mathcal A_{G, \varphi_n}(\mathbb B^5)$ are converging weakly in $L^2$ to a form $F\in\mathcal A_{G}(\mathbb B^5)$ then the forms $\varphi_n$ converge with respect to the distance $\op{dist}$ to a form $\varphi$ and also $F$ belongs to $\mathcal A_{G, \varphi}(\mathbb B^5)$.
\end{theorem}

\begin{rmk}
 The definition of the distance can be extended as in \cite{Ptrace} and allows to extend the definition of the boundary value to arbitrary domains.
\end{rmk}

\appendix
\section{Proof of Proposition \ref{realization}}\label{proofp14}
\begin{proof}[Proof of Proposition \ref{realization}:]
\textbf{Construction of a representative.} By definition of $\mathfrak A_G(M)$ there exists a cover $\{U_\alpha\}_{\alpha\in I}$ and gauge changes $g_\alpha\in W^{1,2}(U_\alpha,G)$ such that $A^{g_\alpha}\in W^{1,2}$. We may suppose without loss of generality that $\{U_\alpha\}_{\alpha\in I}$ is a good cover too. We then define for all $\alpha,\beta \in I$:
\[
 A_\alpha:=A^{g_\alpha}\text{ on }U_\alpha, \quad g_{\alpha\beta}(x):=
g_\alpha^{-1}(x)g_\beta(x)\text{ when }x\in U_\alpha\cap U_\beta.
\]
Then the $g_{\alpha\beta}$ are easily seen to verify the cocycle condition, and verify $A_\beta=g_{\alpha\beta}^{-1}dg_{\alpha\beta} + g_{\alpha\beta}^{-1}A_\alpha g_{\alpha\beta}$ by just expanding the definition. Thus all that remains to be proved is that they belong to $W^{2,2}$. To prove it, rewrite the last formula as
\begin{equation}\label{chg}
dg_{\alpha\beta} = A^{g_\alpha} g_{\alpha\beta} - A^{g_\beta}g_{\alpha\beta}.
\end{equation}
Then because the $A^{g_\alpha},A^{g_\beta}\in L^2$ and $g_{\alpha\beta}\in L^\infty$ we have that $g_{\alpha\beta}\in W^{1,2}\cap L^\infty$, and then using this and the fact that $A^{g_\alpha},A^{g_\beta}\in W^{1,2}$ we also obtain that $dg_{\alpha\beta}\in W^{1,2}$ as desired, by Gagliardo-Nirenberg arguments like in the appendix of \cite{isobe1}.\\
\textbf{Independence on the choice of local gauges.} Choose now a possibly different set of local gauges $h_\alpha$ such that still $A^{h_\alpha}$ is $W^{1,2}$ as above. They lead to choices $h_{\alpha\beta}=h_\alpha h_\beta^{-1}$, which then satisfy by definition $h_{\alpha\beta}=k_\alpha^{-1}g_{\alpha\beta}k_\beta$ for $k_\alpha:= g_\alpha^{-1}h_\alpha$.
Now note that the functions $k_\alpha$ also satisfy 
\[
dk_\alpha = k_\alpha A^{h_\alpha} - A^{g_\alpha}k_\alpha, 
\]
and the same reasoning as for \eqref{chg} shows that $k_\alpha\in W^{2,2}$ using the fact that $A^{g_\alpha},A^{h_\alpha}\in W^{1,2}$. Thus the $\mathcal G^{2,2}(M)$-equivalence class of $P$ obtained by the above construction of representatives does indeed depend only on $A$ and not on the choice of local $W^{1,2}$-gauges $g_\alpha$ such that $A^{g_\alpha}\in W^{1,2}$.\\
\textbf{Surjectivity of $\mathfrak R$.} We consider now a fixed $P\in W^{2,2}(M)$ and $\tilde A\in\mathcal A^{1,2}(P)$ with the above notations and we construct $A\in\mathcal A_G(M)$ such that $\mathfrak R(A)=([P]_{2,2}, \tilde A)$. In fact we recall that in $4$ dimensions we have the Sobolev embedding $W^{2,2}\to W^{1,4}$, and we will only use the fact that $P$ is a $W^{1,4}$-presheaf. This allows to use the framework of \cite{isobe2} Section 3, to which we refer for details.\\
We consider a Lipschitz triangulation\footnote{here $K, K^{(k)}, S^{(k)}_i$ are combinatorial objects and $|\cdot|$ gives the corresponding supports in $M$. We here suppose that the supports (of all subcomplexes) are bilipschitz images of polyhedral sets.} $K$ of $M$ the supports of whose $k$-faces is contained into the intersection of $4-k$ of the $U_\alpha$ for $0\le k\le 4$. Up to perturbation we may suppose that the trivializations $g_{\alpha\beta}$ are all $W^{1,4}$ on $U_\alpha\cap U_\beta\cap |S^{(k)}_i|$ for each such $k$-dimensional face $S^{(k)}_i$ . Then for $0\le k\le 3$ note that in $k$-dimensions $W^{1,4}\to C^0$. Therefore if $K^{(k)}$ is the $k$-skeleton of $K$ and $i_{|K^{(k)}|}:|K^{(k)}|\to M$ is the inclusion map of its support, then $i_{|K^{(k)}|}^*P$ is a topological bundle.\\
As long as $\pi_k(G)=0$ (which is true for $k=0,1,2$ due to our hypotheses on $G$) we can iterate the following procedure starting with $k=0$ where the hypothesis of 1) is automatically verified: \\
1) Suppose that $i_{|K^{(k)}|}^*P$ has a global section $\sigma_k$ in $C^0\cap W^{1,4}$ and let $S^{(k+1)}\in K^{(k+1)}$ be a simplex and consider the restriction of $\sigma_k$ to $i_{|\partial S^{(k+1)}|}^*P$.\\
2) Then by the hypothesis $\pi_k(G)=0$ we may extend $\sigma_k$ to a $C^0\cap W^{1,4}$ section of $P$ over $|S^{(k+1)}|$. Doing this for all $(k+1)$-simplices allows to extend $\sigma_k$ to a continuous section of $i_{|K^{(k+1)}|}^*P$. \\
This procedure allows in particular to define a global $C^0\cap W^{1,4}$-section of $i_{|K^{(3)}|}^*P$. In particular this bundle is a trivial $W^{1,4}$-bundle. We use the fact that at this regularity by \cite{isobe2} we have $\check H^1(|K^{(k)}|, \mathcal C^0_G)\simeq\check H^1(|K^{(k)}|, \mathcal W^{1,4}_G)$ for $k\le 3$.\\
We conclude that there exists a global section $\sigma_3$ that is controlled in $W^{1,4}$. Since in particular we have $W^{1,3}$-bounds, we may apply Theorem B of \cite{PR2} or Theorem 2 of \cite{pvs} to define $W^{1,(4,\infty)}$-extensions over the $4$-cells which are elements of $K^{(4)}$. Thus we find a global section $\sigma_4$ of $P$ which is $W^{1,(4,\infty)}$ (and in particular is $W^{1,2}$). This gives local trivializations $g_\alpha$ of $P$ over each $U_\alpha$ that are $W^{1,2}$ and satisfy
\[
g_\alpha^{-1}(x)g_\beta(x)= g_{\alpha\beta}(x)\text{ for a. e. }x\in U_\alpha\cap U_\beta, \alpha,\beta\in I. 
\]
Then define 
\[
A|_{U_\alpha}:=g_\alpha d (g_\alpha^{-1}) + g_\alpha A_\alpha g_\alpha^{-1},
\]
and we obtain that $A\in L^2$ is well-defined and, as desired, belongs to $\mathfrak A_G(M)$ and satisfies $\mathfrak R(A)=([P]_{2,2},\tilde A)$.
\end{proof}
\begin{rmk}
As noted in the above proof, the realization of $\mathfrak A_G(M^4)$ is even surjective onto $[P]_{1,4}$-representatives of $\mathcal P^{1,4}$ bundles $P$.
\end{rmk}

\section{Controlled gauges on the 4-sphere}\label{app1}
Recall that $\pi:L^2(\mathbb S^4, \mathfrak{g})\to\left(\op{Span}\left\{i^*_{\mathbb S^4}dx_k,\:k=1,\ldots,5\right\}\right)^\perp$ denotes the $L^2$ projection operator.\\

In this section we follow the overall structure of the argument from \cite{Uhl2} to prove the following result:
\begin{theorem}\label{coulstrange}
 There exist constants $\epsilon_0, C$ with the following properties. If $A\in W^{1,2}(\mathbb S^4, \mathfrak{g})$ is a (global) connection form over $\mathbb S^4$ such that the corresponding curvature form $F$ satisfies 
 \[
  \|F\|_{L^2(\mathbb S^4)}+  \|A\|_{L^2(\mathbb S^4)}\leq \epsilon_0
 \]
then there exists a gauge transformation $g\in W^{2,2}(\mathbb S^4, G)$ such that 
\[
 d^*_{\mathbb S^4}(g^{-1}dg)=d^*_{\mathbb S^4}(\pi(g^{-1}dg))
\]
and denoting $A^g=g^{-1}dg+g^{-1}Ag$ the new expression of the connection form after the gauge transformation $g$ there holds
\[
 d^*_{\mathbb S^4}\left(\pi\left(A^g\right)\right)=0\quad\text{ and }\quad\|A^g\|_{W^{1,2}(\mathbb S^4)}\leq C(\|F\|_{L^2(\mathbb S^4)}+\|A\|_{L^2(\mathbb S^4)})\ .
\]
For the above gauge $g$ and if $\bar A\in \wedge^1\mathbb R^5\otimes\mathfrak g$ is a constant $1$-form with $|\bar A|\le \epsilon_0$then we further have the bounds 
\begin{equation}
\label{A-2}
\|dg\|_{L^2(\mathbb S^4)}\leq C\|A-i^*_{\mathbb S^4}\bar A\|_{L^2(\mathbb S^4)}.
\end{equation}
\end{theorem}
The proof consists in studying the case where the integrability exponent $2$ is replaced by $p>2$ first, and then obtaining the $p=2$ cases as a limit. Note that for $p>2$ the space $W^{2,p}(\mathbb S^4, G)$ embeds continuously in $C^0(\mathbb S^4, G)$, thus gauges $g$ of small $W^{2,p}$-norm will be expressible as $g=\op{exp}(v)$ for some $v\in W^{2,p}(\mathbb S^4, \mathfrak{g})$, due to the local invertibility of the exponential map $\op{exp}:G\to\mathfrak{g}$.\\

We then consider the space 
\[
 E_p:=\left\{v\in W^{2,p}(\mathbb S^4, \mathfrak{g}):\:\int_{\mathbb S^4}vx_k=0,\:k=1,\ldots,5\right\}
\]
where $x_k$ are the ambient coordinate functions relative to the canonical immersion $\mathbb S^4\to \mathbb R^5$. In case $p>2$ the Banach space $E_p$ is, by the above considerations, the local model of the Banach manifold
\[
 M_p:=\left\{ g\in W^{2,p}(\mathbb S^4, G):\:\int\langle g^{-1}dg, i^*_{\mathbb S^4}dx_k\rangle=0,\:k=1,\ldots,5\right\}\ .
\]
We then consider the sets
\[
 \mathcal U^\epsilon_p:=\left\{A\in W^{1,p}(\mathbb S^4,\wedge^1T\mathbb S^4\otimes \mathfrak{g}):\:\|F_A\|_{L^2(\mathbb S^4)}+\|A\|_{L^2(\mathbb S^4)}\leq \epsilon_0\right\}
\]
and their subsets
\[
 \mathcal V^{\epsilon,C_p}_p:=\left\{
\begin{array}{c}
A\in\mathcal U^\epsilon_p:\:\exists g\in M_2\text{ s.t. }d^*_{\mathbb S^4}(\pi(A^g))=0,\\[3mm]
\|\pi(A^g)\|_{W^{1,q}}\leq C_q(\|F\|_{L^q}+\|A\|_{L^q})\text{ for }q=2,p\\[3mm]
\text{and }\|F\|_{L^2}+\|A\|_{L^2}<\epsilon
\end{array}
\right\}\ .
\]
\subsection{Proof of Theorem \ref{coulstrange}}
Like in \cite{Uhl2} we prove theorem \ref{coulstrange} by showing that if $\epsilon_0>0$ is small enough then for $p\geq 2$ we may find $C_p$ such that 
\begin{equation}\label{ueqv}
 \mathcal V^{\epsilon_0,C_p}_p=\mathcal U^{\epsilon_0}_p\ .
\end{equation}
We are interested in \eqref{ueqv} just for $p=2$ but we use the cases $p>2$ in the proof: we successively prove the following statements.
\begin{enumerate}
 \item $\mathcal U^\epsilon_p$ is path-connected.
 \item For $p\geq 2$ the set $\mathcal V^{\epsilon, C_p}_p$ is closed in $W^{1,p}(\mathbb S^4, \wedge^1T\mathbb S^4\otimes\mathfrak{g})$.
 \item For $p>2$ there exists $C_p, \epsilon_0$ such that the set $\mathcal V^{\epsilon_0,C_p}_p$ is open relative to $\mathcal U^{\epsilon_0}_p$. In particular \eqref{ueqv} is true for $p>2$.
 \item There exists $K$ such that if $g\in M_p,\:\|A^g\|_{L^4}\leq K$ and 
 \[
  d^*_{\mathbb S^4}(\pi(A^g))=0,\quad \|F\|_{L^2}+\|A\|_{L^2}<\epsilon_0
 \]
 then 
 \[
\|A^g\|_{W^{1,2}}\leq C_2(\|F\|_{L^2}+\|A\|_{L^2})\ .
 \]
 \item The case $p=2$ of \eqref{ueqv} follows from the case $p>2$.
 \item From the above bounds, \eqref{A-2} follows.
\end{enumerate}
\subsubsection*{Proof of step 1}
Fix $p\geq 2, \epsilon, A\in\mathcal U^\epsilon_p$. We observe that $0\in \mathcal U^\epsilon_p$. Moreover the connection forms $A_t(x):=tA(tx)$ for $t\in[0,1]$ all belong to $\mathcal U^\epsilon_p$ as well, like in \cite{Uhl2}.
\subsubsection*{Proof of step 2}
Let $A_k\in \mathcal V^{\epsilon,C_p}_p$ be a sequence of connection forms converging in $W^{1,p}$ to $A$. Consider the gauges $g_k$ as in the definition of $\mathcal V^{\epsilon,C_p}_p$. We may assume that the  $A_k^{g_k}$ have a weak $W^{1,p}$-limit $\tilde A$. The bounds and equation in the definition of $\mathcal V^{\epsilon,C_p}_p$ are preserved under weak limit thus we finish if we prove that $\tilde A$ is gauge-equivalent to $A$ via a gauge $g\in M_p$. We note that from $dg_k=g_k A_k^{g_k} - A_kg_k$ and the fact that $G\subset\mathbb R^N$ is bounded it follows that $\|dg_k\|_{L^{p^*}}\lesssim \|A_k^{g_k}\|_{W^{1,p}}+\|A_k\|_{W^{1,p}}$, thus it has a weakly convergent subsequence, $g_k\stackrel{W^{1,p^*}}{\rightharpoonup}g$. Thus we may pass to the limit the gauge change equation and obtain indeed $\tilde A=A^g$ and also $g\in M_p$.

\subsubsection*{Proof of step 3}
Fix $p>2$ and let $A\in \mathcal V^{\epsilon,C_p}_p$. Consider the following data:
\begin{eqnarray*}
 g&\in&M_p\ ,\\
 \eta&\in&W^{1,p}(\mathbb S^4, \wedge^1T\mathbb S^4\otimes\mathfrak{g})\ .
\end{eqnarray*}
Consider the following function of such $g,\eta$, with values in $L^p\cap\{x_k,\:k=1,\ldots,5\}^{\perp_{L^2}}$:
\[
 N_A(g,\eta):=d^*_{\mathbb S^4}\left(\pi\left(g^{-1}dg + g^{-1}(A+\eta)g\right)\right)=d^*_{\mathbb S^4}\left(g^{-1}dg + \pi\left(g^{-1}(A+\eta)g\right)\right)\ .
\]
Note that $N_A(id,0)=0$ and $N_A$ is $C^1$. We want to apply the implicit function theorem in order to solve in $g$ the equation $N_A(g,\eta)=0$ for $\eta$ in a $W^{1,p}$-neighborhood of $id\in M_p$. The implicit function theorem will imply also that the dependence of $g$ on $\eta$ will be continuous. Note that up to order $1$ in $t$ there holds $\op{exp}(tv)^{\pm 1}\sim 1\pm tv$. Using this and the fact that $E_p$ is the tangent space to $M_p$ at $id$ we find the linearization of $N_A$ at $(id,0)$ in the first variable: 
\begin{eqnarray*}
 H_A(v)&:=&\partial_gN_A(id,0)[v]\\
&=&\left.\frac{\partial}{\partial t}\right|_{t=0}\left[d^*_{\mathbb S^4}\left(\pi\left((\op{exp} (tv))^{-1}d\op{exp}( tv) + \op{exp}(tv)^{-1}(A+\eta)\op{exp}(tv)\right)\right)\right]\\
 &=& d^*_{\mathbb S^4}\left(dv +\pi([A,v])\right)\\
 &=&d^*_{\mathbb S^4}dv + [\pi(A), dv]\ .
\end{eqnarray*}
In the last passage we utilized the fact that $\pi$ acts only on the coefficients of $A$ and thus $\pi[A,v]=[\pi A, v]$ and the fact that $d^*_{\mathbb S^4}[\pi(A), v]=[d^*_{\mathbb S^4}(\pi(A)), v]+[\pi(A), dv]$ where the first term vanishes by hypothesis. We see that $H_A:E_p\to L^p\cap\{x_k,\:k=1,\ldots,5\}^{\perp_{L^2}}$ is thus given by
\[
 H_A(v)=\Delta_{\mathbb S^4}v + [\pi(A), dv]\ .
\]
By elliptic theory and Sobolev and H\"older inequalities in dimension $4$ we have 
\begin{eqnarray*}
\|H_A(v)\|_{L^p}&\geq& \|\Delta_{\mathbb S^4}v\|_{L^p} - \|[\pi(A), dv]\|_{L^p}\\
&\geq &c_p\|v\|_{W^{2,p}}  - c'_p\|\pi(A)\|_{L^4}\|v\|_{W^{2,p}}\ .
\end{eqnarray*}
For $c'_p/c_p\|\pi(A)\|_{L^4}<\frac{1}{2}$ we find that $H_A$ is invertible and the thesis follows.
\subsubsection*{Proof of step 4}
We start by observing that since $d^*_{\mathbb S^4}(\pi(A^g))=0, \langle g^{-1}dg, i^*_{\mathbb S^4}dx_k\rangle_{L^2}=0$ there holds
\begin{eqnarray*}
 d^*_{\mathbb S^4}A^g&=&\sum_{k=1}^55x_k\frac{1}{|\mathbb S^4|}\int_{\mathbb S^4}\langle A^g,i^*_{\mathbb S^4}dx_k\rangle
\\
&=&\sum_{k=1}^55x_k\frac{1}{|\mathbb S^4|}\int_{\mathbb S^4}\langle g^{-1}Ag,i^*_{\mathbb S^4}dx_k\rangle\ ,
\end{eqnarray*}
thus by invariance of the norm and Jensen's inequality
\begin{eqnarray*}
 \|d^*_{\mathbb S^4}A^g\|_{L^2}&=&\left(\int_{\mathbb S^4}\left|\sum_{k=1}^55x_k\frac{1}{|\mathbb S^4|}\int_{\mathbb S^4}\langle g^{-1}Ag,i^*_{\mathbb S^4}dx_k\rangle\right|^2\right)^{\frac{1}{2}}\\
 &\leq&C\left(\int_{\mathbb S^4}|A|^2\right)^{\frac{1}{2}}=C\|A\|_{L^2}\ .
\end{eqnarray*}
By Hodge inequality
\begin{eqnarray*}
 \|\nabla A^g\|_{L^2}&\lesssim&\|dA^g\|_{L^2}+\|d^*_{\mathbb S^4}A^g\|_{L^2}\\
 &\lesssim&\|F\|_{L^2} +\|A^g\|_{L^4}^2 +\|A\|_{L^2}\ .
\end{eqnarray*}
If $\|A^g\|_{L^4}\leq K$ small enough then the second term above is estimated by $K\|\nabla A^g\|_{L^2}$ which can then be absorbed to the left side of the inequality, giving the desired estimate.
\subsubsection*{Proof of step 5}
We approximate $A\in \mathcal U^{\epsilon_0}_2$ by smooth $A_k$ in $W^{1,2}$ norm. In particular there holds $A_k\in W^{1,p}$ for all $p>2$. We may obtain that $A_k\in \mathcal U^{\epsilon_0}_p= \mathcal V^{\epsilon_0,C_p}_p, p>2$ and in particular we find $g_k\in M_p$ such that 
\[
\|A_k^{g_k}\|_{L^4}\lesssim\|A_k\|_{W^{1,2}}\lesssim \|F_k\|_{L^2} +\|A_k\|_{L^2}\lesssim\epsilon_0\ ,
\]
where the constants depend only on the exponents $p$ and $2$. By possibly diminishing $\epsilon_0$ we thus achieve $\|A_k^{g_k}\|_{L^4}\leq K$ for all $k$. By the closure result of Step 2 for $p=2$ we thus obtain that the same estimate holds for $A$ and for some gauge $g\in M_2$ and by Step 4 we conclude that $A\in \mathcal V^{\epsilon_0,K}_p$, as desired. $\square$
\subsubsection*{Proof of step 6} Let $g$ be the change of gauge $g$ given by the first part of Theorem \ref{coulstrange}, i.e. such that 
\begin{equation}
\label{A-1}
\left\{
\begin{array}{l}
d^\ast_{\mathbb S^4}\pi(A^g)=d^\ast_{\mathbb S^4}(g^{-1}dg+\pi(g^{-1} A g))=0\ ,\\[3mm]
\|A^g\|_{W^{1,2}(\mathbb S^4)}\le C(\|F\|_{L^2(\mathbb S^4)}+\|A\|_{L^2(\mathbb S^4)})\ .
\end{array}
\right.
\end{equation}
From the equation defining $A^g$, namely 
\[
 A^g=g^{-1}dg+g^{-1}Ag\ ,
\]
we obtain (in our notation we identify $1$-forms and vector fields using the metric)
\begin{eqnarray*}
 \Delta_{\mathbb S^4}g&=&d^*_{\mathbb S^4}(g\,A^g - A\,g)\\[3mm]
 &=&dg\cdot A^g + (g-id)\,d^*_{\mathbb S^4}A^g +d^*_{\mathbb S^4}A^g\\[3mm]
&& - d^*_{\mathbb S^4}[(A-i^*_{\mathbb S^4}\bar A)\,g] - d^*_{\mathbb S^4}[i^*_{\mathbb S^4}\bar A\,(g-id)] - d^*_{\mathbb S^4}i^*_{\mathbb S^4}\bar A\\[3mm]
  &=&dg\cdot A^g +(g-id)\,d^*_{\mathbb S^4}A^g - d^*_{\mathbb S^4}[(A-i^*_{\mathbb S^4}\bar A)\,g] - d^*_{\mathbb S^4}[i^*_{\mathbb S^4}\bar A\,(g-id)] + \\
  &&+d^*_{\mathbb S^4}\left(\sum_{k=1}^5i^*_{\mathbb S^4}dx_k\,\frac{1}{|\mathbb S^4|}\int_{\mathbb S^4}\langle i^*_{\mathbb S^4}(\bar A-A^g), i^*_{\mathbb S^4}dx_k\rangle\right)\\[3mm]
  &=&dg\cdot A^g +(g-id)\,d^*_{\mathbb S^4}A^g - d^*_{\mathbb S^4}[(A-i^*_{\mathbb S^4}\bar A)\,g] - d^*_{\mathbb S^4}[i^*_{\mathbb S^4}\bar A\,(g-id)] + \\
  &&+5\sum_{k=1}^5x_k\,\frac{1}{|\mathbb S^4|},\int_{\mathbb S^4}\langle i^*_{\mathbb S^4}(\bar A-g^{-1}Ag), i^*_{\mathbb S^4}dx_k\rangle\ .
\end{eqnarray*}
In the last row we used the fact that $\int_{\mathbb S^4}\langle i^*_{\mathbb S^4}(g^{-1}dg), i^*_{\mathbb S^4}dx_k\rangle=0$. Note that if $\bar g$ is the average of $g$ on $\mathbb S^4$ taken in $\mathbb R^5$, then using the mean value formula there exists $x\in\mathbb S^4$such that $|g(x) - \bar g|\leq C\|g-\bar g\|_{L^2}$ and up to changing $g$ to $gg_0$ where $g_0$ is a constant rotation, we may also assume $g(x)=id$. Now by elliptic estimates and using the embedding $W^{-1,2}\to L^{4/3}$ and the H\"older estimate $\|ab\|_{L^{4/3}}\leq\|a\|_{L^2}\|b\|_{L^4}$ we deduce:
\begin{eqnarray*}
 \|dg\|_{L^2(\mathbb S^4)}^2&\lesssim&\|dg\|_{L^2}^2\|A^g\|_{L^4}^2 + \|g-id\|_{L^4}^2\|A^g\|_{L^4}^2\\[3mm]
 &&+\,\|A-i^*_{\mathbb S^4}\bar A\|_{L^2}^2 +\|g-id\|_{L^4}^2\|i^*_{\mathbb S^4}\bar A\|_{L^2}^2 + \|i^*_{\mathbb S^4}\bar A - A\|_{L^2}^2\|g-id\|_{L^2}^2\ .
\end{eqnarray*}
Utilizing the Sobolev inequality $\|g-id\|_{L^4}\lesssim \|dg\|_{L^2}$ and the facts that
\begin{eqnarray*}
 \|A^g\|_{L^4}^2&\lesssim&\|F\|_{L^2}^2+\|A\|_{L^2}^2\lesssim\epsilon_0\ ,\\[3mm]
 \|i^*_{\mathbb S^4}\bar A\|_{L^p}^2&\lesssim&|\bar A|\lesssim\epsilon_0\ ,
\end{eqnarray*}
we absorb the terms not containing $A-\bar A$ from the right hand side to the left hand side. For $\epsilon_0>0$ small enough we thus obtain \eqref{A-2}.

\section{Control of slices of harmonic extensions}\label{ch:slice}
\begin{lemma}
\label{bddconst}
Let $\Omega\subset\mathbb R^5$ be a bounded domain such that there exists a diffeomorphism $\phi:\Omega\to\mathbb B^5$ satisfying the following bounds:
\begin{equation}
 \label{estdiffeo}
 \|D\phi\|_{L^\infty(\Omega)} + \|D^2\phi\|_{L^{5/3}}\le C_\phi\ ,\quad \|D(\phi^{-1})\|_{L^\infty(\Omega)} + \|D^2(\phi^{-1})\|_{L^{5/3}}\le C_\phi\ .
\end{equation}
Then there exists a constant $C_\Omega>0$ depending only on $C_\phi$ such that whenever $\tilde u\in W^{\frac{3}{2},2}(\Omega)$ and $u\in W^{1,2}(\partial\Omega)$ is the trace of $\tilde u$, there holds
\[
 \|\tilde u \|_{W^{\frac{3}{2},2}(\Omega)}\leq C_\Omega\|u\|_{W^{1,2}(\partial\Omega)}\ .
\]
\end{lemma}
\begin{proof}
 \textbf{Step 1.} \emph{The deformation bounds.} We start by proving that under the given assumption on $\phi$ we may bound the $W^{3/2,2}$-norm of $\tilde u$ in terms of that of $\tilde u\circ\phi^{-1}$. We may assume that $\tilde u$ is smooth and by change of variables we obtain:
 \begin{eqnarray*}
  \int_\Omega |D^2\tilde u|^p&\leq&\|D\phi\|_{L^\infty}^5\int_{\mathbb B^5}|(D^2\tilde u)|^p\circ\phi^{-1}\\
  &\leq&\|D\phi\|_{L^\infty}^5\left(\|D\phi\|_{L^\infty}^p\int_{\mathbb B^5}|D^2(\tilde u\circ\phi^{-1})|^p + \int_{\mathbb B^5}|D(\tilde u\circ \phi^{-1})|^p|D^2(\phi^{-1})|^p\right)\ .
 \end{eqnarray*}
Therefore we have for $p^*=\frac{5p}{5-p}$ the critical exponent for the Sobolev embedding $W^{2,p}\to W^{1,p^*}$ and $q=\frac{5p}{6p-5}$ its dual exponent,
\begin{eqnarray*}
 \|D^2\tilde u\|_{L^p(\Omega)}&\leq&C\|D\phi\|_{L^\infty}^{5/p}\left(\|D\phi\|_{L^\infty}\|D^2(\tilde u\circ\phi^{-1})\|_{L^p} + \|D(\tilde u\circ\phi^{-1})\|_{L^{p^*}}\|D^2(\phi^{-1})\|_{L^{\frac{5p}{6p-5}}}\right)\\
 &\leq& C\|D\phi\|_{L^\infty}^{5/p}\left(\|D\phi\|_{L^\infty}\|D^2(\tilde u\circ\phi^{-1})\|_{L^p} + \|D^2(\tilde u\circ\phi^{-1})\|_{L^p}\|D^2(\phi^{-1})\|_{L^{\frac{5p}{6p-5}}}\right)\ .
\end{eqnarray*}
We note that $W^{2,\frac{5}{3}}\to W^{\frac{3}{2},2}$ in $5$ dimensions and thus we choose $p=5/3$, which gives $q=5/3$ as well. Thus the estimate on $\phi,\phi^{-1}$ allows to conclude that
\[
 \|D^2\tilde u\|_{L^{5/3}(\Omega)}\leq C C_\phi^4\|D^2(\tilde u\circ\phi^{-1})\|_{L^{5/3}(\mathbb B^5)}\ .
\]
In a similar way we can obtain \(L^p\)-control of \(\tilde u\) in terms pf \(\tilde u\circ\phi^{-1}\) as well, then by interpolation we have the estimate
\[
 \|\tilde u\|_{W^{\frac{3}{2},2}(\Omega)}\leq C_1\|\tilde u\circ\phi^{-1}\|_{W^{\frac{3}{2},2}(\mathbb B^5)}\ ,
\]
and similarly
\[
 \|u\circ\phi^{-1}\|_{W^{1,2}(\partial\mathbb B^5)}\leq C_2\|u\|_{W^{1,2}(\partial\Omega)}\ ,
\]
with \(C_1, C_2\) depending only on \(C_\phi\). We thus reduced to proving the result for \(\Omega=\mathbb B^5\) only.\\
\textbf{Step 2.} \emph{Proof for $\Omega=\mathbb B^5$.} We may prove the result by contradiction. If no general constant $C_{\mathbb B^5}$ would exist then we would obtain $u_i\in W^{1,2}(\partial\mathbb B^5)$ of norm $1$ which have extensions $\tilde u_i\in W^{\frac{3}{2},2}(\mathbb B^5)$ converging strongly to zero, contradicting the continuity of the trace.
\end{proof}
\begin{lemma}
 \label{bdddeform} Consider a domain $\Omega\subset\mathbb R^5$ which is the largest part of $\mathbb B^5\setminus S$ where $S$ is a sphere of radius $r\ge 1/2$ and center $x$ with $|x|>1$. Then there exists a diffeomorphism $\phi$ such that \eqref{estdiffeo} holds with a uniform choice of $C_\phi$ independent of such $S$.
\end{lemma}
\begin{proof}
 Let $y$ be a coordinate direction orthogonal to $x$. We will define $\phi$ to be rotationally equivariant with respect to those orientation preserving rotations which keep the direction of $x$ fixed (with respect to them $\Omega$ is indeed invariant by definition). Thus to specify $\phi$ we may just define its restriction to the plane $x,y$ in such a way as to be symmetric with respect to the $x$-axis. Take this restriction to be the map given by Riemann's uniformization theorem.\\
So-defined $\phi$ is clearly $C^1$ with norm bounds uniform in the parameters defining $\Omega$ and $D^2\phi, D^2(\phi^{-1})$ are locally bounded except near the singular part $\Sigma=S\cap \mathbb S^4$ of $\partial\Omega$. Near $p\in\Sigma$ however, $\phi$ is locally well approximated in the $2$-plane orthogonal to $T_p\Sigma$ by the complex map $z\mapsto z^\alpha$ with $\alpha<1/2$. From this we easily obtain $|D^2\phi|(x)\lesssim \op{dist}(x,\Sigma)^{-1}$, in particular $\phi\in W^{2,2-\epsilon}$ for any $\epsilon>0$ with bounds which are uniform in $\Omega$, as desired.
\end{proof}

 \end{document}